\documentclass[a4paper,reqno]{amsart}

\usepackage[T1]{fontenc}
\usepackage[utf8x]{inputenc}
\usepackage[english]{babel}
\usepackage{yfonts}
\usepackage{dsfont}
\usepackage{amscd,amssymb,amsmath,amsthm,amsfonts}
\usepackage{mathtools}
\usepackage{tensor}
\usepackage{graphicx}
\usepackage{tikz}
\usetikzlibrary{calc,matrix,arrows,decorations.pathmorphing}
\usepackage{calc}
\usepackage[cal=boondox,scr=boondoxo]{mathalfa}
\usepackage{enumitem}
\usepackage{marginnote}
\usepackage{booktabs}
\usepackage{scalerel}[2016/12/29]
\usepackage{url}
\usepackage{contour}
\usepackage{ulem}
\usepackage{accents}
\usepackage{placeins}

\usepackage{hyperref}
\hypersetup{colorlinks=true,pageanchor=false,
linkcolor=blue,citecolor=red,urlcolor=red}
\usepackage[all]{hypcap}

\allowdisplaybreaks

\newtheorem{theorem}{Theorem}[section]

\newtheorem{proposition}[theorem]{Proposition}
\newtheorem{lemma}[theorem]{Lemma}

\theoremstyle{definition}
\newtheorem{definition}[theorem]{Definition}

\theoremstyle{remark}
\newtheorem{remark}[theorem]{Remark}

\newcommand{\N}{\mathbb{N}}
\newcommand{\Z}{\mathbb{Z}}

\newcommand{\R}{\mathbb{R}}
\newcommand{\C}{\mathbb{C}}

\newcommand{\rmb}{\mathrm{b}}

\newcommand{\rmt}{\mathrm{t}}

\newcommand{\calB}{\mathcal{B}}
\newcommand{\calC}{\mathcal{C}}

\newcommand{\calF}{\mathcal{F}}

\newcommand{\calI}{\mathcal{I}}

\newcommand{\calT}{\mathcal{T}}

\renewcommand{\epsilon}{\varepsilon}
\renewcommand{\theta}{\vartheta}
\renewcommand{\phi}{\varphi}
\renewcommand{\Gamma}{\varGamma}
\renewcommand{\Sigma}{\varSigma}

\newcommand{\ad}{\mathrm{ad}}

\newcommand{\id}{\mathrm{id}}

\newcommand{\tr}{\mathrm{tr}}
\newcommand{\ptr}{\mathrm{ptr}}
\newcommand{\pcl}{\mathrm{pc}}
\newcommand{\tcl}{\mathrm{tc}}
\newcommand{\lk}{\mathrm{lk}}
\newcommand{\slk}{\mathrm{fr}}

\newcommand{\rb}{\mathrm{rb}}
\newcommand{\lb}{\mathrm{lb}}

\newcommand{\ev}{\mathrm{ev}}
\newcommand{\coev}{\mathrm{coev}}
\newcommand{\lev}{\smash{\stackrel{\leftarrow}{\mathrm{ev}}}}
\newcommand{\lcoev}{\smash{\stackrel{\longleftarrow}{\mathrm{coev}}}}
\newcommand{\rev}{\smash{\stackrel{\rightarrow}{\mathrm{ev}}}}
\newcommand{\rcoev}{\smash{\stackrel{\longrightarrow}{\mathrm{coev}}}}

\newcommand{\leqs}{\leqslant}
\newcommand{\geqs}{\geqslant}

\newcommand{\mods}[1]{\operatorname{\mathnormal{#1}-mod}}

\newcommand{\fsl}{\mathfrak{sl}}
\newcommand{\SL}{\mathrm{SL}}

\newcommand{\SO}{\mathrm{SO}}

\newcommand{\Hom}{\mathrm{Hom}}

\newcommand{\End}{\mathrm{End}}

\newcommand{\Vect}{\mathrm{Vect}}

\newcommand{\WRT}{\tau}
\newcommand{\HKR}{\psi}

\newcommand{\SLF}{\mathrm{SLF}}
\newcommand{\QC}{\mathrm{QC}}

\newcommand{\Proj}{\mathrm{Proj}}

\newcommand{\TL}{\mathrm{TL}}

\newcommand{\Idf}{\Proj(\TL)}
\newcommand{\IdX}{\Proj(\mods{\bar{U}})}
\newcommand{\ssKirby}{\omega}
\newcommand{\dsKirby}{\tilde{\omega}}
\newcommand{\nsKirby}{\Omega}
\newcommand{\sstabp}{\delta_+}
\newcommand{\sstabm}{\delta_-}
\newcommand{\nstabp}{\Delta_+}
\newcommand{\nstabm}{\Delta_-}

\newcommand{\beadC}{\TL_{\bar{U}}}

\newcommand{\ribTL}{\calT_\TL}

\newcommand{\up}{\uparrow}
\newcommand{\down}{\downarrow}

\makeatletter
\newcommand{\subalign}[1]{
  \vcenter{
    \Let@ \restore@math@cr \default@tag
    \baselineskip\fontdimen10 \scriptfont\tw@
    \advance\baselineskip\fontdimen12 \scriptfont\tw@
    \lineskip\thr@@\fontdimen8 \scriptfont\thr@@
    \lineskiplimit\lineskip
    \ialign{\hfil$\m@th\scriptstyle##$&$\m@th\scriptstyle{}##$\crcr
      #1\crcr
    }
  }
}
\makeatother

\def\clap#1{\hbox to 0pt{\hss#1\hss}}
\def\mathllap{\mathpalette\mathllapinternal}

\def\mathllapinternal#1#2{%
\llap{$\mathsurround=0pt#1{#2}$}}

\newcommand{\pic}[2][0]{\raisebox{-0.5\height + 2.5pt + #1pt}{\includegraphics{#2.pdf}}}

\newcommand\arxiv[2]{\href{https://arXiv.org/abs/#1}{\texttt{arXiv:#1} #2}}
\newcommand\doi[2]{\href{https://doi.org/#1}{#2}}

\contourlength{1pt}

\DeclareRobustCommand{\myuline}[1]{
 \ifmmode \text{\uline{$\phantom{#1}$}\llap{\contour{white}{$#1$}}}
 \else \uline{\phantom{#1}}\llap{\contour{white}{#1}} \fi
}
\newcommand{\ubar}[1]{\underaccent{\bar}{#1}}

\begin{document}

\raggedbottom

\title[Non-Semisimple 3-Manifold Invariants From the Kauffman Bracket]{Non-Semisimple 3-Manifold Invariants Derived From the Kauffman Bracket}

\author[M. De Renzi]{Marco De Renzi} 
\address{Institute of Mathematics, University of Zurich, Winterthurerstrasse 190, CH-8057 Zurich, Switzerland} 
\email{marco.derenzi@math.uzh.ch}
\address{Department of Mathematics, Faculty of Science and Engineering, Waseda University, 3-4-1 \={O}kubo, Shinjuku-ku, Tokyo, 169-8555, Japan} 
\email{m.derenzi@kurenai.waseda.jp}

\author[J. Murakami]{Jun Murakami} 
\address{Department of Mathematics, Faculty of Science and Engineering, Waseda University, 3-4-1 \={O}kubo, Shinjuku-ku, Tokyo, 169-8555, Japan} 
\email{murakami@waseda.jp}

\begin{abstract}
 We recover the family of non-semisimple quantum invariants of closed oriented 3-manifolds associated with the small quantum group of $\fsl_2$ using purely combinatorial methods based on Tem\-per\-ley--Lieb algebras and Kauffman bracket polynomials. These invariants can be understood as a first-order extension of Witten--Reshetikhin--Turaev invariants, which can be reformulated following our approach in the case of rational homology spheres.
\end{abstract}

\maketitle
\setcounter{tocdepth}{3}

\section{Introduction}\label{S:intro}

The distinction between \textit{semisimple} and \textit{non-semisimple} constructions in quantum topology refers to the properties of the algebraic ingredients involved. One of the most celebrated families of quantum invariants, known as Witten--Reshetikhin--Turaev (or WRT) invariants, is of the first kind. Indeed, if $r \geqs 3$ is an integer called the \textit{level} of the theory, then the WRT invariant\footnote{In this paper, the acronym WRT will not refer to the larger family of quantum invariants constructed by Reshetikhin and Turaev in terms of the representation theory of modular Hopf algebras, but rather to the specific subfamily recovering the topological invariants first obtained by Witten using Chern--Simons gauge theory and the Feynman path integral \cite{W88}.} $\WRT_r$ can be constructed using a semisimple quotient of the category of representations of the small quantum group $\bar{U}_q \fsl_2$ at the $r$th root of unity $q = e^{\frac{2 \pi i}{r}}$ \cite{RT91}. This invariant extends to a Topological Quantum Field Theory (TQFT for short), which can also be obtained using several different approaches based on methods ranging from combinatorics and skein theory \cite{BHMV95} to geometric topology and conformal field theory \cite{AU11}. On the other hand, the family of quantum invariants $Z_r$ considered in this paper is of the second kind. It has already been defined using the non-semisimple representation theory of quantum groups (without quotient operation), as well as more general categorical methods. By contrast, the approach developed here relies uniquely on Temperley--Lieb algebras and Kauffman bracket polynomials. In particular, we provide the first reformulation of a non-semisimple quantum invariant of closed 3-manifolds that completely bypasses Hopf algebras and their representation theory. This is the first step towards a purely combinatorial construction of non-semisimple TQFTs which will naturally induce new families of representations of Kauffman bracket skein algebras of surfaces.

The invariant $Z_r$ is defined for odd levels $r \geqs 3$, it takes values in complex numbers, and it coincides with the renormalized Hennings invariant associated with $\bar{U}_q \fsl_2$ at $q = e^{\frac{2 \pi i}{r}}$, as defined in \cite{DGP17}. Since the category of finite-dimensional representations of $\bar{U}_q \fsl_2$ is modular (in the non-semisimple sense of \cite{L94}), $Z_r$ fits into the larger family of quantum invariants constructed in \cite{DGGPR19}, and both of these approaches produce TQFT extensions whose properties are in sharp contrast with those of their semisimple counterparts, see for instance \cite[Proposition~1.4]{DGGPR20}. It should also be noted that the family of invariants considered here is very closely related to the generalized Kashaev invariants of knots in 3-manifolds defined in \cite{M13}, which have been extended to logarithmic Hennings invariants of links in 3-manifolds \cite{BBG17}, although both constructions focus on a somewhat complementary case, namely when the level $r \geqs 4$ is even. All these constructions build on the structure and properties of quantum groups and ribbon categories, and thus have a distinct algebraic flavor.

The goal of this paper is to reproduce the renormalized Hennings invariant associated with $\bar{U}_q \fsl_2$ relying exclusively on the technical setup used by Lickorish for the construction of WRT invariants \cite{L93}. One of the basic ingredients for this approach is given by the family of \textit{Temperley--Lieb algebras}\footnote{The connection between the two approaches stems from the well-known equivalence between the Temperley--Lieb algebra $\TL(m)$ and the centralizer algebra for the $m$th tensor power of the fundamental representation of a closely related Hopf algebra, \textit{Lusztig's divided power quantum group} $U_q \fsl_2$, which contains the small quantum group $\bar{U}_q \fsl_2$ as a Hopf subalgebra. In light of this, we can say our purpose is to reformulate the renormalized Hennings invariant in diagrammatic terms, rather than algebraic ones.} $\TL(m)$ of parameter $\delta = - q - q^{-1}$, where $m$ is a natural number, and by specific idempotent elements $f_m \in \TL(m)$ defined for $0 \leqs m \leqs r-1$ called (\textit{simple}) \textit{Jones--Wenzl idempotents}. In particular, a leading role is played by a formal linear combination of simple Jones--Wenzl idempotents in the range $0 \leqs m \leqs r-2$ called (\textit{semisimple}) \textit{Kirby color}, and denoted $\ssKirby$. The name comes from the fact that the scalar associated with an $\ssKirby$-labeled framed link by the graphical calculus based on the \textit{Kauffman bracket polynomial} \cite{K87} with variable $A = q^{\frac{r+1}{2}}$ is invariant under Kirby II moves. Our main technical achievement is the introduction of a \textit{non-semisimple Kirby color} $\nsKirby$, which is given by Definition~\ref{D:ns_Kirby_color} in terms of \textit{non-semisimple Jones--Wenzl idempotents} $g_m \in \TL(m)$ for $r \leqs m \leqs 2r-2$, which are in turn given by equations~\eqref{E:g_r_def}-\eqref{E:g_m_def}. Although this generalization of simple Jones--Wenzl idempotents dates back to \cite{GW93}, the formulas reported here were found in \cite{BDM19}, and were inspired by similar ones, for even values of the level $r$, due to Ibanez \cite{I15} and Moore \cite{M18}. It should also be noted that, when the level is a prime number $p$, then non-semisimple Jones--Wenzl idempotents recover $p$-Jones--Wenzl idempotents in the corresponding range, as defined in \cite{BLS19}.

\subsection{Outline of the construction}

The topological notion underlying the graphical calculus developed in this paper is that of a \textit{bichrome tangle}. Roughly speaking, a bichrome tangle is the union of a \textit{blue} framed tangle, which is both oriented and labeled by idempotent morphisms of the \textit{Temperley--Lieb category} $\TL$ of parameter $\delta = - q - q^{-1}$, and a \textit{red} framed link, which carries neither orientations nor labels. When a bichrome tangle $T$ is embedded inside a 3-manifold $M$, one should think of its blue part as an element of the corresponding Kauffman bracket skein module, and of its red part as a surgery prescription. \textit{Bichrome links}, which are closed bichrome tangles, allow us to revisit, in Section~\ref{S:3-manifold_invariant}, a construction due to Blanchet \cite{B92}. Indeed, the $\SO(3)$ version of the WRT invariant $\WRT_r(M,T)$ can be defined for a closed oriented 3-manifold $M$ decorated with a bichrome link $T \subset M$. This is done by means of a topological invariant $F_\ssKirby$ of bichrome links, taking values in $\C$, which is constructed using the Kirby color $\ssKirby$ and the Kauffman bracket polynomial. If $M$ is a closed oriented 3-manifold, $T \subset M$ is a bichrome link, and $L \subset S^3$ is a red surgery presentation of $M$ with positive signature~$\sigma_+$ and negative signature~$\sigma_-$, then
\[
 \WRT_r(M,T) := \frac{F_{\ssKirby}(L \cup T)}{\sstabp^{\sigma_+}\sstabm^{\sigma_-}}
\]
is a topological invariant of the pair $(M,T)$, where 
\begin{align*}
 \sstabp &:= \frac{i^{-\frac{r-1}{2}} r^{\frac 12} q^{\frac{r-3}{2}}}{\{ 1 \}}, &
 \sstabm &:= -\frac{i^{\frac{r-1}{2}} r^{\frac 12} q^{\frac{r+3}{2}}}{\{ 1 \}}.
\end{align*}

Similarly, the non-semisimple invariant $Z_r(M,T)$ is defined for a closed oriented 3-man\-i\-fold $M$ decorated with a bichrome link $T \subset M$, but not an arbitrary one. Indeed, $T$ needs to satisfy a certain \textit{admissibility} condition which consists in requiring the presence, among the labels of its blue components, of an idempotent morphism of $\TL$ belonging to the ideal generated by $f_{r-1}$, which can be understood as the ideal of projective objects of the idempotent completion of $\TL$. For instance, $g_r,\ldots,g_{2r-2}$ all correspond to projective objects, and the same holds for their tensor product with any other idempotent morphism of $\TL$, but $f_0,\ldots,f_{r-2}$ do not. In particular, the red part of an admissible bichrome link is allowed to be empty, while the blue part is not. In Section~\ref{S:3-manifold_invariant}, we define a topological invariant $F'_\nsKirby$ of admissible bichrome links, with values in $\C$, using the non-semisimple Kirby color $\nsKirby$, the Kauffman bracket polynomial, and the theory of \textit{modified traces} \cite{GKP10}. We point out that the admissibility assumption is required precisely in order to use this last ingredient, without which non-semisimple quantum invariants essentially boil down to a reformulation of semisimple ones, as in \cite{CKS07}. Indeed, in the case of non-semisimple ribbon categories such as $\TL$, non-degenerate modified traces can only be consistently defined on proper tensor ideals, and general existence results usually focus on the special case of the ideal of projective objects, see Section~\ref{S:m-trace}. Our main result can then be stated as follows.

\begin{theorem}
 If $M$ is a closed oriented 3-manifold, $T \subset M$ is an admissible bichrome link, and $L \subset S^3$ is a red surgery presentation of $M$ with positive signature~$\sigma_+$ and negative signature~$\sigma_-$, then
 \[
  Z_r(M,T) := \frac{F'_\nsKirby(L \cup T)}{\nstabp^{\sigma_+}\nstabm^{\sigma_-}}
 \]
 is a topological invariant of the pair $(M,T)$, where 
 \begin{align*}
  \nstabp &:= i^{-\frac{r-1}{2}} r^{\frac 32} q^{\frac{r-3}{2}}, &
  \nstabm &:= i^{\frac{r-1}{2}} r^{\frac 32} q^{\frac{r+3}{2}}.
 \end{align*}
\end{theorem}

\subsection{Strategy of the proof}\label{S:strategy_proof}

Although the small quantum group $\bar{U} := \bar{U}_q \fsl_2$ and its category of finite-dimensional representations $\mods{\bar{U}}$ do not appear in the definition of $Z_r$, they play an important role in the proof of its topological invariance. Indeed, a well-known faithful braided monoidal linear functor $F_\TL : \TL \to \mods{\bar{U}}$ allows us to interpret morphisms of $\TL$ as intertwiners between tensor powers of the \textit{fundamental representation} $X$ of $\bar{U}$, as recalled in Section~\ref{S:beads}. Then, the idea is essentially to check that our definition of $Z_r$ computes exactly the renormalized Hennings invariant associated with $\bar{U}$.

In Section~\ref{S:small_quantum_group} we prepare the ground for this comparison by introducing our algebraic setup. In particular, we fix a \textit{left integral} of $\bar{U}$, which is a linear form $\lambda \in \bar{U}^*$ satisfying a crucial condition that can be understood as an algebraic version of the invariance under Kirby II moves, see Section~\ref{S:proof_3-manifold_invariant}. This provides the key ingredient for the definition of both the original Hennings invariant and its renormalized version\footnote{Both \cite{H96} and \cite{DGP17} actually use a \textit{right integral}, but the difference is simply a matter of conventions. Other related choices involve the use of top tangles, bottom tangles, or string links, the use of the adjoint, the coadjoint, or the regular representation, and so on. Changing one of these conventions requires changing accordingly all the others.}. The left integral $\lambda$ belongs to the space $\QC(\bar{U})$ of \textit{quantum characters} of $\bar{U}$, which admits a basis composed of \textit{quantum traces} and \textit{pseudo quantum traces} corresponding to simple and indecomposable projective $\bar{U}$-modules. In Section~\ref{SS:decomposition_integral} we adapt computations of Arike \cite{A08} to the odd level case, and obtain an explicit decomposition of $\lambda$ into this basis of $\QC(\bar{U})$. Next, we use the fact that every quantum character can be interpreted as a $\bar{U}$-module morphism with target the trivial representation $\C$ and source the adjoint representation $\ad$, which has been studied in detail by Ostrik \cite{Os95}. An important property of $\ad$ is that it admits a $\Z$-grading, and that, as explained in Section~\ref{SS:decomposition_adjoint}, every quantum character is completely determined by its restriction to the subspace of degree 0 vectors of $\ad$. Then, the rest of the paper is devoted to explain why and how the non-semisimple Kirby color $\nsKirby$ provides a diagrammatic implementation of the left integral $\lambda$, and for the proof it is sufficient to focus on the degree 0 part of $\ad$.

The next step consists in reviewing the algorithm for the computation of the Hennings invariant. We point out that there exist already several places in the literature where different methods have been explained in detail. The interested reader can check \cite{H96} for the original definition, \cite{KR94} for an improved construction that avoids the use of orientations, \cite{L94, Oh95, K96, V03, H05} for several reformulations, and \cite{DGP17} for the renormalized version involving modified traces. The common idea behind all these different approaches is essentially to get rid of the representation theory in the original construction of WRT invariants \cite{RT91}, to figure out explicitly the relevant combinatorics for elements of the quantum group, and to evaluate these using the left integral $\lambda$. We will briefly explain the procedure once again in Section~\ref{S:beads} for convenience, but it should be noted that, once we establish that the non-semisimple Kirby color $\nsKirby$ implements the left integral $\lambda$, the rest of the proof should be regarded as a well-known consequence of the Hennings--Kauffman--Radford (or HKR) theory. More precisely, in Section~\ref{SS:bead_category} we introduce the \textit{bead category} $\beadC$ by allowing elements of $\bar{U}$ to sit on strands of morphisms of the Temperley--Lieb category $\TL$. This allows us to rephrase the HKR algorithm, and in particular the one presented by Kerler and Virelizier, as a procedure which, starting from a top tangle, returns a morphism of the bead category. Then, in Section~\ref{SS:diagrammatic_integral} we prove our main technical result, which can be explained as follows: completing the HKR algorithm with the algebraic evaluation based on the left integral $\lambda$ yields the same result as completing it with the diagrammatic evaluation based on the non-semisimple Kirby color $\nsKirby$. The proof is obtained by explicit computation, and it is based on a series of formulas involving non-semisimple Jones--Wenzl idempotents which are established in Section~\ref{S:computations}.

\subsection{Relation with WRT}

As explained in Section~\ref{S:3-manifold_invariant}, an invariant of closed oriented 3-manifolds decorated with non-admissible (possibly empty) bichrome links can be obtained by setting
\[
 Y_r(M,T) := Z_r(M \# S^3,T \cup O),
\]
where $O \subset S^3$ is a blue unknot of framing 0 and label $f_{r-1}$. Then, using \cite[Theorem~1]{CKS07}, we can show that
\[
 Y_r(M,T) = h_1(M) \WRT_r(M,T),
\]
where $h_1(M) = \left| H_1(M) \right|$ if $\left| H_1(M) \right|$ is finite, and $h_1(M) = 0$ otherwise. Furthermore, if $T' \subset M'$ is an admissible bichrome link, then we have
\[
 Z_r(M \# M',T \cup T') = Y_r(M,T) Z_r(M',T).
\]
Therefore, we can think of the invariant $Z_r$ as a first-order extension of $\WRT_r$, at least for rational homology spheres, in the same spirit of \cite[Section~1.3]{CGP12}.

\subsection{Future perspectives}

An extended version $\bar{\TL}$ of the Temperley--Lieb category $\TL$ was introduced in \cite{BDM19} in order to recover a diagrammatic description of the full monoidal subcategory of $\mods{\bar{U}}$ generated by the fundamental representation of $\bar{U}$. Indeed, roughly speaking, $\TL$ misses a few morphisms, since it is equivalent to the category of \textit{tilting modules} of a different Hopf algebra, namely Lusztig's divided power quantum group $U_q \fsl_2$, of which the small quantum group $\bar{U}$ is a Hopf subalgebra. Now, although $\bar{\TL}$ can be avoided for the definition of $Z_r$, we expect it to play a major role in any purely diagrammatic proof of its topological invariance, as well as in the skein model for its TQFT extension based on the universal construction of \cite{BHMV95}.

Using $\bar{\TL}$, we can define appropriate bichrome versions of Kauffman bracket skein modules. State spaces of non-semisimple TQFTs are quotients of these bichrome skein modules, and they naturally carry actions of Kauffman bracket skein algebras. The prospect of obtaining new families of representations for these algebraic structures is especially interesting, since most geometric applications of WRT TQFTs exploit this technology. More generally, the development of alternative models of non-semisimple TQFTs is a crucial step for enhancing the flexibility of the theory, and for promoting its applications to the deep and mysterious questions concerning the geometric and dynamic content of quantum constructions in topology.

\subsection*{Acknowledgments}

This work was supported by KAKENHI Grant-in-Aid for JSPS Fellows 19F19765. We would like to thank Christian Blanchet, for several helpful discussions, as well as the referees, for many useful remarks which helped us improve the paper.

\section{Temperley--Lieb category and modified trace}\label{S:TL}

In this section we recall definitions for the main tools required by our construction: Temperley--Lieb algebras, Kauffman bracket skein relations, and modified traces. In order to do this, we fix once and for all an odd integer $3 \leqs r \in \Z$, and we consider the primitive $r$th root of unity $q = e^{\frac{2 \pi i}{r}}$. For every natural number $k \in \N$ we adopt the notation 
\[
 \{ k \} := q^k - q^{-k},
 \quad [k] := \frac{\{ k \}}{\{ 1 \}},
 \quad [k]! := \prod_{j=1}^k [j],
 \quad \{ k \}' := q^k + q^{-k}.
\]

\subsection{Temperley--Lieb category}\label{SS:TL_category}

For the definition of the Temperley--Lieb category, we will follow the approach of \cite[Section~3.3]{L93} and \cite[Section~3]{BHMV95}, which is based on the category of unoriented framed tangles \cite[Section~7]{T89}. Let us consider the cube $I^3 \subset \R^3$, where $I \subset \R$ denotes the interval $[0,1]$. An \textit{$(m,m')$-tangle} is the unoriented image of a proper embedding into $I^3$ of a disjoint union of finitely many copies of $I$ and $S^1$, whose boundary is composed of $m$ points on the bottom line $I \times \{ \frac 12 \} \times \{ 0 \} \subset I^3$ and $m'$ points on the top line $I \times \{ \frac 12 \} \times \{ 1 \} \subset I^3$. A \textit{framed tangle} is a tangle equipped with a \textit{framing}, that is, a transverse vector field along each of its components. We represent framed tangles in $I^3$ as planar diagrams projected orthogonally to $I \times \{ 0 \} \times I$, and we adopt the blackboard framing convention, which means the framing always points to the reader. The \textit{Temperley--Lieb category} $\TL$ is the linear category with set of objects $\N$, and with vector space of morphisms from $m \in \TL$ to $m' \in \TL$ denoted $\TL(m,m')$, and given by the quotient of the vector space generated by isotopy classes of framed $(m,m')$-tangles in $I^3$ modulo the subspace generated by vectors of the form
\begin{gather*}
 \pic{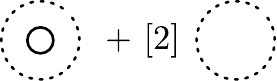} \tag{S1}\label{E:S1} \\*
 \pic{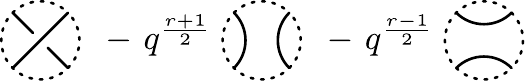} \tag{S2}\label{E:S2}
\end{gather*}
These pictures represent operations performed inside a disc $D^3$ embedded into $I^3$, and they leave tangles unchanged in the complement. Composition of morphisms of $\TL$ is given by gluing vertically two copies of $I^3$, in the opposite order with respect to \cite{T89}, and then shrinking the result into $I^3$. Then, for every $m \in \TL$, the $m$th \textit{Temperley--Lieb algebra} is defined as $\TL(m) := \TL(m,m)$.

The Temperley--Lieb category $\TL$ can be given a \textit{ribbon structure}, see \cite[Section~8.10]{EGNO15} for a definition. Tensor product of objects of $\TL$ is given by taking their sum, while tensor product of morphisms of $\TL$ is given by gluing horizontally two copies of $I^3$, and then shrinking the result into $I^3$. When representing graphically a morphism of $\TL$, we will sometimes allow components (or their endpoints) to carry labels given by natural numbers, as a shorthand for the number of parallel strands with respect to the framing (although sometimes, when this information can be deduced from the rest of the diagram, labels can be omitted), and we will allow morphisms to be replaced by boxes containing their name. Then left and right evaluation and coevaluation, braiding, and twist morphisms are defined, for all $m,m' \in \TL$, by
\begin{align}
 \lev_m = \rev_m &= \pic{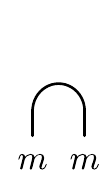} &
 \lcoev_m = \rcoev_m &= \pic{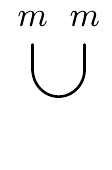} \label{E:TL_ev_coev} \\*[5pt]
 c_{m,m'} &= \pic{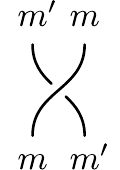} &
 \theta_m &= \pic{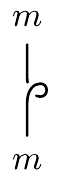} \label{E:TL_braiding_twist}
\end{align}
In particular, the dual $u^* \in \TL(m',m)$ of a morphism $u \in \TL(m,m')$ is obtained by a rotation of angle $\pi$.
By abuse of notation, we still denote by $\TL$ the idempotent completion of $\TL$. This means we promote idempotent endomorphisms $p \in \TL(m)$ to objects of $\TL$, and for all $p \in \TL(m)$ and $p' \in \TL(m')$ we set
\begin{equation}\label{E:morphism_completion}
 \TL(p,p') := \{ u \in \TL(m,m') \mid u p = u = p' u \}.
\end{equation}
By the same abuse of notation, we sometimes consider natural numbers $m \in \TL$ as idempotent endomorphisms of $\TL$.

Next, let us recall the definition of a special family of idempotents of $\TL$, first defined in \cite{J83,W87}, and later generalized in \cite{GW93}. We will follow the approach of \cite{BDM19}, where more details can be found. For every integer $0 \leqs m \leqs r-1$ the \textit{$m$th simple Jones--Wenzl idempotent} $f_m \in \TL(m)$ is recursively defined as 
\begin{align}
 \pic{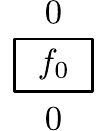} &:= \pic{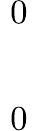} \label{E:f_0_def} \\*[5pt]
 \pic{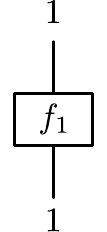} &:= \pic{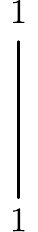} \label{E:f_1_def} \\*
 \pic{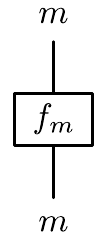} &:= \pic{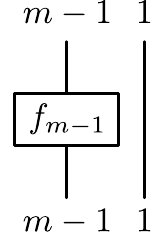} + \frac{[m-1]}{[m]} \cdot \pic{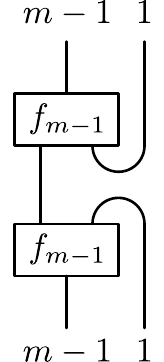} \label{E:f_m_def}
\end{align}
Let us recall some basic properties of simple Jones--Wenzl idempotents. We have
\begin{align}\label{E:f_f}
 \pic{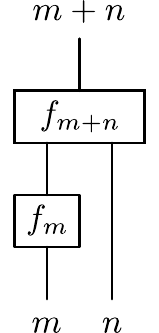} &= \pic{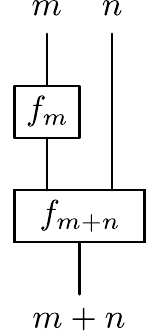} = \pic{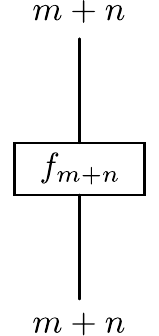}
\end{align}
for all integers $0 \leqs m \leqs r-1$ and $0 \leqs n \leqs r-m-1$, and
\begin{align}\label{E:t_f}
 \pic{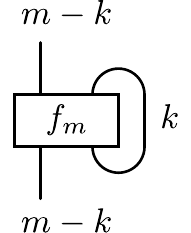} &= (-1)^k \frac{[m+1]}{[m-k+1]} \cdot \pic{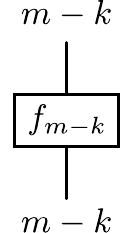}
\end{align}
for every integer $0 \leqs k \leqs m$. Simple Jones--Wenzl idempotents satisfy $f_m^* = f_m$ for every integer $0 \leqs m \leqs r-1$ with respect to the rigid structure determined by equation~\eqref{E:TL_ev_coev}. In other words, a rotation of angle $\pi$ fixes $f_m$. 

Similarly, thanks to \cite[Lemma~3.2]{BDM19}, for every integer $r \leqs m \leqs 2r-2$ the \textit{$m$th non-semisimple Jones--Wenzl idempotent} $g_m \in \TL(m)$ is recursively defined as
\begin{align}
 \pic{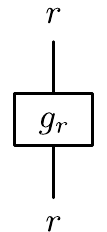} &:= \pic{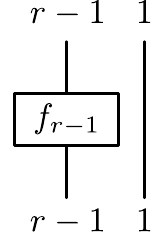} \label{E:g_r_def} \\*
 \pic{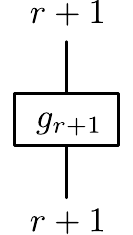} &:= \pic{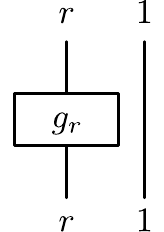} - \pic{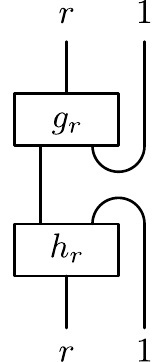} - \pic{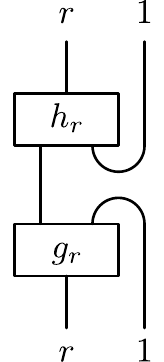} - [2] \cdot \pic{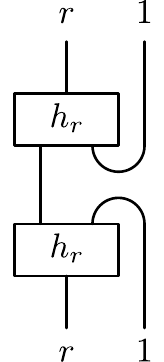} \label{E:g_r+1_def} \\*
 \pic{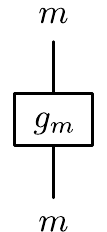} &:= \pic{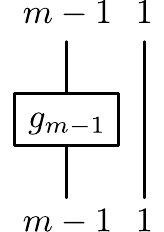} + \frac{[m-1]}{[m]} \cdot \pic{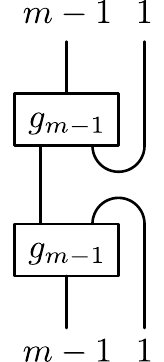} - \frac{2}{[m]^2} \cdot \pic{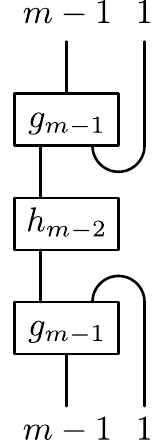} \label{E:g_m_def}
\end{align}
where $h_m \in \TL(m)$ is the nilpotent endomorphism defined as
\begin{equation}
 \pic{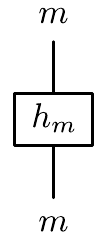} := (-1)^{m+1} [m+1] \cdot \pic{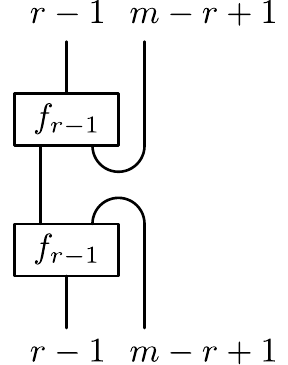} \label{E:h_def}
\end{equation}
As explained in \cite[Section~3]{BDM19}, these endomorphisms satisfy
\begin{align}\label{E:g_f}
 \pic{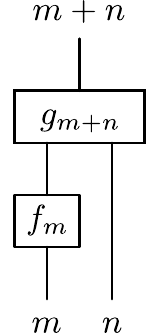} &= \pic{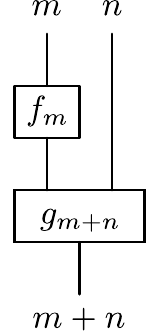} = \pic{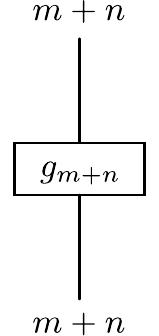}
\end{align}
for all integers $0 \leqs m \leqs r-1$ and $r-m \leqs n \leqs 2r-m-2$,
\begin{align}\label{E:g_h}
 \pic{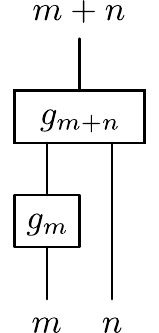} &= \pic{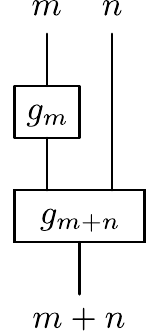} = \pic{jw_g_m+n} \\*
 \pic{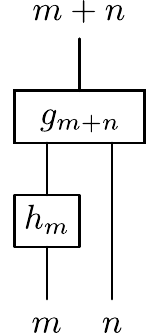} &= \pic{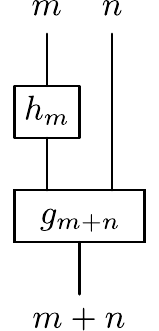} = \pic{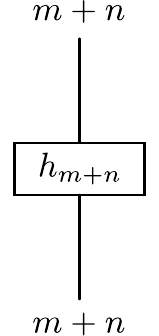} \\*
 \pic{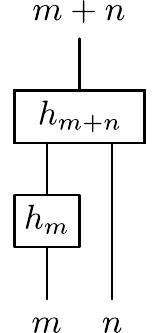} &= \pic{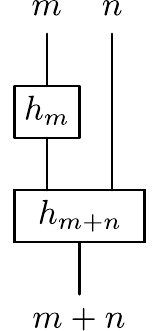} = 0
\end{align}
for all integers $r \leqs m \leqs 2r-2$ and $0 \leqs n \leqs 2r-m-2$,
\begin{align}
 \pic{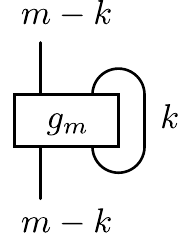} &= (-1)^k \frac{[m+1]}{[m-k+1]} \cdot \pic{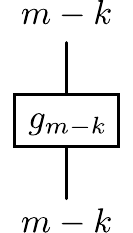} + (-1)^k \frac{2[k]}{[m-k+1]^2} \cdot \pic{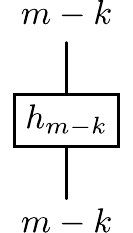} \label{E:t_g_k_low} \\*
 \pic{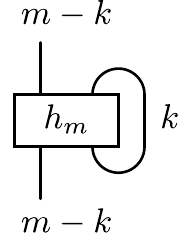} &= (-1)^k \frac{[m+1]}{[m-k+1]} \cdot \pic{jw_h_m-k} \label{E:t_h_k_low}
\end{align}
for every integer $0 \leqs k \leqs m-r$,
\begin{align}
 \pic{jw_t_g_m_k} &= (-1)^m \{ m+1 \}' \cdot \pic{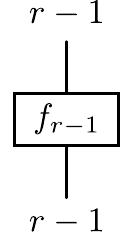} \label{E:t_g_k_middle} \\*
 \pic{jw_t_h_m_k} &= (-1)^{m+1} [m+1] \cdot \pic{jw_f_r-1} \label{E:t_h_k_middle}
\end{align}
for $k = m-r+1$, and
\begin{align}
 \pic{jw_t_g_m_k} &= \pic{jw_t_h_m_k} = 0  \label{E:t_g_h_k_high}
\end{align}
for every integer $m-r+2 \leqs k \leqs m$, see \cite[Lemma~3.1]{BDM19}. While for every integer $r \leqs m \leqs 2r-2$ we have $h_m^* = h_m$, non-semisimple Jones--Wenzl idempotents do not meet, in general, this condition, with the only exception of $g_{2r-2}^* = g_{2r-2}$, as explained in \cite[Remark~3.4]{BDM19}. For an explicit computation of non-semisimple Jones--Wenzl idempotents in the special case $r=3$, see Section~\ref{SS:level_3}. For a representation theoretic interpretation of the recursive relations~\eqref{E:f_0_def}-\eqref{E:f_m_def} and \eqref{E:g_r_def}-\eqref{E:g_m_def}, compare with equations~\eqref{E:simple}-\eqref{E:indecomp_proj} using \cite[Lemma~4.1]{BDM19}.

\subsection{Modified trace}\label{S:m-trace}

We recall now a few important definitions which will require the abstract language of \textit{ribbon categories}, see again \cite[Section~8.10]{EGNO15} for a general definition. What is most important for our purposes is the fact that a ribbon category $\calC$ comes equipped with a tensor product and a tensor unit, as well as structural data given by left and right evaluation and coevaluation, braiding, and twist morphisms, such as those introduced in equations~\eqref{E:TL_ev_coev} and \eqref{E:TL_braiding_twist} for $\TL$. Using this structure, a diagrammatic calculus for morphisms of $\calC$, based on the \textit{Penrose graphical notation}, can be developed. We point out that our convention for orientations will be opposite with respect to the one of \cite[Section~I.1.6]{T94}. Consequently, if $\calC$ is a ribbon category, then structure morphisms are represented, for all $x,x' \in \calC$, by
\begin{align*}
 \lev_x &= \pic{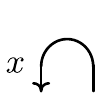} &
 \lcoev_x &= \pic{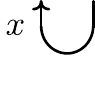} &
 \rev_x &= \pic{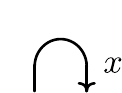} &
 \rcoev_x &= \pic{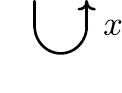} \\*[5pt]
 & &
 c_{x,x'} &= \pic{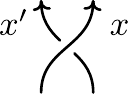} &
 \theta_x &= \pic{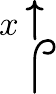} 
 & &
\end{align*}


An \textit{ideal} of a ribbon category $\calC$ is a full subcategory $\calI$ of $\calC$ which is absorbent under tensor products and closed under retracts. In other words, if $x \in \calI$, then for every $x' \in \calC$ we have $x \otimes x' \in \calI$, and for all $f \in \calC(x',x)$ and $f' \in \calC(x,x')$ satisfying $f' \circ f = \id_{x'}$ we have $x' \in \calI$. The \textit{ideal generated by an object} $x \in \calC$ is the ideal of $\calC$ whose objects $x'$ satisfy $\id_{x'} = f' \circ f$ for some object $x'' \in \calC$ and morphisms $f \in \calC(x',x \otimes x'')$ and $f' \in \calC(x \otimes x'',x')$. Remark that equations~\eqref{E:g_r_def}-\eqref{E:g_m_def} immediately imply that, for every integer $r \leqs m \leqs 2r-2$, the non-semisimple Jones--Wenzl idempotent $g_m$ belongs to the ideal of $\TL$ generated by $f_{r-1}$. 

The \textit{partial trace} of an endomorphism $f \in \End_\calC(x \otimes x')$ is the endomorphism $\ptr_{x'}(f) \in \End_\calC(x)$ defined as
\[
 \ptr_{x'}(f) := \pic{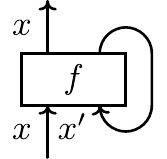}
\]

Following \cite{GKP10}, a \textit{trace $\rmt$ on an ideal $\calI$} of a ribbon linear category $\calC$ over $\C$ is a family of linear maps
\[
 \{ \rmt_x : \End_\calC(x) \to \C \mid x \in \calI \}
\]
satisfying:
\begin{enumerate}
 \item \textit{Cyclicity}: $\rmt_x(f' \circ f) = \rmt_{x'}(f \circ f')$ for all objects $x,x' \in \calI$ and all morphisms $f \in \calC(x,x')$, $f' \in \calC(x',x)$;
 \item \textit{Partial trace}: $\rmt_{x \otimes x'}(f) = \rmt_x(\ptr_{x'}(f))$ for all objects $x \in \calI$, $x' \in \calC$, and every endomorphism $f \in \End_\calC(x \otimes x')$.
\end{enumerate}
A trace $\rmt$ on $\calI$ is \textit{non-degenerate} if, for all $x \in \calI$ and $x' \in \calC$, the bilinear pairing $\rmt_x(\_ \circ \_) : \calC(x',x) \times \calC(x,x') \to \C$ determined by $(f',f) \mapsto \rmt_x(f' \circ f)$ is non-degenerate. Let us denote by $\Idf$ the ideal of \textit{projective objects of $\TL$}. 

\begin{proposition}\label{P:m-trace_def}
 The ideal $\Idf$ is generated by $f_{r-1} \in \TL$, and there exists a unique trace $\rmt^{\TL}$ on $\Idf$ satisfying
 \begin{equation}\label{E:m-trace_def}
  \rmt^{\TL}_{f_{r-1}}(f_{r-1}) = 1.
 \end{equation}
 Furthermore, $\rmt^{\TL}$ is non-degenerate.
\end{proposition}

A proof of Proposition~\ref{P:m-trace_def} is postponed to Section~\ref{S:proof_S:m-trace}. For the moment, let us simply point out that Proposition~\ref{P:m-trace_def} implies that every idempotent $p \in \TL(m)$ of $\Idf$ can be written as $p = u'u$ for some morphisms $u \in \TL(m,f_{r-1} \otimes m')$ and $u' \in \TL(f_{r-1} \otimes m',m)$. Furthermore, equations~\eqref{E:t_g_k_middle} and \eqref{E:t_h_k_middle} imply
\begin{align}\label{E:m-trace_TL}
 \rmt^{\TL}_{g_m}(g_m) &= (-1)^m \{ m+1 \}', & 
 \rmt^{\TL}_{g_m}(h_m) &= (-1)^{m+1} [m+1]
\end{align} 
for all $r \leqs m \leqs 2r-2$. See also \cite[Corollary~3.4]{HW19} for more existence results concerning modified traces, and \cite[Proposition~3.12]{TW19} for related formulas.

\section{3-Manifold invariant}\label{S:3-manifold}

In this section we introduce bichrome links, and we define a topological invariant $Z_r$ of closed 3-manifolds decorated with admissible ones. We also explain how to use $Z_r$ to obtain a topological invariant $Y_r$ of closed 3-manifolds without decorations, and show that this invariant recovers the $\SO(3)$ WRT invariant $\WRT_r$ for rational homology spheres. Every 3-manifold is assumed to be oriented.

\subsection{Bichrome tangles}\label{S:bichrome_tangles}

The notion of bichrome link will be based on a few preliminary definitions. First of all, the \textit{category $\ribTL$ of $\TL$-labeled oriented framed tangles} is the category whose objects are finite sequences $(\myuline{\epsilon},\myuline{p}) = ((\epsilon_1,p_1), \ldots, (\epsilon_j,p_j))$, where $\epsilon_i \in \{ +,- \}$ is a sign and $p_i \in \TL(m_i)$ is an idempotent for every integer $1 \leqs i \leqs j$, and whose morphisms from $(\myuline{\epsilon},\myuline{p}) = ((\epsilon_1,p_1), \ldots, (\epsilon_j,p_j))$ to $(\myuline{\epsilon'},\myuline{p'}) = ((\epsilon'_1,p'_1), \ldots, (\epsilon'_{j'},p'_{j'}))$ are isotopy classes of framed $(j,j')$-tangles whose components carry orientations and labels given by idempotents of $\TL$ matching those specified by $(\myuline{\epsilon},\myuline{p})$ and $(\myuline{\epsilon'},\myuline{p'})$, with $+$ and $-$ corresponding to upward and downward orientation respectively. Composition of morphisms of $\ribTL$ is given by gluing vertically two copies of $I^3$ and then shrinking the result into $I^3$. We will adopt a shorthand notation omitting signs for positive sequences. In other words, a sequence $\myuline{p} = (p_1,\ldots,p_j)$ will stand for the object $((+,p_1),\ldots,(+,p_j))$ of $\ribTL$. Furthermore, an object of $\ribTL$ which is not underlined will stand for a sequence with a single entry, and $\varnothing$ will stand for the empty sequence. 

The category $\ribTL$ naturally supports a ribbon structure. Tensor product of objects of $\ribTL$ is given by taking their concatenation, while tensor product of morphisms of $\ribTL$ is given by gluing horizontally two copies of $I^3$ and then shrinking the result into $I^3$. When representing graphically a morphism of $\ribTL$, we use the color blue. Then, left and right evaluation and coevaluation, braiding, and twist morphisms are defined, for all $p,p' \in \ribTL$, by
\begin{align}
 \lev_p &= \pic{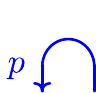} &
 \lcoev_p &= \pic{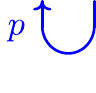} &
 \rev_p &= \pic{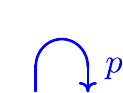} &
 \rcoev_p &= \pic{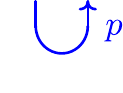} \label{E:rib_ev_coev} \\*[5pt]
 & &
 c_{p,p'} &= \pic{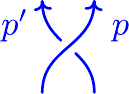} &
 \theta_p &= \pic{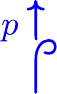} \label{E:rib_braiding_twist}
 & &
\end{align}
There exists a ribbon functor
\[
 \langle \_ \rangle : \ribTL \to \TL,
\]
which we refer to as the \textit{Kauffman bracket functor}, satisfying
\begin{align}\label{E:cabling_functor}
 \left\langle \pic{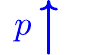} \right\rangle &= \pic{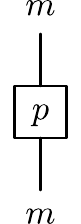} &
 \left\langle \pic{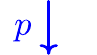} \right\rangle &= \pic{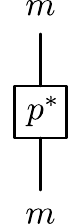}
\end{align}
for every idempotent $p \in \TL(m)$, and sending structure morphisms of equations~\eqref{E:TL_ev_coev} and \eqref{E:TL_braiding_twist} to those of equations~\eqref{E:rib_ev_coev} and \eqref{E:rib_braiding_twist} respectively.

For every $k \in \N$ a \textit{bichrome $k$-top tangle from $(\myuline{\epsilon},\myuline{p})$ to $(\myuline{\epsilon'},\myuline{p'})$}, sometimes simply called a \textit{top tangle}, is the union of a \textit{blue} $\TL$-labeled oriented framed tangle from $(\myuline{\epsilon},\myuline{p})$ to $(\myuline{\epsilon'},\myuline{p'})$ and a \textit{red} framed $(0,2k)$-tangle satisfying the following condition: for every $1 \leqs i \leqs k$, the $2i$th and $2i-1$th outgoing boundary points (starting from the left) are connected by a red component, while all the other incoming and outgoing boundary points belong to blue components. An example of a bichrome $2$-top tangle is given in Figure~\ref{F:bichrome_tangle}. In order to distinguish red components from blue ones in black-and-white versions of this paper, remark that red components are unoriented and unlabeled, while blue ones are oriented and labeled.

\begin{figure}[t]
 \centering
 \includegraphics{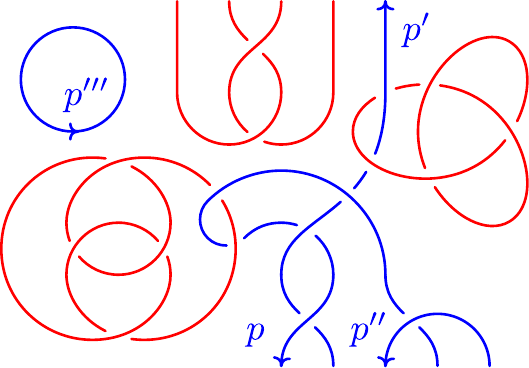}
 \caption{An example of a bichrome $2$-top tangle with source $((-,p),(+,p'),(-,p''),(+,p),(+,p''))$ and target $(+,p')$.}
\label{F:bichrome_tangle}
\end{figure}

We denote by $\calT_k(\myuline{p},\myuline{p'})$ the set of isotopy classes of bichrome $k$-top tangles from $\myuline{p}$ to $\myuline{p'}$ featuring no closed red component, and we adopt the shorthand notation $\calT_k(\myuline{p})$ when $\myuline{p} = \myuline{p'}$. A bichrome $k$-top tangle is simply called a \textit{bichrome tangle} if $k=0$. We denote by $\calB(\myuline{p},\myuline{p'})$ the set of isotopy classes of bichrome tangles from $\myuline{p}$ to $\myuline{p'}$, and we adopt the shorthand notation $\calB(\myuline{p})$ when $\myuline{p} = \myuline{p'}$. Every bichrome $k$-top tangle $T \in \calT_k(\myuline{p},\myuline{p'})$ determines a bichrome tangle $\pcl(T) \in \calB(\myuline{p},\myuline{p'})$ obtained by considering its \textit{plat closure}, that is,
\begin{equation}\label{E:plat_closure}
 \pcl \left( \pic{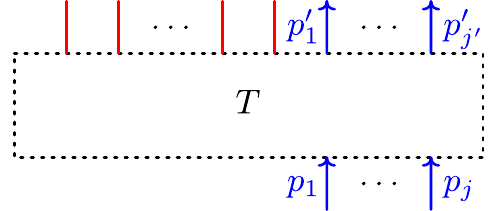} \right) := \pic{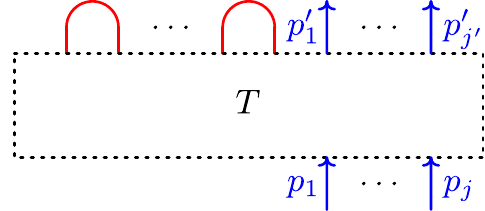} \hspace*{-5pt}
\end{equation}
Then, a \textit{top tangle presentation} of a bichrome tangle $T \in \calB(\myuline{p},\myuline{p'})$ is a bichrome top tangle $T' \in \calT_k(\myuline{p},\myuline{p'})$ whose plat closure is $T$. By definition, a top tangle presentation of a bichrome tangle has no closed red component.

A \textit{bichrome link} is a bichrome tangle from $\varnothing$ to itself. A bichrome link is \textit{admissible} if it features a \textit{projective blue component}, which is a blue component labeled by an idempotent $p \in \TL(m)$ of $\Idf$. If $p \in \TL(m)$ is an idempotent of $\Idf$, then every bichrome tangle $T \in \calB(p)$ determines an admissible bichrome link $\tcl(T) \in \calB(\varnothing)$ obtained by considering its \textit{trace closure}, that is,
\begin{equation}\label{E:trace_closure}
 \tcl \left( \pic{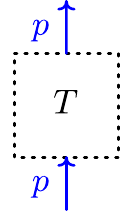} \right) := \pic{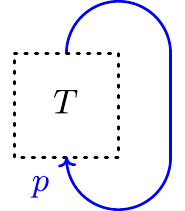}
\end{equation}
A \textit{cutting presentation} of an admissible bichrome link $T \in \calB(\varnothing)$ is a top tangle presentation $T'' \in \calT_k(p)$ of a bichrome tangle $T' \in \calB(p)$ whose trace closure is $T$, for some idempotent $p \in \TL(m)$ of $\Idf$. In other words, if $T \in \calB(\varnothing)$ is an admissible bichrome link and $T' \in \calT_k(p)$ is a cutting presentation of $T$, then
\begin{equation*}
 \pic{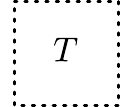} = \pic{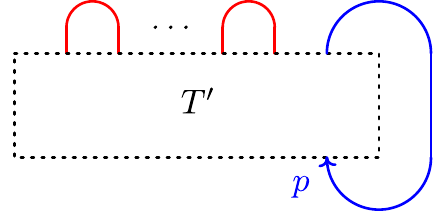}
\end{equation*}

\subsection{Ribbon Kauffman bracket}\label{S:ribbon_bracket}

In order to define a topological invariant of admissible bichrome links, we first need to adjust the sign in front of the Kauffman bracket appropriately. The reason for this is rather subtle, and will be explained more carefully in Section~\ref{S:ribbon_structures}, once the definition of the small quantum group $\bar{U} := \bar{U}_q \fsl_2$ will have been recalled. Very briefly, the problem originates from the comparison between the Temperley--Lieb category $\TL$ and the category of finite-dimensional representations $\mods{\bar{U}}$. Indeed, both are ribbon categories, and a dictionary between the two is provided
by a faithful linear functor $F_\TL : \TL \to \mods{\bar{U}}$ which preserves braided monoidal structures. However, $F_\TL$ does not preserve ribbon structures. In particular, partial traces and twist morphisms of $\TL$ are translated to those of $\mods{\bar{U}}$ only up to a sign. If we want the diagrammatic construction based on $\TL$ to replicate the algebraic computation based on $\bar{U}$, as defined in \cite{DGP17}, this sign needs to be controlled. This is precisely what we will do now, following \cite[Theorem~H.3]{O02}.

First of all, we recall that, if $K$ and $K'$ are disjoint oriented knots in $S^3$, then their \textit{linking number $\lk(K,K')$} is defined as the transverse intersection number $S \pitchfork K' \in \Z$, where $S$ is a Seifert surface for $K$. Alternatively, $\lk(K,K')$ can be computed as the difference between the number of positive and negative crossings of $K$ over $K'$ in a diagram of $K \cup K'$. We also recall that, if $K$ is a framed oriented knot in $S^3$, its \textit{framing number $\slk(K)$} is defined as $\lk(K,\tilde{K}) \in \Z$, where $\tilde{K}$ is a parallel copy of $K$ determined by its framing. The framing number is actually independent of the orientation of $K$.

A \textit{blue link} is a bichrome link without red components. If $T$ is a blue link with components $T_1, \ldots, T_\ell$ labeled by $p_1 \in \TL(m_1), \ldots, p_\ell \in \TL(m_\ell)$ respectively, we define its \textit{ribbon number $\rb(T)$} as the integer
\begin{equation}\label{E:ribbon_number}
 \rb(T) := \sum_{i=1}^\ell m_i \left( \slk(T_i) + 1 \right).
\end{equation}
The \textit{ribbon Kauffman bracket of a blue link $T$} is the scalar $\langle T \rangle^\rb \in \TL(0) = \C$ defined by
\begin{equation}\label{E:ribbon_bracket}
 \langle T \rangle^\rb := (-1)^{\rb(T)} \langle T \rangle,
\end{equation}
as considered by Ohtsuki in \cite[Theorem~H.3]{O02}.

Let now $r-1 \leqs n_1,\ldots,n_k \leqs 2r-2$ be integers. 
The \textit{$(n_1,\ldots,n_k)$-labeling of a bichrome top tangle $T \in \calT_k(p)$} is defined as the blue $\TL$-labeled oriented framed tangle $\lb_{n_1,\ldots,n_k}(T)$ from $(+,p)$ to $((+,n_1),(-,n_1),\ldots,(+,n_k),(-,n_k),(+,p))$ obtained from $T$ by turning all its red components blue, by orienting them from right to left, and by labeling the $i$th one by $n_i$ for all $1 \leqs i \leqs k$, that is,
\begin{equation}\label{E:labeling_top_tangle}
 \lb_{n_1,\ldots,n_k} \left( \pic{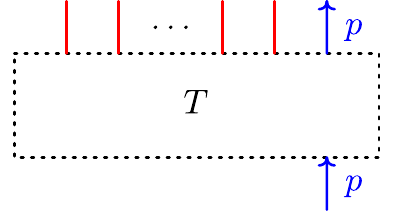} \right) := \pic{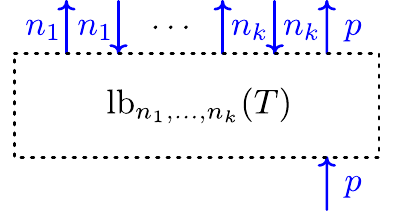} 
\end{equation}
Then the \textit{$(n_1,\ldots,n_k)$-bracket of a bichrome top tangle $T \in \calT_k(p)$} is the morphism $\langle T \rangle_{n_1,\ldots,n_k} \in \TL(p,n_1 \otimes n_1 \otimes \ldots \otimes n_k \otimes n_k \otimes p)$ defined by
\begin{equation}\label{E:bracket_top_tangle}
 \langle T \rangle_{n_1,\ldots,n_k} := \langle \lb_{n_1,\ldots,n_k}(T) \rangle.
\end{equation}
We can extend the ribbon number defined in equation~\eqref{E:ribbon_number} for blue links to the $(n_1,\ldots,n_k)$-labeling of a bichrome top tangle $T \in \calT_k(p)$ with a single non-closed blue component $B$ labeled by $p \in \TL(m)$, with closed blue components $B_1, \ldots, B_j$ labeled by $p_1 \in \TL(m_1), \ldots, p_j \in \TL(m_\ell)$, and with red components $R_1, \ldots, R_k$, by setting, in the notation of equations~\eqref{E:plat_closure} and \eqref{E:trace_closure} for plat and trace closures,
\begin{equation}\label{E:ribbon_number_top_tangle}
 \rb(\lb_{n_1,\ldots,n_k}(T)) := m \ \slk( \tcl(B) ) + \sum_{i=1}^j m_i \left( \slk(B_i) + 1 \right) + \sum_{i=1}^k n_i \ \slk( \pcl(R_i) ).
\end{equation}
The \textit{ribbon $(n_1,\ldots,n_k)$-bracket of a bichrome top tangle $T \in \calT_k(p)$} is the morphism $\langle T \rangle_{n_1,\ldots,n_k}^\rb \in \TL(p,n_1 \otimes n_1 \otimes \ldots \otimes n_k \otimes n_k \otimes p)$ defined by
\begin{equation}\label{E:ribbon_bracket_top_tangle}
 \langle T \rangle_{n_1,\ldots,n_k}^\rb := (-1)^{\rb(\lb_{n_1,\ldots,n_k}(T))} \langle T \rangle_{n_1,\ldots,n_k}.
\end{equation}

\subsection{Bichrome link invariant}\label{S:bichrome_link_invariant}

Let us define a topological invariant of admissible bichrome links. In order to do this, let us set
\begin{align*}
 \pic{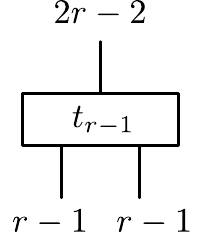} &:= \pic{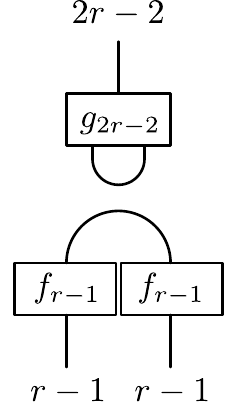} \in \TL(f_{r-1} \otimes f_{r-1},g_{2r-2}), \\
 \pic{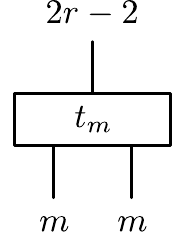} &:= \pic{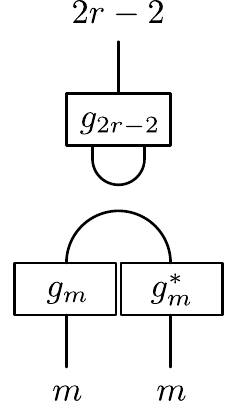} \in \TL(g_m \otimes g_m^*,g_{2r-2}), \\
 \pic{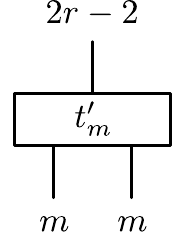} &:= \pic{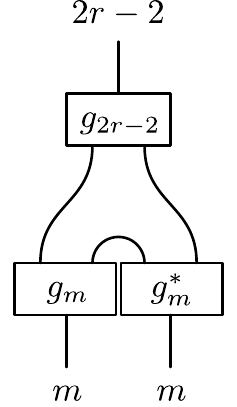} \in \TL(g_m \otimes g_m^*,g_{2r-2})
\end{align*}
for every integer $r \leqs m \leqs 2r-2$. Remark that we are using the notation introduced in equation~\eqref{E:morphism_completion}, which means we are considering Jones--Wenzl idempotents as objects of $\TL$. See Section~\ref{SS:level_3} for an explicit computation in the case $r=3$.

\begin{definition}\label{D:ns_Kirby_color} 
 The \textit{non-semisimple Kirby color $\nsKirby$} is the linear combination 
 of morphisms
 \[
  \nsKirby := \sum_{m=r-1}^{2r-2} \nsKirby_m \in \bigoplus_{m=r-1}^{2r-2} \TL(m \otimes m,2r-2),
 \]
 where $\nsKirby_m \in \TL(m \otimes m,2r-2)$ is given by
 \begin{align*}
  \nsKirby_{r-1} &:= t_{r-1}, &
  \nsKirby_m &:= (-1)^m \frac{\{ m+1 \}'}{2} t_m - (-1)^m [m+1] t'_m.
 \end{align*}
\end{definition}

If $p \in \TL(m)$ is an idempotent of $\Idf$, the non-semisimple Kirby color $\nsKirby$ can be used to associate with every bichrome $k$-top tangle $T \in \calT_k(p)$ an endomorphism $F_{\nsKirby,p}(T) \in \TL(p)$, in the notation of equation~\eqref{E:morphism_completion}. Indeed, if $u \in \TL(m,f_{r-1} \otimes m')$ and $u' \in \TL(f_{r-1} \otimes m',m)$ are morphisms satisfying $p = u'u$, then, using the ribbon Kauffman bracket of equation~\eqref{E:ribbon_bracket_top_tangle}, which coincides with the standard Kauffman bracket of equation~\eqref{E:bracket_top_tangle} up to a sign, we can set
\[
 F_{\nsKirby,p} \left( \pic{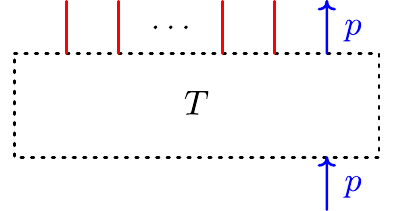} \right) := \hspace*{-10pt} \sum_{n_1,\ldots,n_k = r-1}^{2r-2} \pic[37.5]{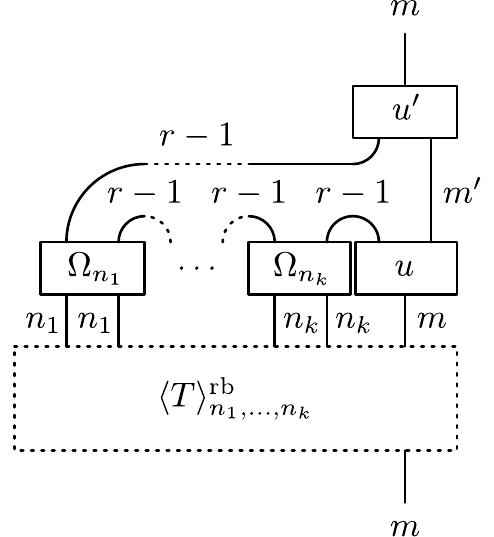}
\]

\begin{proposition}\label{P:ad_bic_link_inv}
 If $T \in \calB(\varnothing)$ is an admissible bichrome link, $p \in \TL(m)$ is an idempotent of $\Idf$, and $T' \in \calT_k(p)$ is a cutting presentation of $T$, then
 \begin{equation}\label{E:ad_bic_link_inv}
  F'_\nsKirby(T) := (-1)^m \rmt^{\TL}_p(F_{\nsKirby,p}(T'))
 \end{equation}
 is a topological invariant of $T$, meaning it is independent of the choice of $p$ and $T'$.
\end{proposition}

For a proof of Proposition~\ref{P:ad_bic_link_inv}, see Section~\ref{S:proof_bichrome_link_invariant}.

\begin{remark}\label{R:equivalence}
 Let us consider the equivalence relation $\sim$ on the vector space
 \[
  \bigoplus_{m=0}^{2r-2} \TL(m \otimes m,2r-2)
 \]
 determined, for all $0 \leqs m,n \leqs 2r-2$, $u \in \TL(n \otimes m,2r-2)$, and $v \in \TL(m,n)$, by 
 \[
  \pic{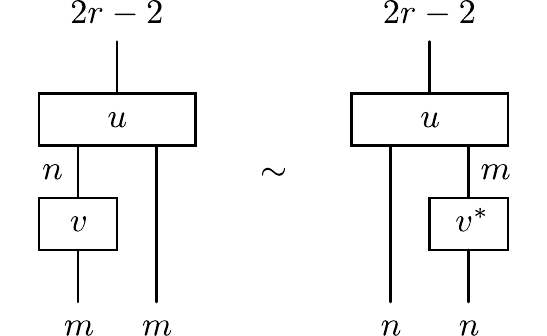}
 \]
 Using the Kauffman bracket of equation~\eqref{E:bracket_top_tangle}, every $T \in \calT_1(\varnothing)$ satisfies
 \[
  \pic{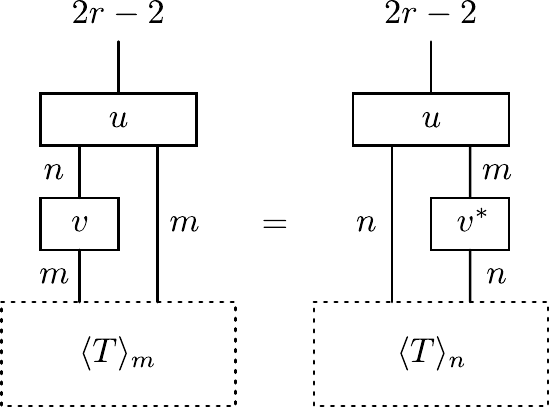}
 \] 
 Indeed, this can be shown using isotopy and the naturality of the braiding of $\TL$. Furthermore, $\TL(m,n)$ is non-zero only if $m-n$ is even, in which case
 \[
  (-1)^{m \ \slk(\pcl(T))} = (-1)^{n \ \slk(\pcl(T))}.
 \]
 Then it is easy to see that every linear combination of morphisms which is equivalent to the non-sem\-i\-sim\-ple Kirby color $\nsKirby$ determines the same topological invariant $F'_\nsKirby$.
\end{remark}

\subsection{3-Manifold invariant}\label{S:3-manifold_invariant}

We are now ready to define a topological invariant of closed 3-manifolds decorated with admissible bichrome links. The definition relies on the following computation.

\begin{lemma}\label{L:nstab}
 For every admissible bichrome link $T \in \calB(\varnothing)$ we have
 \begin{align}\label{E:nstab}
  F'_\nsKirby \left( \hspace*{-2.5pt} \pic{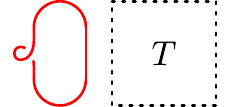} \right) &= \nstabp F'_\nsKirby(T), &
  F'_\nsKirby \left( \hspace*{-2.5pt} \pic{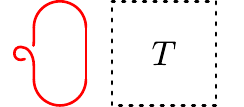} \right) &= \nstabm F'_\nsKirby(T),
 \end{align}
 where
 \begin{align*}
  \nstabp &:= i^{-\frac{r-1}{2}} r^{\frac 32} q^{\frac{r-3}{2}}, &
  \nstabm &:= i^{\frac{r-1}{2}} r^{\frac 32} q^{\frac{r+3}{2}}.
 \end{align*}
\end{lemma}

A proof of Lemma~\ref{L:nstab} will be given in Section~\ref{S:proof_3-manifold_invariant}. We are now ready to recall our main statement.

\begin{theorem}\label{T:3-manifold_inv}
 If $M$ is a closed 3-manifold, $T \subset M$ is an admissible bichrome link, and $L \subset S^3$ is a red surgery presentation of $M$ with positive signature~$\sigma_+$ and negative signature~$\sigma_-$, then
 \begin{equation}\label{E:ad_3-manifold_inv}
  Z_r(M,T) := \frac{F'_\nsKirby(L \cup T)}{\nstabp^{\sigma_+}\nstabp^{\sigma_-}}
 \end{equation}
is a topological invariant of $(M,T)$, meaning it is independent of the choice of $L$.
\end{theorem}

For a proof of Theorem~\ref{T:3-manifold_inv}, see Section~\ref{S:proof_3-manifold_invariant}. In the meantime, let us relate $Z_r$ to its semisimple counterpart by reviewing the construction of the $\SO(3)$ WRT invariant $\WRT_r$. In order to do this, let us set
\begin{align*}
 \pic{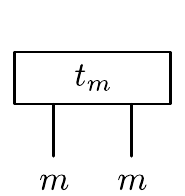} &:= \pic{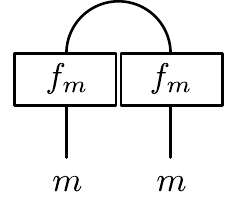} \in \TL(f_m \otimes f_m,f_0)
\end{align*}
for every integer $0 \leqs m \leqs r-2$. The \textit{semisimple Kirby color $\ssKirby$} is the linear combination of morphisms
\[
 \ssKirby := \sum_{m=0}^{r-2} \ssKirby_m \in \bigoplus_{m=0}^{r-2} \TL(m \otimes m,0),
\]
where $\ssKirby_m \in \TL(m \otimes m,0)$ is given by
\[
 \ssKirby_m := 
 \begin{cases}
  [m+1] t_m & \mbox{if} \quad m \equiv 0 \mod 2, \\
  0 & \mbox{if} \quad m \equiv 1 \mod 2.
 \end{cases}
\]

\begin{remark}
 Another possibility is to consider the \textit{double Kirby color $\dsKirby$}, which is the linear combination of morphisms
 \[
  \dsKirby := \sum_{m=0}^{r-2} \dsKirby_m \in \bigoplus_{m=0}^{r-2} \TL(m \otimes m,0),
 \]
 where $\dsKirby_m \in \TL(m \otimes m,0)$ is given by
 \[
  \dsKirby_m := (-1)^m [m+1] t_m.
 \]
 However, the invariant obtained from $\dsKirby$ turns out to be essentially equivalent to the one obtained from $\ssKirby$. Furthermore, in order to extend $\WRT_r$ to a TQFT, only the (idempotent completion of the) full subcategory of $\TL$ whose objects are even integers should be considered. A proof of this fact, in the equivalent language of the small quantum group $\bar{U}_q \fsl_2$, is given in \cite[Section~5]{BD20}. Therefore, in accordance with \cite[Section~2.3]{CKS07}, we use $\ssKirby$ instead of $\dsKirby$.
\end{remark}

The semisimple Kirby color $\ssKirby$ can be used to associate with every bichrome $k$-top tangle $T \in \calT_k(\varnothing)$ a scalar $F_{\ssKirby}(T) \in \TL(0) = \C$. Indeed, using the Kauffman bracket of equation~\eqref{E:bracket_top_tangle}, we can set
\[
 F_{\ssKirby} \left( \pic[7.5]{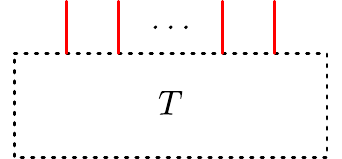} \right) := \hspace*{-10pt} \sum_{n_1,\ldots,n_k = 0}^{r-2} \pic[15]{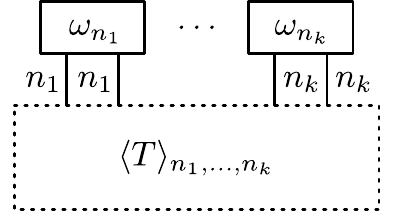}
\]
Remark that, since $\ssKirby$ is a linear combination of even terms, replacing the Kauffman bracket of equation~\eqref{E:bracket_top_tangle} with its ribbon version of equation~\eqref{E:ribbon_bracket_top_tangle} in the definition of $F_{\ssKirby}$ yields the same result. Now, if $T \in \calB(\varnothing)$ is a bichrome link and $T' \in \calT_k(\varnothing)$ is a top tangle presentation of $T$, then
\[
 F_\ssKirby(T) := F_\ssKirby(T')
\]
is a topological invariant of $T$, meaning it is independent of the choice of $T'$. Indeed, this is clear because the semisimple Kirby color $\ssKirby$ has target $0 \in \TL$, which means $F_\ssKirby(T)$ can be alternatively defined starting directly from a diagram of $T$, without ever choosing a top tangle presentation $T'$.

\begin{lemma}\label{L:sstab}
 We have
 \begin{align}\label{E:sstab}
  F_{\ssKirby} \left( \pic{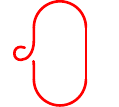} \right) &= \sstabp, &
  F_{\ssKirby} \left( \pic{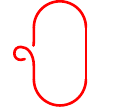} \right) &= \sstabm,
 \end{align}
 where
 \begin{align*}
  \sstabp &:= \frac{i^{-\frac{r-1}{2}} r^{\frac 12} q^{\frac{r-3}{2}}}{\{ 1 \}}, &
  \sstabm &:= -\frac{i^{\frac{r-1}{2}} r^{\frac 12} q^{\frac{r+3}{2}}}{\{ 1 \}}.
 \end{align*}
\end{lemma}

A proof of Lemma~\ref{L:sstab} can be found in Section~\ref{S:proof_3-manifold_invariant}. The following result follows from \cite[Theorem~III.1]{B92}, where $\WRT_r$ is denoted $\tilde{\theta}_A$.

\begin{theorem}[Blanchet]
 If $M$ is a closed 3-manifold, $T \subset M$ is a bichrome link, and $L \subset S^3$ is a red surgery presentation of $M$ with positive signature~$\sigma_+$ and negative signature~$\sigma_-$, then
 \begin{equation}\label{E:3-manifold_inv}
  \WRT_r(M,T) := \frac{F_{\ssKirby}(L \cup T)}{\sstabp^{\sigma_+}\sstabm^{\sigma_-}}
 \end{equation}
 is a topological invariant of $(M,T)$, meaning it is independent of the choice of $L$. 
\end{theorem}

If $\rmb_1(M)$ denotes the first Betti number of $M$, let us set
\[
 h_1(M) = 
 \begin{cases}
  \left| H_1(M) \right| & \mbox{if} \quad \rmb_1(M) = 0, \\
  0 & \mbox{if} \quad \rmb_1(M) > 0.
 \end{cases}
\]

\begin{proposition}\label{P:connected_sum}
 If $M$ and $M'$ are closed 3-manifolds, $T \subset M$ is a bichrome link, and $T' \subset M'$ is an admissible bichrome link, then
 \[
  Z_r(M \# M',T \cup T') = h_1(M) \WRT_r(M,T) Z_r(M',T').
 \]
\end{proposition}

For a proof of Proposition~\ref{P:connected_sum}, see Section~\ref{S:proof_3-manifold_invariant}. For the moment, let us draw a few simple consequences from it. First of all, if both $T$ and $T'$ are admissible, then
\[
 Z_r(M \# M',T \cup T') = 0,
\]
because $\WRT_r$ vanishes against 3-manifolds decorated with admissible bichrome links, as a consequence of equation~\eqref{E:t_f} with $m = r-1$. Next, if $M$ is a closed 3-manifold and $O \subset S^3$ is a blue unknot of framing 0 and label $f_{r-1}$, then
\[
 Y_r(M) := Z_r(M \# S^3,O)
\]
is a topological invariant of $M$. Furthermore, the chosen normalization of the modified trace $\rmt$ on $\Idf$ implies, by definition, that
\[
 Z_r(S^3,O) = 1.
\]
Then Proposition~\ref{P:connected_sum} immediately yields
\[
 Y_r(M) = h_1(M) \WRT_r(M).
\]

\subsection{Level 3}\label{SS:level_3}

Let us unpack some definitions for the first level $r = 3$. In this case, we have $q = e^{\frac{2 \pi i}{3}}$, which means $[2] = -1$. For what concerns both simple and non-semisimple Jones--Wenzl idempotents, as well as their nilpotent endomorphisms, a direct computation gives
\begin{align*}
 \pic{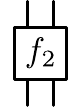} &= \pic{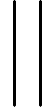} - \pic{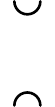} \\[10pt] 
 \pic{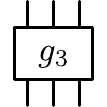} &= \pic{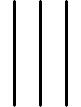} - \pic{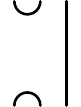} \\[10pt] 
 \pic{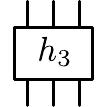} &= \pic{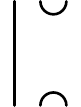} - \pic{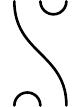} - \pic{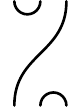} + \pic{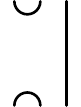} \\[10pt] 
 \pic{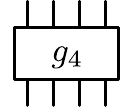} &= \pic{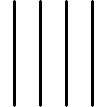} - \pic{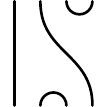} - \pic{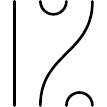} - \pic{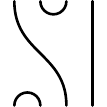} - \pic{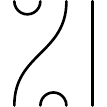} + \pic{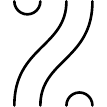} + \pic{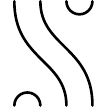} \\*[10pt] 
 &\phantom{{}={}} \mathllap{+} \pic{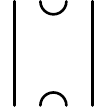} - \pic{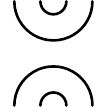} + 2 \cdot \hspace*{-4.75pt} \pic{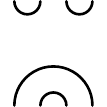} + 2 \cdot \hspace*{-4.75pt} \pic{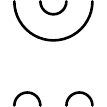} - 3 \cdot \hspace*{-4.75pt} \pic{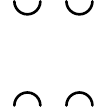} \\[10pt] 
 \pic{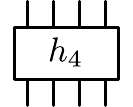} &= \pic{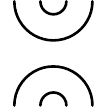} - \pic{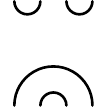} - \pic{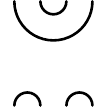} + \pic{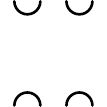}
\end{align*}
For what concerns components of the non-semisimple Kirby color, we have
\begin{align*}
 \pic{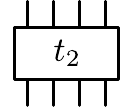} &\sim \pic{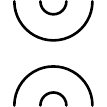} - \pic{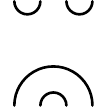} - \pic{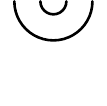} + \pic{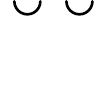} \\[10pt] 
 \pic{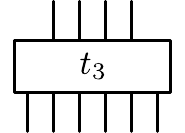} &\sim \pic{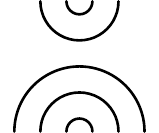} - \pic{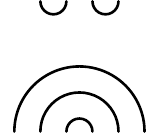} - \pic{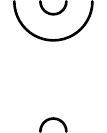} + \pic{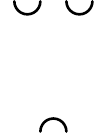} \\[10pt] 
 \pic{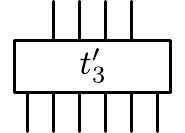} &\sim \pic{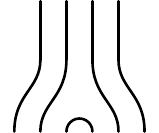} - \pic{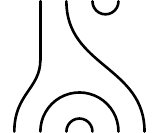} + \pic{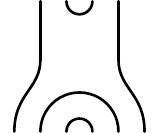} - \pic{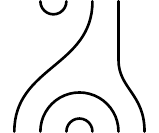} \\*[10pt] 
 &\phantom{{}={}} \mathllap{-} \pic{level_3_t_3_1} + 2 \cdot \hspace*{-4.75pt} \pic{level_3_t_3_2} - \pic{level_3_t_3_4} \\[10pt] 
 \pic{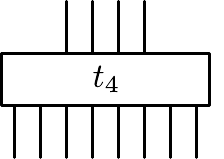} &\sim \pic{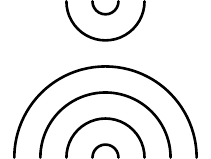} - \pic{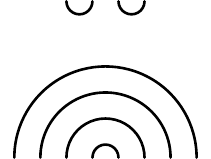} - 3 \cdot \hspace*{-4.75pt} \pic{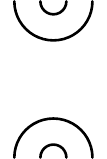} + 3 \cdot \hspace*{-4.75pt} \pic{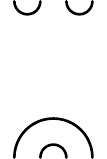} \\*[10pt] 
 &\phantom{{}={}} \mathllap{{}+ 2 \cdot{}} \pic{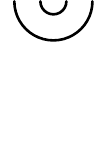} - 2 \cdot \hspace*{-4.75pt} \pic{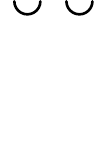} \\[10pt] 
 \pic{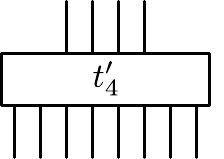} &\sim \pic{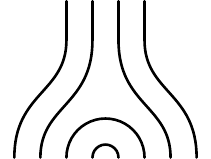} - \pic{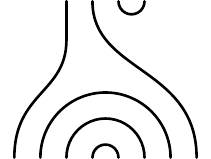} + \pic{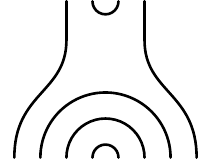} - \pic{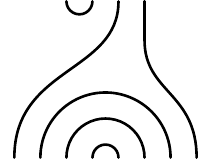} \\*[10pt] 
 &\phantom{{}={}} \mathllap{-} \pic{level_3_t_4_1} + 2 \cdot \hspace*{-4.75pt} \pic{level_3_t_4_2} \\*[10pt] 
 &\phantom{{}={}} \mathllap{{}- 2 \cdot{}} \pic{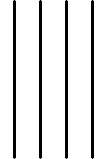} + 2 \cdot \pic{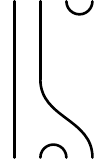} - 2 \cdot \pic{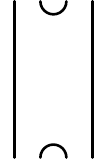} + 2 \cdot \hspace*{-4.75pt} \pic{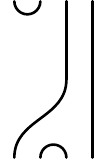} \\*[10pt] 
 &\phantom{{}={}} \mathllap{{}+ 3 \cdot{}} \pic{level_3_t_4_3} - 6 \cdot \hspace*{-4.75pt} \pic{level_3_t_4_4} - \pic{level_3_t_4_5} + 3 \cdot \hspace*{-4.75pt} \pic{level_3_t_4_6}
\end{align*}
where $\sim$ denotes the equivalence relation introduced in Remark~\ref{R:equivalence}.

\section{Small quantum group}\label{S:small_quantum_group}

In this section we recall the definition of the small quantum group of $\fsl_2$ at odd roots of unity, which was first given by Lusztig in \cite{L90}, as well as crucial results concerning its representation theory.

\subsection{Definition}\label{SS:small_quantum_group}

We denote by $\bar{U} = \bar{U}_q \fsl_2$ the \textit{small quantum group of $\fsl_2$}, which is defined as the algebra over $\C$ with generators $\{ E,F,K \}$ and relations
\begin{gather*}
 E^r = F^r = 0, \qquad K^r = 1, \\*
 K E K^{-1} = q^2 E, \qquad K F K^{-1} = q^{-2} F, \qquad [E,F] = \frac{K - K^{-1}}{q-q^{-1}}.
\end{gather*}
A Poincaré--Birkhoff--Witt basis of $\bar{U}$ is given by
\[
 \left\{ E^a F^b K^c \mid \ 0 \leqs a,b,c \leqs r - 1 \right\}.
\]
We make $\bar{U}$ into a Hopf algebra by setting
\begin{align*}
 \Delta(E) &= E \otimes K + 1 \otimes E, & \varepsilon(E) &= 0, & S(E) &= -E K^{-1}, \\
 \Delta(F) &= K^{-1} \otimes F + F \otimes 1, & \varepsilon(F) &= 0, & S(F) &= - K F, \\
 \Delta(K) &= K \otimes K, & \varepsilon(K) &= 1, & S(K) &= K^{-1}.
\end{align*}
We also fix an \textit{R-matrix} $R = R' \otimes R'' \in \bar{U} \otimes \bar{U}$ given by
\[
 R := \frac{1}{r} \sum_{a,b,c=0}^{r-1} \frac{\{ 1 \}^a}{[a]!}
 q^{\frac{a(a-1)}{2} - 2bc} K^b E^a \otimes K^c F^a,
\]
whose inverse $R^{-1} = S(R') \otimes R'' \in \bar{U} \otimes \bar{U}$ is given by
\[
 R^{-1} = \frac{1}{r} \sum_{a,b,c=0}^{r-1} \frac{\{ -1 \}^a}{[a]!}
 q^{-\frac{a(a-1)}{2} + 2bc} E^a K^b \otimes F^a K^c.
\]
Furthermore, we fix a \textit{ribbon element} $v_+ \in \bar{U}$ given by
\[
 v_+ := \frac{i^{\frac{r-1}{2}}}{\sqrt{r}} \sum_{a,b = 0}^{r - 1} \frac{\{ -1 \}^a}{[a]!} q^{-\frac{a(a-1)}{2} + \frac{(r+1)(a-b-1)^2}{2}} F^a K^b E^a,
\]
whose inverse $v_- \in \bar{U}$ is given by
\[
 v_- = \frac{i^{-\frac{r-1}{2}}}{\sqrt{r}} \sum_{a,b = 0}^{r - 1} \frac{\{ 1 \}^a}{[a]!} q^{\frac{a(a-1)}{2} + \frac{(r-1)(a+b-1)^2}{2}} F^a K^b E^a.
\]
These data make $\bar{U}$ into a ribbon Hopf algebra, and also determine further additional structures. For instance, the unique \textit{pivotal element} $g \in \bar{U}$ compatible with the specified ribbon structure, in the sense that $g = S(R'')R'v_-$, is given by $g := K$, see \cite[Proposition~XIV.6.5]{K95}. Furthermore, the \textit{M-matrix} $M = M' \otimes M'' \in \bar{U} \otimes \bar{U}$ defined by 
\[
 M := R_{21}R_{12},
\]
where $R_{12} = R$ is the R-matrix and $R_{21}$ is obtained from $R$ by reversing the order of its components, is given by
\[
 M = \frac{1}{r} \sum_{a,b,c,d=0}^{r-1} \frac{\{ 1 \}^{a+b}}{[a]![b]!} q^{\frac{a(a-1)+b(b-1)}{2} - 2 cd - (b+c)(b-d)} F^b K^c E^a \otimes E^b K^d F^a.
\]
This determines a linear map $D : \bar{U}^* \to \bar{U}$ defined by
\[
 D(\phi) := \phi(M') M''
\]
for every $\phi \in \bar{U}^*$. Maps of this form were first considered in \cite[Proposition~3.3]{D90}, and $D$ is therefore known as the \textit{Drinfeld map}. The ribbon Hopf algebra $\bar{U}$ is \textit{factorizable} in the sense that $D$ is an isomorphism, as first shown in \cite[Corollary~A.3.3]{L94}, see also \cite[Example~3.4.3]{M95}. In particular, thanks to the construction of \cite{DGP17}, the small quantum group $\bar{U}$ gives rise to a topological invariant of closed 3-manifolds, and more generally to a TQFT.

\subsection{Center}\label{S:center}

Let us describe the center $Z(\bar{U})$ of the algebra $\bar{U}$, which has been studied in detail by Kerler in \cite{K94} starting from the \textit{quantum Casimir element}
\begin{equation}\label{E:Casimir_def}
 C := EF + \frac{q^{-1}K + qK^{-1}}{\{ 1 \}^2} = FE + \frac{qK + q^{-1}K^{-1}}{\{ 1 \}^2} \in Z(\bar{U}).
\end{equation}
The minimal polynomial of $C$ is
\[
 \Psi(X) = \prod_{m=0}^{r-1} ( X - \beta_m ) = ( X - \beta_{r-1} ) \prod_{m=0}^{\frac{r-3}{2}} ( X - \beta_m )^2,
\]
where
\[
 \beta_m := \frac{\{ m+1 \}'}{\{ 1 \}^2}.
\]
Indeed, $\beta_{r-m-2} = \beta_m$ for every integer $0 \leqs m \leqs r-2$. If we set
\begin{align*}
 \Psi_{r-1}(X) &= \frac{\Psi(X)}{( X - \beta_{r-1} )}, &
 \Psi_m(X) &= \frac{\Psi(X)}{( X - \beta_m )^2}
\end{align*}
for every integer $0 \leqs m \leqs r-2$, then we can define the central elements
\begin{align}
 e_{r-1} &:= \frac{\Psi_{r-1}(C)}{\Psi_{r-1} ( \beta_{r-1} )} \in Z(\bar{U}), \label{E:central_idem_r-1} \\*
 e_m &:= \frac{\Psi_m(C)}{\Psi_m ( \beta_m )} - \frac{\Psi'_m ( \beta_m )( C - \beta_m ) \Psi_m(C)}{\Psi_m ( \beta_m )^2} \in Z(\bar{U}), \label{E:central_idem_m} \\*
 w_m &:= \frac{( C - \beta_m ) \Psi_m(C)}{\Psi_m ( \beta_m )} \in Z(\bar{U}). \label{E:central_nil_m}
\end{align}
Furthermore, if we consider the non-central projector
\begin{equation}\label{E:non-central_proj}
 v_m := \frac{1}{r} \sum_{a = 0}^{r-1} q^{-am} K^a \in \bar{U}
\end{equation}
on the eigenspace of eigenvalue $q^m$ for the regular action of $K$ on $\bar{U}$, then setting
\begin{equation}\label{E:sum_of_non-central_proj}
 T_m := \sum_{j=0}^m v_{m-2j} \in \bar{U}
\end{equation}
for every integer $0 \leqs m \leqs r-2$ allows us to decompose 
\[
 w_m = w_m^+ + w_m^-
\]
for the central elements
\begin{align}
 w_m^+ &:= T_m \frac{( C - \beta_m ) \Psi_m(C)}{\Psi_m ( \beta_m )} \in Z(\bar{U}),  \label{E:central_nil_m_+} \\*
 w_m^- &:= (1-T_m) \frac{( C - \beta_m ) \Psi_m(C)}{\Psi_m ( \beta_m )} \in Z(\bar{U}).  \label{E:central_nil_m_-}
\end{align}

It is proven in \cite[Lemma~14]{K94} that a basis of $Z(\bar{U})$ is given by
\begin{align*}
 \{ e_{r-1} \} \cup \left\{ e_m,w_m^+,w_m^- \Bigm| 0 \leqs m \leqs \frac{r-3}{2} \right\},
\end{align*}
and that basis vectors satisfy
\begin{align}\label{E:products}
 e_m e_{m'} &= \delta_{m,m'} e_m, & 
 w^\epsilon_m e_{m'} &= \delta_{m,m'} w^\epsilon_m, & 
 w^\epsilon_m w^{\epsilon'}_{m'} &= 0.
\end{align}
Remark that $e_{r-m-2} = e_m$ and $w_{r-m-2}^\epsilon = w_m^{- \epsilon}$ for every integer $0 \leqs m \leqs r-2$. For future convenience, we also set
\begin{align*}
 e_m &:= e_{2r-m-2} & w_m^\epsilon &:= w_{2r-m-2}^\epsilon & w_m &:= w_{2r-m-2}
\end{align*}
for every integer $r \leqs m \leqs 2r-2$.

\begin{lemma}\label{L:ribbon_Jordan}
 The ribbon element and its inverse $v_+,v_- \in Z(\bar{U})$ admit the Jordan decompositions
 \begin{equation*}
  v_\pm = q^{\frac{r \pm 1}{2}} e_{r-1} + \sum_{m=0}^{\frac{r-3}{2}} q^{\frac{r \mp 1}{2} m^2 \mp m} \left( e_m \mp \frac{(m+1) \{ 1 \}}{[m+1]} w_m^+ \mp \frac{(m-r+1) \{ 1 \}}{[m+1]} w_m^- \right).
 \end{equation*}
\end{lemma}

\begin{proof}
 The formula for $v_-$ is obtained from \cite[Lemma~15]{K94} by carefully comparing Kerler's conventions with ours. Indeed, what Kerler refers to as the ribbon element is actually $v_-$ in our notation. Furthermore, Kerler uses a different set of generators for the definition of $\bar{U}_q \fsl_2$, which can be obtained from ours by replacing $E$ with $\{ 1 \} E$. Consequently, Kerler's quantum Casimir element is equal to $\{ 1 \} C$, and his basis of $Z(\bar{U})$ is
 \begin{align*}
  \{ e_{r-1} \} \cup \left\{ e_m, \{ 1 \} w_m^+, \{ 1 \} w_m^- \Bigm| 0 \leqs m \leqs \frac{r-3}{2} \right\}.
 \end{align*}
 The formula for $v_+$ is obtained from the one for $v_-$ using equation~\eqref{E:products}.
\end{proof}

\subsection{Finite-dimensional representations}

The \textit{category $\mods{\bar{U}}$ of finite-di\-men\-sion\-al representations of $\bar{U}$} supports the structure of a ribbon category. In order to recall it, let us adopt Sweedler's notation for iterated coproducts, that is,
\[
 \Delta^{(m)}(x) = x_{(1)} \otimes \ldots \otimes x_{(m+1)} \in \bar{U}^{\otimes m+1}
\]
for all $x \in \bar{U}$ and $m \in \N$. Then the coproduct $\Delta$ is used to define, for all $V,W \in \mods{\bar{U}}$, the tensor product $V \otimes W$, which is determined by
\[
 x \cdot v \otimes w := (x_{(1)} \cdot v) \otimes (x_{(2)} \cdot w)
\]
for all $x \in \bar{U}$, $v \in V$, and $w \in W$. The counit $\epsilon$ is used to define the tensor unit $\C \in \mods{\bar{U}}$, which is determined by the trivial representation over $\C$ given by
\[
 x \cdot 1 := \epsilon(x)
\]
for every $x \in \bar{U}$. The antipode $S$ is used to define, for every $V \in \mods{\bar{U}}$, the two-sided dual $V^*$, which is determined by
\[
 (x \cdot \phi)(v) := \varphi(S(x) \cdot v)
\]
for all $x \in \bar{U}$, $\phi \in V^*$, and $v \in V$. The pivotal element $g$ and its inverse are used to define, for every $V \in \mods{\bar{U}}$, left and right evaluation and coevaluation morphisms $\lev_V : V^* \otimes V \to \C$, $\lcoev_V : \C \to V \otimes V^*$, $\rev_V : V \otimes V^* \to \C$, and $\rcoev_V : \C \to V^* \otimes V$, which are determined by
\begin{align*}
 \lev_V(\phi \otimes v) &:= \phi(v), &
 \lcoev_V(1) &:= \sum_{i=1}^n v_i \otimes \phi^i \\*
 \rev_V(v \otimes \phi) &:= \phi(g \cdot v), &
 \rcoev_V(1) &:= \sum_{i=1}^n \phi^i \otimes (g^{-1} \cdot v_i)
\end{align*}
for all $v \in V$ and $\phi \in V^*$, where $\{ v_i \in V \mid 1 \leqs i \leqs n \}$ and $\{ \phi^i \in V^* \mid 1 \leqs i \leqs n \}$ are dual bases. The R-matrix $R$ is used to define, for all $V,W \in \mods{\bar{U}}$, the braiding morphism $c_{V,W} : V \otimes W \to W \otimes V$, which is determined by
\[
 c_{V,W}(v \otimes w) := (R'' \cdot w) \otimes (R' \cdot v)
\]
for all $v \in V$ and $w \in W$. The inverse ribbon element $v_-$ is used to define, for every $W \in \mods{\bar{U}}$, the twist morphism $\theta_W : W \to W$, which is determined by
\[
 \theta_W(w) := v_- \cdot w
\]
for every $w \in W$.

\subsection{Simple and projective modules}\label{S:simple_indec_proj}

Let us recall the classification of \textit{simple} and \textit{indecomposable projective} $\bar{U}$-modules. For every integer $0 \leqs m \leqs r-1$ we denote by $X_m$ the simple $\bar{U}$-module with basis 
\[
 \{ a_j^m \mid 0 \leqs j \leqs m \}
\]
and action given, for all integers $0 \leqs j \leqs m$, by
\begin{align*}
 K \cdot a_j^m &= q^{m-2j} a_j^m, \\*
 E \cdot a_j^m &= [j][m-j+1] a_{j-1}^m, \\*
 F \cdot a_j^m &= a_{j+1}^m,
\end{align*}
where $a_{-1}^m := a_{m+1}^m := 0$. Among these, we highlight the \textit{fundamental representation} $X := X_1$, which is a monoidal generator for the family of $\bar{U}$-modules considered in this paper (in the sense that every simple and every indecomposable projective $\bar{U}$-module is a direct summand of a tensor power of $X$), and the \textit{Steinberg module} $X_{r-1}$, which is the only $\bar{U}$-module which is both simple and projective. Every simple object of $\mods{\bar{U}}$ is isomorphic to $X_m$ for some integer $0 \leqs m \leqs r-1$. Next, for every integer $r \leqs m \leqs 2r-2$ we denote by $P_m$ the indecomposable projective $\bar{U}$-module with basis 
\[
 \{ a_j^m, x_k^m, y_k^m, b_j^m \mid 0 \leqs j \leqs 2r-m-2, 0 \leqs k \leqs m-r \}
\] 
and action given, for all integers $0 \leqs j \leqs 2r-m-2$ and $0 \leqs k \leqs m-r$, by
\begin{align*}
 K \cdot a_j^m &= q^{-m-2j-2} a_j^m, \\*
 E \cdot a_j^m &= -[j][m+j+1] a_{j-1}^m, \\*
 F \cdot a_j^m &= a_{j+1}^m, \\
 K \cdot x_k^m &= q^{m-2k} x_k^m, \\*
 E \cdot x_k^m &= [k][m-k+1] x_{k-1}^m, \\*
 F \cdot x_k^m &= 
 \begin{cases}
  x_{k+1}^m & \mbox{if} \quad 0 \leqs k < m-r, \\
  a_0^m & \mbox{if} \quad k = m-r,
 \end{cases} \\
 K \cdot y_k^m &= q^{m-2k} y_k^m, \\*
 E \cdot y_k^m &= 
 \begin{cases}
  a_{2r-m-2}^m & \mbox{if} \quad k = 0, \\
  [k][m-k+1] y_{k-1}^m & \mbox{if} \quad 0 < k \leqs m-r,
 \end{cases} \\*
 F \cdot y_k^m &= y_{k+1}^m, \\
 K \cdot b_j^m &= q^{-m-2j-2} b_j^m, \\*
 E \cdot b_j^m &= 
 \begin{cases}
  x_{m-r}^m & \mbox{if} \quad j = 0, \\
  a_{j-1}^m - [j][m+j+1] b_{j-1}^m & \mbox{if} \quad 0 < j \leqs 2r-m-2, 
 \end{cases} \\*
 F \cdot b_j^m &= 
 \begin{cases}
  b_{j+1}^m & \mbox{if} \quad 0 \leqs j < 2r-m-2, \\
  y_0^m & \mbox{if} \quad j = 2r-m-2,
 \end{cases}
\end{align*}
where $a_{-1}^m := a_{2r-m-1}^m := x_{-1}^m := y_{m-r+1}^m := 0$. We point out that $P_m$ is sometimes denoted $P_{2r-m-2}$, because it is the projective cover of $X_{2r-m-2}$ for every integer $r \leqs m \leqs 2r-2$. Every indecomposable projective object of $\mods{\bar{U}}$ is isomorphic to either $X_{r-1}$ or $P_m$ for some integer $r \leqs m \leqs 2r-2$. 

Fusion formulas for decompositions of tensor products in $\mods{\bar{U}}$ are given by
\begin{align}
 X_{m-1} \otimes X_1 &\cong 
 \begin{cases}
  X_1 & \mbox{if} \quad m = 1, \\
  X_m \oplus X_{m-2} & \mbox{if} \quad 1 < m \leqs r-1, \\
  P_r & \mbox{if} \quad m = r,
 \end{cases} \label{E:simple} \\
 P_{m-1} \otimes X_1 &\cong 
 \begin{cases} 
  P_{r+1} \oplus X_{r-1} \oplus X_{r-1} & \mbox{if} \quad m = r+1, \\
  P_m \oplus P_{m-2} & \mbox{if} \quad r+1 < m \leqs 2r-2, \\
  X_{r-1} \oplus X_{r-1} \oplus P_{2r-3} & \mbox{if} \quad m = 2r-1.
 \end{cases} \label{E:indecomp_proj}
\end{align}
These formulas should be compared with equations~\eqref{E:f_0_def}-\eqref{E:f_m_def} and \eqref{E:g_r_def}-\eqref{E:g_m_def} using \cite[Lemma~4.1]{BDM19}.

As proved in \cite[Theorem~4.7.1]{GKP11}, the ideal $\IdX$ of projective objects of $\mods{\bar{U}}$ is generated by the Steinberg module $X_{r-1}$, and there exists a unique trace $\rmt^{\bar{U}}$ on $\IdX$ satisfying
\begin{equation}
 \rmt^{\bar{U}}_{X_{r-1}}(\id_{X_{r-1}}) = 1.
\end{equation}
Furthermore, $\rmt^{\bar{U}}$ is non-degenerate.

\subsection{Regular representation}

The regular representation of $\bar{U}$, which is determined by the regular action of $\bar{U}$ onto itself by left multiplication, decomposes into a direct sum of indecomposable projective $\bar{U}$-modules. Explicit bases for indecomposable projective factors of $\bar{U}$ are described in \cite[Section~4]{A08} when $r$ is even, but can be easily generalized to our case. Basis vectors are defined in terms of the non-central projectors $v_m \in \bar{U}$ given by equation~\eqref{E:non-central_proj} for $m \in \Z$. If, for every integer $0 \leqs j \leqs r-1$, we set
\[
 a_j^{r-1,n} := F^j E^{r-1} F^{r-n-1} v_{-2n-1},
\]
then
\[
 \left\{ a_j^{r-1,n} \in \bar{U} \Bigm| 0 \leqs j \leqs r-1 \right\}
\]
is a basis for a submodule $\bar{X}_{r-1,n}$ of $\bar{U}$ which is isomorphic to $X_{r-1}$. Similarly, if, for all integers $r \leqs m \leqs 2r-2$, $0 \leqs j \leqs 2r-m-2$, and $0 \leqs k \leqs m-r$, we set
\begin{align*}
 a_j^{m,n} &:= F^j E^{r-1} F^{r-n-1} v_{-m-2n-2}, \\*
 x_k^{m,n} &:= \sum_{h=0}^{m-r} \frac{[k]![h]!}{[m-k-r]![m-h-r]!} E^{m-k-h-1} F^{r-n-h-2} v_{-m-2n-2}, \\*
 y_k^{m,n} &:= \sum_{h=0}^{m-r} \frac{[m-r]![h]!}{[m-h-r]!} F^{2r-m+k-1} E^{r-h-2} F^{r-n-h-2} v_{-m-2n-2}, \\*
 b_j^{m,n} &:= \sum_{h=0}^{m-r} \frac{[m-r]![h]!}{[m-h-r]!} F^j E^{r-h-2} F^{r-n-h-2} v_{-m-2n-2},
\end{align*}
then
\[
 \left\{ a_j^{m,n}, x_k^{m,n}, y_k^{m,n}, b_j^{m,n} \in \bar{U} \Bigm| 0 \leqs j \leqs 2r-m-2, 0 \leqs k \leqs m-r \right\} 
\]
is a basis for a submodule $\bar{P}_{m,n}$ of $\bar{U}$ which is isomorphic to $P_m$. If we set 
\[
 \bar{X}_{r-1} := \bigoplus_{n=0}^{r-1} \bar{X}_{r-1,n}, \qquad
 \bar{P}_m := \bigoplus_{n=0}^{2r-m-2} \bar{P}_{m,n},
\]
then it can be easily checked that
\[
 \bar{U} \cong \bar{X}_{r-1} \oplus \bigoplus_{m=r}^{2r-2} \bar{P}_m.
\]

\subsection{Integral}\label{SS:decomposition_integral}

A key ingredient for the HKR approach to the construction of 3-manifold invariants are integrals, whose theory is well-established \cite{S69,R11}. A \textit{left integral} $\lambda \in \bar{U}^*$ and a \textit{right integral} $\mu \in \bar{U}^*$ are linear forms satisfying
\begin{align*}
 \lambda(x_{(2)}) x_{(1)} &= \lambda(x) 1 \in \bar{U}, &
 \mu(x_{(1)}) x_{(2)} &= \mu(x) 1 \in \bar{U}^*
\end{align*}
for every $x \in \bar{U}^*$. Since $\bar{U}$ is finite-di\-men\-sion\-al, left integrals and right integrals span one-di\-men\-sion\-al vector spaces. If $\lambda$ is a left integral, then $\lambda \circ S$ is a right integral, and similarly, if $\mu$ is a right integral, then $\mu \circ S$ is a left integral. Every left integral $\lambda \in \bar{U}^*$ satisfies
\begin{align}\label{E:quantum_char}
 \lambda(xy) &= \lambda(yS^2(x))
\end{align}
for all $x,y \in \bar{U}^*$, which means it can be regarded as an intertwiner between the trivial representation $\C$ and the \textit{adjoint representation} $\ad$, which is determined by the action of $\bar{U}$ onto itself given by
\[
 \ad_x(y) := x_{(1)}yS(x_{(2)})
\]
for all $x,y \in \bar{U}^*$. Every left integral $\lambda \in \bar{U}^*$ is of the form
\[
 \lambda \left( E^a F^b K^c \right) := \xi \delta_{a,r-1} \delta_{b,r-1} \delta_{c,r-1}
\]
for some $\xi \in \C$. For the purpose of our construction, it will be convenient to fix the left integral $\lambda \in \bar{U}^*$ determined by the coefficient
\[
 \xi := r ([r-1]!)^2 = \frac{r^3}{\{ 1 \}^{2r-2}} \in \C^*.
\]
The \textit{stabilization parameters} corresponding to this normalization are
\begin{align}\label{E:ns_stab_par}
 \lambda(v_+) &= i^{\frac{r-1}{2}} r^{\frac 32} q^{\frac{r+3}{2}}, &
 \lambda(v_-) &= i^{-\frac{r-1}{2}} r^{\frac 32} q^{\frac{r-3}{2}}.
\end{align}

Following the approach of \cite[Section~5]{A08}, we give an explicit decomposition of the left integral $\lambda \in \bar{U}^*$ into a linear combination of \textit{quantum traces} and \textit{pseudo quantum traces} corresponding to indecomposable projective $\bar{U}$-modules. This requires a few preliminary definitions. First of all, for every Hopf algebra $H$, we denote by $\QC(H)$ the space of \textit{quantum characters} of $H$, which are linear forms $\varphi \in H^*$ satisfying $\varphi(xy) = \varphi(yS^2(x))$ for all $x,y \in H$. Thanks to equation~\eqref{E:quantum_char}, left integrals are quantum characters. Similarly, for every algebra $A$, we denote by $\SLF(A)$ the space of \textit{symmetric linear functions} on $A$, which are linear forms $\varphi \in A^*$ satisfying $\varphi(xy) = \varphi(yx)$ for all $x,y \in A$. A pivotal element $g \in H$, which is a group-like element satisfying $gxg^{-1} = S^2(x)$ for every $x \in H$, defines an isomorphism between $\QC(H)$ and $\SLF(H)$ sending every quantum character $\phi$ to the symmetric linear function $\phi(\_ g^{-1})$. For a finite-dimensional ribbon Hopf algebra $H$ and a non-zero left integral $\lambda \in H^*$, \cite[Proposition~4.2]{H96} gives
\[
 \SLF(H) = \{ \lambda(\_ g^{-1}z) \in H^* \mid z \in Z(H) \}.
\]
Therefore, the idea is to use the basis of $Z(\bar{U})$ given in Section~\ref{S:center} in order to find a basis of $\QC(\bar{U})$, and then to compute the corresponding coefficients of $\lambda$.

If, for every integer $r-1 \leqs m \leqs 2r-2$, we denote by $\bar{Q}_m$ the generalized eigenspace of eigenvalue $\beta_m$ for the regular action of the quantum Casimir element $C \in Z(\bar{U})$ on the regular representation $\bar{U}$, then the same argument of \cite[Proposition~D.1.1]{FGST05} shows
\[
 \bar{Q}_m = 
 \begin{cases}
  \bar{X}_{r-1} & \mbox{if} \quad m = r-1, \\
  \bar{P}_m \oplus \bar{P}_{3r-m-2} & \mbox{if} \quad r \leqs m \leqs 2r-2.
 \end{cases}
\]
The regular action of the canonical central element $e_m \in Z(\bar{U})$ recovers the projector onto $\bar{Q}_m$ with respect to the decomposition
\[
 \bar{U} = \bigoplus_{m=r-1}^{3\frac{r-1}{2}} \bar{Q}_m,
\]
and a basis for $Z(\bar{Q}_m)$ is given by $\{ e_{r-1} \}$, if $m = r-1$, and by $\{ e_m,w_m^+,w_m^- \}$, if $r \leqs m \leqs 2r-2$. We will first decompose the symmetric linear form
\[
 \lambda(\_ K^{-1} e_m) \in \SLF(\bar{Q}_m)
\]
with respect to the corresponding basis of $\SLF(\bar{Q}_m)$.

If, for every integer $0 \leqs n \leqs r-1$, we set
\[
 A_j^{r-1,n} := \frac{1}{([r-1]!)^2} a_j^{r-1,n} \in \bar{X}_{r-1},
\]
then
\[
 \left\{ A_j^{r-1,n} \in \bar{X}_{r-1} \Bigm| 0 \leqs j,n \leqs r-1 \right\}
\]
is a basis of $\bar{X}_{r-1}$, and we can denote by 
\[
 \left\{ \psi^j_{r-1,n} \in \bar{X}^*_{r-1} \Bigm| 0 \leqs j,n \leqs r-1 \right\}
\]
the dual basis. Then, in the standard basis of $X_{r-1}$, the action of $\bar{X}_{r-1}$ is represented by the matrix 
\[
 \psi_{r-1} \in M_{r \times r}(\bar{X}^*_{r-1})
\]
whose $(j,n)$th entry is given by $\psi^j_{r-1,n}$. Therefore, if we set
\[
 \tau_{r-1} := \tr(\psi_{r-1}) = \sum_{n=0}^{r-1} \psi^n_{r-1,n},
\]
then $\{ \tau_{r-1} \}$ is a basis of $\SLF(\bar{Q}_{r-1})$. We call $\tau_{r-1}$ the \textit{$r-1$-trace}.

\begin{lemma}\label{L:dec_int_r-1}
 We have
 \[
  \lambda(\_ K^{-1} e_{r-1}) = \tau_{r-1}.
 \]
\end{lemma}

\begin{proof}
 On one hand, we have
 \begin{align*}
  \lambda(a_j^{r-1,n}K^{-1} e_{r-1}) 
  &= \lambda(F^jE^{r-1}F^{r-n-1}v_{-2n-1}K^{-1}) 
  = ([r-1]!)^2 \delta_{j,n}.
 \end{align*}
 This implies
 \begin{align*}
  \lambda(A_j^{r-1,n}K^{-1} e_{r-1}) &= \delta_{j,n}.
 \end{align*}
 On the other hand,
 \[ 
  \tau_{r-1}(A_j^{r-1,n}) = \delta_{j,n}. \qedhere
 \]
\end{proof}

Next, if, for all integers $r \leqs m \leqs 2r-2$ and $0 \leqs n \leqs 2r-m-2$, we set
\begin{align*}
 A_j^{m,n} &:= \frac{[m+1]^2}{([r-1]!)^2} a_j^{m,n} \in \bar{P}_m, \\*
 X_k^{m,n} &:= \frac{[m+1]^2}{([r-1]!)^2} x_k^{m,n} \in \bar{P}_m, \\*
 Y_k^{m,n} &:= \frac{[m+1]^2}{([r-1]!)^2} y_k^{m,n} \in \bar{P}_m, \\*
 B_j^{m,n} &:= \frac{[m+1]^2}{([r-1]!)^2} b_j^{m,n} + \frac{\{ m+1 \}'}{([r-1]!)^2} a_j^{m,n} \in \bar{P}_m,
\end{align*}
then
\[
 \left\{ A_j^{m,n}, X_k^{m,n}, Y_k^{m,n}, B_j^{m,n} \in \bar{P}_m \Bigm| 
 0 \leqs j,n \leqs 2r-m-2, 0 \leqs k \leqs m-r \right\}
\]
is a basis of $\bar{P}_m$, and we can denote by
\[
 \left\{ \psi^j_{m,n}, \xi^k_{m,n}, \zeta^k_{m,n}, \varphi^j_{m,n} \in \bar{P}^*_m \Bigm| 0 \leqs j,n \leqs 2r-m-2, 0 \leqs k \leqs m-r \right\}
\]
the dual basis. Then, in the standard bases of $P_m$ and of $P_{3r-m-2}$, the action of $\bar{P}_m$ is represented by the matrices 
\[
 \left( 
 \begin{matrix}
  \varphi_m & 0 & 0 & \psi_m \\
  0 & 0 & 0 & \xi_m \\
  0 & 0 & 0 & \zeta_m \\
  0 & 0 & 0 & \varphi_m
 \end{matrix} \right),
 \left( 
 \begin{matrix}
  0 & \zeta_m & \xi_m & 0 \\
  0 & \varphi_m & 0 & 0 \\
  0 & 0 & \varphi_m & 0 \\
  0 & 0 & 0 & 0
 \end{matrix} \right) \in M_{2r \times 2r}(\bar{P}^*_m),
\]
where 
\[
 \psi_m, \varphi_m \in M_{(2r-m-1) \times (2r-m-1)}(\bar{P}^*_m)
\]
denote the ma\-tri\-ces whose $(j,n)$th entries are given by $\psi^j_{m,n}$, $\varphi^j_{m,n}$, and where 
\[
 \xi_m, \zeta_m \in M_{(m-r+1) \times (2r-m-1)}(\bar{P}^*_m)
\]
denote the ma\-tri\-ces whose $(k,n)$th entries are given by $\xi^k_{m,n}$, $\zeta^k_{m,n}$ respectively. Therefore, if we set
\[
 \tau_m := \tr(\varphi_m) = \sum_{n=0}^{2r-m-2} \varphi^n_{m,n}, \qquad
 \tau'_m := \tr(\psi_m) = \sum_{n=0}^{2r-m-2} \psi^n_{m,n},
\]
then $\{ \tau_m, \tau_{3r-m-2}, \tau'_m + \tau'_{3r-m-2} \}$ is a basis of $\SLF(\bar{Q}_m)$. We call $\tau_m$ the \textit{$m$-trace} and $\tau'_m$ the \textit{pseudo $m$-trace}. Remark that neither $\tau'_m$ nor $\tau'_{3r-m-2}$ is a symmetric linear function, only their sum $\tau'_m + \tau'_{3r-m-2}$ is.

\begin{lemma}\label{L:dec_int_m}
 For every integer $r \leqs m \leqs 2r-2$ we have
 \[
  \lambda(\_ K^{-1} e_m) =  \{ m+1 \}' \left( \tau_m + \tau_{3r-m-2} \right) + [m+1]^2 \left( \tau'_m + \tau'_{3r-m-2} \right).
 \]
\end{lemma}

\begin{proof}
 On one hand, we have
 \begin{align*}
  &\lambda(a_j^{m,n}K^{-1} e_m) \\*
  &\hspace*{\parindent} = \lambda(F^jE^{r-1}F^{r-n-1}v_{-m-2n-2}K^{-1}) \\*
  &\hspace*{\parindent} = ([r-1]!)^2 \delta_{j,n}, \\
  &\lambda(x_k^{m,n}K^{-1} e_m) \\*
  &\hspace*{\parindent} = \sum_{h=0}^{m-r} \frac{[k]![h]!}{[m-k-r]![m-h-r]!} \lambda(E^{m-k-h-1} F^{r-n-h-2}v_{-m-2n-2}K^{-1}) \\*
  &\hspace*{\parindent} = 0, \\
  &\lambda(y_k^{m,n}K^{-1} e_m) \\*
  &\hspace*{\parindent} = \sum_{h=0}^{m-r} \frac{[m-r]![h]!}{[m-h-r]!} \lambda(F^{2r-m+k-1}E^{r-h-2} F^{r-n-h-2}v_{-m-2n-2}K^{-1}) \\*
  &\hspace*{\parindent} = 0, \\
  &\lambda(b_j^{m,n}K^{-1} e_m) \\*
  &\hspace*{\parindent} = \frac{[m+1]^2}{([r-1]!)^2} \sum_{h=0}^{m-r} \frac{[m-r]![h]!}{[m-h-r]!} \lambda(F^jE^{r-h-2} F^{r-n-h-2}v_{-m-2n-2}K^{-1}) \\*
  &\hspace*{\parindent} = 0.
 \end{align*}
 This implies
 \begin{align*}
  \lambda(A_j^{m,n}K^{-1} e_m) &= [m+1]^2 \delta_{j,n}, \\*
  \lambda(X_k^{m,n}K^{-1} e_m) &= 0, \\*
  \lambda(Y_k^{m,n}K^{-1} e_m) &= 0, \\*
  \lambda(B_j^{m,n}K^{-1} e_m) &= \{ m+1 \}' \delta_{j,n}.
 \end{align*}
 On the other hand,
 \begin{align*}
  \tau_m(A_j^{m,n}) &= 0, &
  \tau'_m(A_j^{m,n}) &= \delta_{j,n}, \\*
  \tau_m(X_k^{m,n}) &= 0, &
  \tau'_m(X_k^{m,n}) &= 0, \\*
  \tau_m(Y_k^{m,n}) &= 0, &
  \tau'_m(Y_k^{m,n}) &= 0, \\*
  \tau_m(B_j^{m,n}) &= \delta_{j,n}, &
  \tau'_m(B_j^{m,n}) &= 0. \qedhere
 \end{align*}
\end{proof}

The previous discussion implies a basis of $\QC(\bar{U})$ is given by
\[
 \left\{ \tau_m(\_ K) \mid r-1 \leqs m \leqs 2r-2 \right\} \cup
 \left\{ \tau'_m(\_ K) + \tau'_{3r-m-2}(\_ K) \Bigm| r \leqs m \leqs 3 \frac{r-1}{2} \right\}.
\]
For every integer $r-1 \leqs m \leqs 2r-2$ we call $\tau_m(\_ K)$ the \textit{quantum $m$-trace}, and for every integer $r \leqs m \leqs 2r-2$ we call $\tau'_m(\_ K)$ the \textit{pseudo quantum $m$-trace}. We are now ready to decompose the left integral.

\begin{proposition}\label{P:decomposition_integral}
 The left integral $\lambda \in \bar{U}^*$ can be written as
 \begin{equation}\label{E:decomposition_integral}
  \lambda = \tau_{r-1}(\_ K) + \sum_{m=r}^{2r-2} \{ m+1 \}' \tau_m(\_ K) + [m+1]^2 \tau'_m(\_ K).
 \end{equation}
\end{proposition}

\begin{proof}
 The formula follows from
 \[
  \lambda = \sum_{m=r-1}^{3\frac{r-1}{2}} \lambda(\_ e_m). \qedhere
 \]
\end{proof}

\subsection{Adjoint representation}\label{SS:decomposition_adjoint}

Let us highlight an important property of the adjoint representation $\ad$ of $\bar{U}$ which follows directly from the explicit description provided in \cite{Os95}. The adjoint action of $\bar{U}$ onto itself is determined by
\[
 \ad_E(x) = [E,x]K^{-1}, \qquad
 \ad_F(x) = K^{-1}[KF,x], \qquad
 \ad_K(x) = KxK^{-1}
\]
for every $x \in \ad$. We can define a $\Z$-grading on $\ad$ by setting 
\[
 \deg(E) = 1, \qquad
 \deg(F) = -1, \qquad
 \deg(K) = 0.
\]
Remark that for every integer $0 \leqs m \leqs r-1$ the generalized eigenspace $\bar{Q}_m$ defines a subrepresentation $\ad_m$ of $\ad$, because $C$ is central. In other words, we have
\[
 \ad = \bigoplus_{m=r-1}^{3\frac{r-1}{2}} \ad_m.
\]
If $\ad_m^0$ denotes the space of degree 0 vectors of $\ad_m$, a basis of $\ad_{r-1}^0$ is given by
\begin{align*}
 \{ K^a e_{r-1} \mid 0 \leqs a \leqs r-1 \},
\end{align*}
and similarly, for every integer $r \leqs m \leqs 2r-2$, a basis of $\ad_m^0$ is given by
\begin{align*}
 \{ K^a e_m, K^a w_m \mid 0 \leqs a \leqs r-1 \}.
\end{align*}

\begin{lemma}\label{L:ad_m}
 Every $\bar{U}$-module morphism from $\ad_m$ to $P_{2r-2}$ is uniquely determined by its restriction to $\ad_m^0$.
\end{lemma}

\begin{proof}
 As mentioned earlier, the proof follows from the explicit description of $\ad$ given in \cite{Os95}. Indeed, if $m = r-1$, the $\bar{U}$-module $\ad_{r-1}$ is projective, and it decomposes as
 \[
  \ubar{X}_{r-1,r-1} \oplus \bigoplus_{n=1}^{\frac{r-1}{2}} \ubar{P}_{r-1,r+2n-1},
 \]
 where $\ubar{X}_{r-1,r-1}$ is isomorphic to $X_{r-1}$, and $\ubar{P}_{r-1,r+2n-1}$ is isomorphic to $P_{r+2n-1}$. Each of these submodules is generated by some vector in $\ad_{r-1}^0$. On the other hand, if $r \leqs m \leqs 3\frac{r-1}{2}$, the $\bar{U}$-module $\ad_m$ is not projective, and it decomposes as the direct sum of a projective submodule 
 \[
  \ubar{X}_{m,r-1} \oplus \ubar{X}'_{m,r-1} \oplus \left( \bigoplus_{n=1}^{3\frac{r-1}{2}-m} \ubar{P}_{m,r+2n-1} \oplus \ubar{P}'_{m,r+2n-1} \right) \oplus \bigoplus_{n=r}^m \ubar{P}_{m,2(r-m+n-1)}
 \]
 and a non-projective submodule 
 \[
  \bigoplus_{n=r}^m \ubar{X}^{\down}_{m,2(m-n)} \oplus \ubar{X}^+_{m,r-2(m-n+1)} \oplus \ubar{X}^-_{m,r-2(m-n+1)} \oplus \ubar{X}^{\up}_{m,2(m-n)},
 \]
 where $\ubar{X}_{m,r-1}$ and $\ubar{X}'_{m,r-1}$ are isomorphic to $X_{r-1}$, $\ubar{P}_{m,r+2n-1}$ and $\ubar{P}'_{m,r+2n-1}$ are isomorphic to $P_{r+2n-1}$, $\ubar{P}_{m,2(r-m+n-1)}$ is isomorphic to $P_{2(r-m+n-1)}$, $\ubar{X}_{m,2(m-n)}^{\down}$ and $\ubar{X}_{m,2(m-n)}^{\up}$ are isomorphic to $X_{2(m-n)}$, and $\ubar{X}_{m,r-2(m-n+1)}^+$ and $\ubar{X}_{m,r-2(m-n+1)}^-$ are isomorphic to $X_{r-2(m-n+1)}$. Each of these submodules is generated by some vector in $\ad_m^0$, with the exception of $\ubar{X}_{m,r-2(m-n+1)}^+$ and $\ubar{X}_{m,r-2(m-n+1)}^-$ for every integer $r \leqs n \leqs m$. Remark however that these $\bar{U}$-modules admit no non-trivial morphism to $P_{2r-2}$, which is the projective cover of $X_0 = \C$.
\end{proof}

\section{Beads}\label{S:beads}

In this section we set up the technology for the comparison between the non-semisimple invariant $Z_r$ and the renormalized Hennings invariant associated with $\bar{U}$. In order to establish their equivalence, we first need to relate the Temperley--Lieb category $\TL$ to the category $\mods{\bar{U}}$ of finite-dimensional left $\bar{U}$-modules.

\subsection{Ribbon structures}\label{S:ribbon_structures}

Let us consider the monoidal linear functor 
\[
 F_\TL : \TL \to \mods{\bar{U}}
\]
which sends the monoidal generator $1 \in \TL$ to the fundamental representation $X \in \mods{\bar{U}}$ defined in Section~\ref{S:simple_indec_proj}, the evaluation $\ev_1 \in \TL(2,0)$ to the morphism $e \in \Hom_{\bar{U}}(X \otimes X,\C)$ defined by
\begin{equation}\label{E:ev_TL}
 e(a_0^1 \otimes a_0^1) := 0, \quad 
 e(a_0^1 \otimes a_1^1) := -1, \quad
 e(a_1^1 \otimes a_0^1) := q^{-1}, \quad 
 e(a_1^1 \otimes a_1^1) := 0,
\end{equation}
and the coevaluation $\coev_1 \in \TL(0,2)$ to the morphism $c \in \Hom_{\bar{U}}(\C,X \otimes X)$ defined by
\begin{equation}\label{E:coev_TL}
 c(1) := q a_0^1 \otimes a_1^1 - a_1^1 \otimes a_0^1.
\end{equation}

As a consequence of the skein relation~\eqref{E:S1}, the functor $F_\TL$ is braided. This means it sends the braiding $c_{1,1} \in \TL(2,2)$ to the morphism
\begin{equation}\label{E:c_TL}
 q^{\frac{r+1}{2}} \id_{X \otimes X} + q^{\frac{r-1}{2}} c \circ e \in \End_{\bar{U}}(X \otimes X),
\end{equation}
which coincides with the braiding $c_{X,X} \in \End_{\bar{U}}(X \otimes X)$ determined by the R-matrix $R \in \bar{U} \otimes \bar{U}$.

However, $F_\TL$ does not behave well with respect to the other structure morphisms of the ribbon categories $\TL$ and $\mods{\bar{U}}$. Indeed, for every $u \in \TL(m)$, it sends the partial trace $\ptr_1(u) \in \TL(m-1)$ to the morphism
\begin{equation}\label{E:ptr_TL}
 (\id_{X^{\otimes m-1}} \otimes e) \circ (F_{\TL}(u) \otimes \id_X) \circ (\id_{X^{\otimes m-1}} \otimes c) \in \End_{\bar{U}}(X^{\otimes m-1}),
\end{equation}
which is obtained from the partial trace $\ptr_X(F_\TL(u)) \in \End_{\bar{U}}(X^{\otimes m-1})$ determined by the pivotal element $K \in \bar{U}$ by a change of sign. In particular, $F_\TL$ sends the twist $\theta_1 \in \TL(1,1)$ to the morphism
\begin{equation}\label{E:theta_TL}
 - q^{\frac{r+3}{2}} \id_X \in \End_{\bar{U}}(X),
\end{equation}
which coincides with the twist $\theta_X \in \mods{\bar{U}}$ determined by the inverse ribbon element $v_- \in \bar{U}$ only up to the sign.

\begin{remark}\label{R:sign_discrepancy}
 More generally, for all $m,m' \in \TL$ we have
 \begin{align}\label{E:c_TL_general}
  F_\TL(c_{m,m'}) &= c_{F_\TL(m),F_\TL(m')},
 \end{align}
 but for every $u \in \TL(m)$ and $0 \leqs k \leqs m$ we have 
 \begin{align}\label{E:ptr_TL_general}
  F_\TL(\ptr_k(u)) &= (-1)^k \ptr_{F_\TL(k)} (F_\TL(u)).
 \end{align}
 In particular, this implies
 \begin{align}\label{E:theta_TL_general}
  F_\TL(\theta_m) &= (-1)^m \theta_{F_\TL(m)}.
 \end{align}
 This sign discrepancy in the comparison between the ribbon structure of $\TL$ and that of $\mods{\bar{U}}$ is the reason why we need to use the ribbon Kauffman bracket of equation~\eqref{E:ribbon_bracket_top_tangle}, instead of the standard one, in the definition of the admissible bichrome link invariant $F'_\Omega$ of Proposition~\ref{P:ad_bic_link_inv}.
\end{remark}

Although it does not preserve ribbon structures, the functor $F_\TL$ is faithful. Indeed, this can be shown by considering Lusztig's divided power quantum group of $\fsl_2$, which is denoted $U$, and which contains $\bar{U}$ as a Hopf subalgebra. Then, the functor $F_\TL$ can be written as the composition of two faithful functors: the first one is the equivalence from $\TL$ to the full monoidal subcategory of $\mods{U}$ generated by the fundamental representation of $U$, while the second one is the restriction functor from $\mods{U}$ to $\mods{\bar{U}}$. By abuse of notation, we still denote by $F_\TL : \TL \to \Vect_\C$ the composition of $F_\TL$ with the forgetful functor from $\mods{\bar{U}}$ to $\Vect_\C$.

\subsection{Bead category}\label{SS:bead_category}

Let us now revisit the HKR algorithm, which we will adapt to our purposes. This requires a few preliminary definitions. First of all, let us denote by $\calB$ the ribbon linear category obtained from $\TL$ by allowing strands of morphisms to carry beads labeled by elements of $\bar{U}$. Remark that there exists a unique monoidal linear functor $\calF : \calB \to \Vect_\C$ extending $F_\TL$ and sending every $x$-labeled bead to the linear endomorphism of $X$ determined by the action of $x$. The \textit{bead category} $\beadC$ is the ribbon linear category defined as the quotient of $\calB$ with respect to the kernel of $\calF$, and we denote by 
\[
 F_{\bar{U}} : \beadC \to \Vect_\C
\]
the faithful monoidal linear functor induced by $\calF$ on the quotient. In the bead category $\beadC$ we have
\begin{equation}\label{E:bead_notation_1}
 \pic{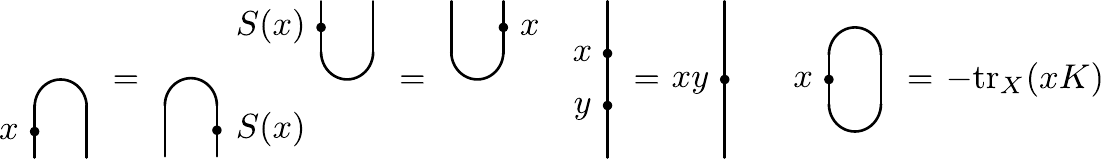}
\end{equation}
When $m$ parallel strands are represented graphically by a single strand with label $m$, we adopt the convention
\begin{equation}\label{E:bead_notation_2}
 \pic{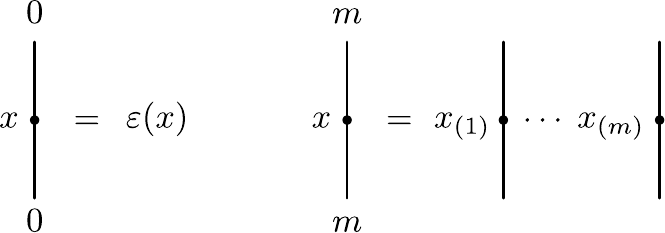}
\end{equation}
Remark that
\begin{equation}\label{E:bead_notation_3}
 \pic{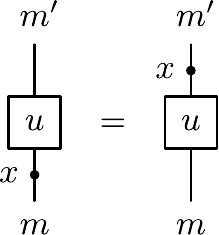}
\end{equation}
for every $x \in \bar{U}$ and every $u \in \TL(m,m')$. Furthermore, as a consequence of \cite[Lemma~4.1]{BDM19}, we also have
\begin{align}
 \pic{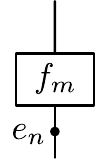} &= \delta_{m,n} \cdot \pic{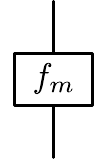} &
 \pic{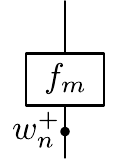} &= \pic{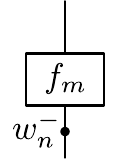} = 0 \label{E:bead_f_m}
\end{align}
for every integer $0 \leqs m \leqs r-1$, and
\begin{align}
 \pic{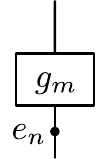} &= (\delta_{m,n} + \delta_{3r-m-2,n}) \cdot \pic{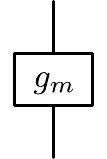} \label{E:bead_g_m_e} \\*
 \pic{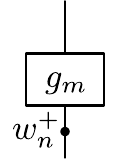} &= - [m+1] \delta_{m,n} \cdot \pic{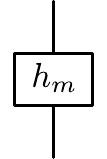}  \label{E:g_m_w_+} \\*
 \pic{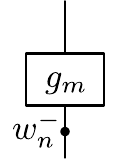} &= - [m+1] \delta_{3r-m-2,n} \cdot \pic{int_lemma_h_m} \label{E:bead_g_m_w_-}
\end{align}
for every integer $r \leqs m \leqs 2r-2$.

Let now $p \in \TL(m)$ be an idempotent and $T \in \calT_k(p)$ be a bichrome $k$-top tangle. If $\lb_{n_1,\ldots,n_k}(T)$ denotes the $(n_1,\ldots,n_k)$-labeling of $T$ introduced in equation~\eqref{E:labeling_top_tangle} for integers $r-1 \leqs n_1,\ldots,n_k \leqs 2r-2$, then let us define a morphism 
\[
 B_{n_1,\ldots,n_k}(T) \in \beadC(p,n_1 \otimes n_1 \otimes \ldots \otimes n_k \otimes n_k \otimes p),
\]
called the \textit{bead presentation of $\lb_{n_1,\ldots,n_k}(T)$}, which will be obtained through the following
version of the HKR algorithm based on \textit{singular diagrams} \cite[Section~3.3]{K94}, also known as \textit{flat diagrams} \cite[Section~4.8]{V03}. A singular diagram of a framed tangle is obtained from a regular diagram by discarding framings and forgetting the difference between overcrossings and undercrossings. On the set of singular diagrams, we consider the equivalence relation generated by all singular versions of the usual local moves corresponding to ambient isotopies of tangles, except for the first Reidemeister move. In particular, two equivalent singular diagrams represent homotopic tangles, but not all homotopies are allowed. Then, let us explain how to define the bead presentation of $\lb_{n_1,\ldots,n_k}(T)$ (the reader is also invited to check the references above, or indeed any of those listed in Section~\ref{S:strategy_proof}, where more details about the HKR algorithm can be found). We start from a regular diagram of $\lb_{n_1,\ldots,n_k}(T)$, and we pass to its singular version while also inserting beads labeled by components of the R-matrix around crossings as shown:
\begin{align*}
 \pic{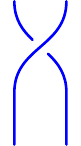} &\mapsto \pic{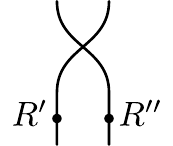} &
 \pic{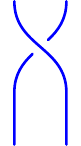} &\mapsto \pic{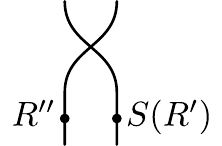}
\end{align*}
Next, we need to collect all beads sitting on the same strand in one place, which has to be next to an upward oriented endpoint, for components which are not closed. As we slide beads past maxima, minima, and crossings, we change their labels according to the rule
\begin{align*}
 \pic{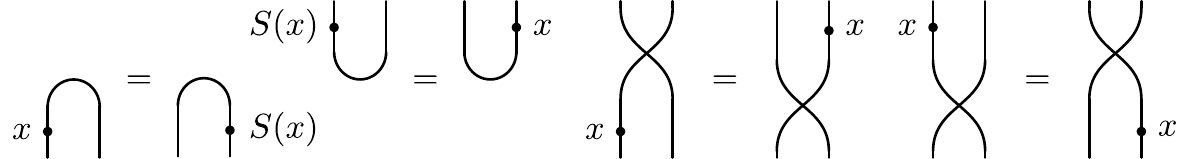}
\end{align*}
Next, we pass from our singular diagram to an equivalent one whose singular crossings all belong to singular versions of twist morphisms, and we replace them with beads labeled by pivotal elements according to the rule
\begin{align*}
 \pic{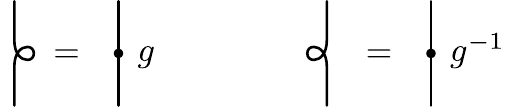}
\end{align*}
This is indeed possible, because we started from a bichrome top tangle featuring a unique blue incoming boundary point and a unique outgoing one. Finally, we collect all remaining beads, changing their labels along the way as before, and we multiply everything together according to the rule
\begin{align*}
 \pic{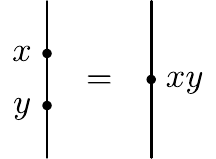}
\end{align*}
In the end, we are left with a planar tangle carrying at most a single bead on each of its components. Inserting idempotents of $\TL$ as shown in equation~\eqref{E:cabling_functor} gives
\begin{equation}\label{E:bead_presentation}
 B_{n_1,\ldots,n_k} \left( \pic{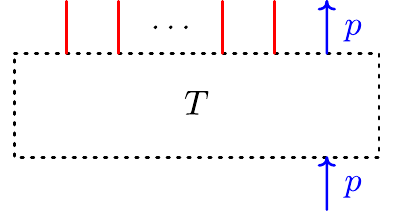} \right) := \hspace*{-5pt} \pic{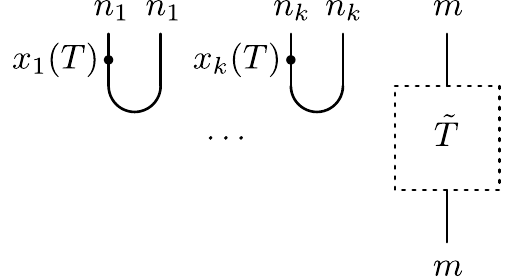}
\end{equation}
for some $x_1(T), \ldots, x_k(T) \in \bar{U}$ and some $\tilde{T} \in \TL_{\bar{U}}(p)$. This defines the bead presentation of the $(n_1,\ldots,n_k)$-labeling $\lb_{n_1,\ldots,n_k}(T)$ of $T$.

\begin{proposition}\label{P:beads}
 The ribbon Kauffman bracket of equation~\eqref{E:ribbon_bracket_top_tangle} yields
 \begin{equation}\label{E:equiv_beads_Kauffman}
  F_{\bar{U}} ( B_{n_1,\ldots,n_k} (T) ) 
  = F_\TL ( \langle T \rangle_{n_1,\ldots,n_k}^\rb )
 \end{equation}
 for all $T \in \calT_k(p)$ and $r-1 \leqs n_1,\ldots,n_k \leqs 2r-2$.
\end{proposition}

Remark that equation~\eqref{E:equiv_beads_Kauffman} uses the abusive notation $F_\TL : \TL \to \Vect_\C$ introduced at the end of Section~\ref{S:ribbon_structures}. In other words, it should be read as an equality between linear maps, where we are omitting the forgetful functor from $\mods{\bar{U}}$ to $\Vect_\C$. The proof of Proposition~\ref{P:beads} follows directly from the construction, which is due to Hennings, Kauffman, and Radford, and will not be given here. We only stress once again the fact that the use of the ribbon version of the Kauffman bracket of equation~\eqref{E:ribbon_bracket_top_tangle} is required by Remark~\ref{R:sign_discrepancy}. Indeed, the ribbon number of equation~\eqref{E:ribbon_number_top_tangle} measures precisely the sign difference between $F_{\bar{U}}(B_{n_1,\ldots,n_k}(T))$ and $F_\TL(\langle T \rangle_{n_1,\ldots,n_k})$, as it counts the total number of times (weighted by the label) partial traces and twist morphisms appear in $\lb_{n_1,\ldots,n_k}(T)$.

\subsection{Diagrammatic integral}\label{SS:diagrammatic_integral}

Let us introduce a key definition for the diagrammatic translation of Hennings' construction.

\begin{definition}\label{D:diagr_int}
 A \textit{diagrammatic integral} $\ell$ of $\TL$ is a family of morphisms 
 \begin{align*}
  \ell_{r-1} &\in \TL(f_{r-1} \otimes f_{r-1},g_{2r-2}), &
  \ell_m &\in \TL(g_m \otimes g_m^*,g_{2r-2})
 \end{align*}
 with $r \leqs m \leqs 2r-2$ satisfying
 \begin{align}
  \pic{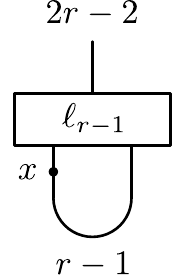} &= \lambda(x e_{r-1}) \cdot \pic{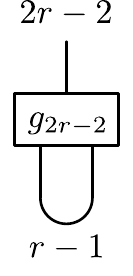} \label{E:int_def_r-1} \\*
  \pic{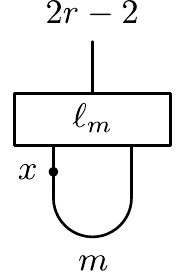} + \pic{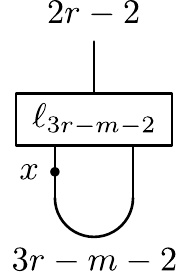} &= \lambda(x e_m) \cdot \pic{int_def_2r-2_a} \label{E:int_def_m}
 \end{align}
 for every $x \in \bar{U}$.
\end{definition}

We point out that, by definition, a diagrammatic integral $\ell$ satisfies
\begin{align*}
 \pic{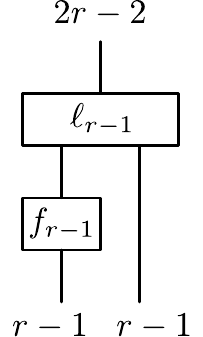} &= \pic{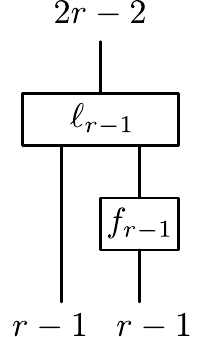} = \pic{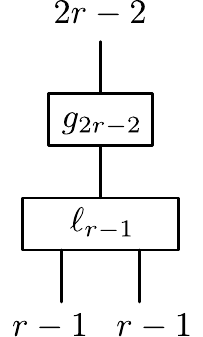} = \pic{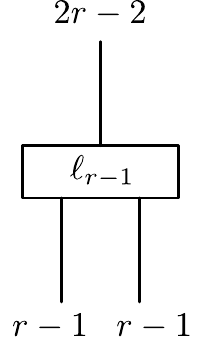} \\*
 \pic{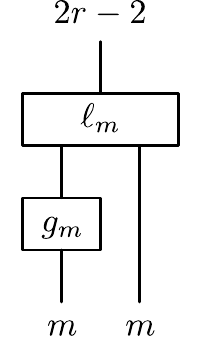} &= \pic{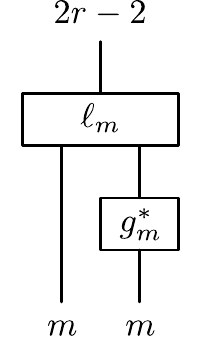} = \pic{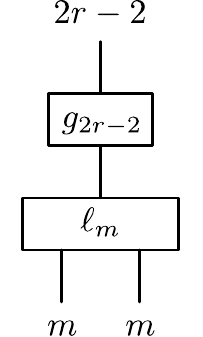} = \pic{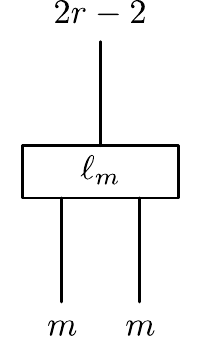}
\end{align*}
for all integers $r \leqs m \leqs 2r-2$.

Now, despite the fact that Definition~\ref{D:diagr_int} determines a system of $\frac{r-1}{2}$ equations for each element of the quantum group $\bar{U}$, which is a vector space of dimension $r^3$, the actual number of conditions we need to verify in order to check whether a family $\ell$ of morphisms of $\TL$ provides a diagrammatic integral or not can be drastically reduced. Indeed, this will follow essentially from Lemma~\ref{L:ad_m}. In order to explain how, let us start with a quick remark.

\begin{remark}\label{R:ad}
 The linear map sending every $x \in \bar{U}$ to 
 \[
  F_{\bar{U}} \left( \pic{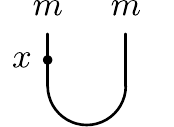} \right) (1) \in X^{\otimes m} \otimes X^{\otimes m} 
 \] 
 defines a $\bar{U}$-module morphism $j_m : \ad \to X^{\otimes m} \otimes X^{\otimes m}$ for every $m \in \N$. Indeed,
 \[
  j_m(\ad_x(y)) = x \cdot j_m(y)
 \]
 follows from equation~\eqref{E:bead_notation_1} for every $x \in \bar{U}$ and $y \in \ad$.
\end{remark}

As we will show now, it is actually sufficient to restrict ourselves to beads labeled by $K^a \in \bar{U}$ with $a \in \Z$ in the range $0 \leqs a \leqs r-1$. Therefore, from now on, for every integer $a \in \Z$ we adopt the shorthand notation
\begin{equation*}
 \pic{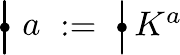}
\end{equation*}
Furthermore, let us set
\begin{align*}
 [k]_a &:= 
 \begin{cases}
  \displaystyle \frac{[ak]}{[a]} & \mbox{if} \quad a \not\equiv 0 \mod r, \\
  \displaystyle k & \mbox{if} \quad a \equiv 0 \mod r,
 \end{cases} &
 \{ k \}'_a &:= \{ ak \}'
\end{align*}
for all integers $a,k \in \Z$. Remark that $[k]_a$ and $\{ k \}'_a$ are obtained from $[k]$ and $\{ k \}'$ by a change of variable replacing $q$ with $q^a$.

\begin{lemma}
 \begin{samepage}
 A family $\ell$ of morphisms
 \begin{align*}
  \ell_{r-1} &\in \TL(f_{r-1} \otimes f_{r-1},g_{2r-2}), &
  \ell_m &\in \TL(g_m \otimes g_m^*,g_{2r-2})
 \end{align*} 
 with $r \leqs m \leqs 2r-2$ is a diagrammatic integral of $\TL$ if and only if
 \begin{align}
  \pic{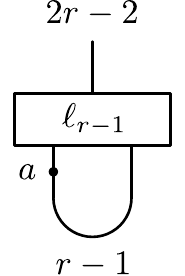} &= [r]_{a+1} \cdot \pic{int_def_2r-2_a} \label{E:int_lemma_r-1} \\*
  \pic{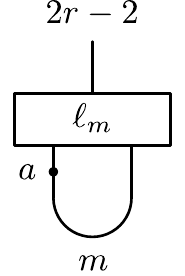} + \pic{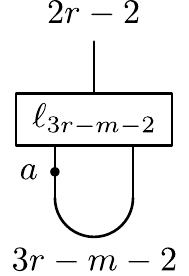} &= [r]_{a+1} \{ m+1 \}' \cdot \pic{int_def_2r-2_a} \label{E:int_lemma_m} \\*
  \pic{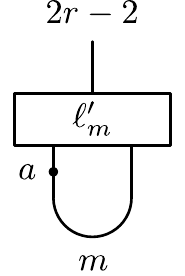} - \pic{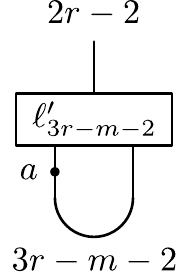} &= - [r]_{a+1} [m+1] \cdot \pic{int_def_2r-2_a} \label{E:int_lemma_m_prime}
 \end{align}
 for every $a \in \Z$, where 
 \[
  \ell'_m := \ell_m (h_m \otimes g_m^*) \in \TL(g_m \otimes g_m^*,g_{2r-2}).
 \]
 \end{samepage}
\end{lemma}

\begin{proof}
 Thanks to Remark~\ref{R:ad}, the left-hand sides of equations~\eqref{E:int_def_r-1} and \eqref{E:int_def_m} determine $\bar{U}$-module morphisms from $\ad_{r-1}$ to $P_{2r-2}$ and from $\ad_m$ to $P_{2r-2}$ respectively. Thanks to Lemma~\ref{L:ad_m}, every morphism of this type is uniquely determined by its restriction to $\ad_{r-1}^0$ and $\ad_m^0$ respectively. Then, we simply need to check that equations~\eqref{E:int_lemma_r-1}, \eqref{E:int_lemma_m}, and \eqref{E:int_lemma_m_prime} are equivalent to equations~\eqref{E:int_def_r-1} and \eqref{E:int_def_m} for $x = K^a$. On one hand, thanks to equations~\eqref{E:bead_g_m_e}-\eqref{E:bead_g_m_w_-}, we have
 \[
  \pic{int_lemma_h_m} = -\frac{1}{[m+1]} \cdot \hspace*{-5pt} \pic{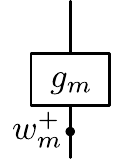} 
  = -\frac{1}{[m+1]} \cdot \hspace*{-5pt} \pic{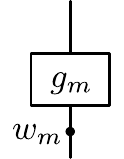}
  = -\frac{1}{[m+1]} \cdot \hspace*{-10pt} \pic{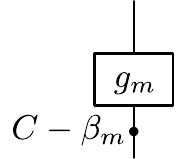}
 \]
 This implies
 \[
  \pic{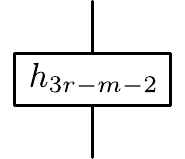} = -\frac{1}{[3r-m-1]} \cdot \hspace*{-10pt} \pic{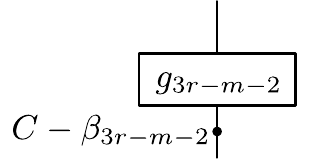}
  = \frac{1}{[m+1]} \cdot \hspace*{-10pt} \pic{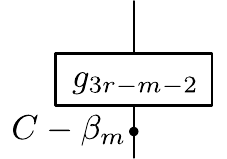}
 \]
 On the other hand, thanks to Lemma~\ref{L:dec_int_r-1} we have
 \begin{align*}
  \lambda(K^a e_{r-1}) &= \tau_{r-1}(K^{a+1}) = [r]_{a+1},
 \end{align*}
 and thanks to Lemma~\ref{L:dec_int_m} for every integer $r \leqs m \leqs 2r-2$ we have
 \begin{align*}
  \lambda(K^a e_m) &= \{ m+1 \}' \left( \tau_m(K^{a+1}) + \tau_{3r-m-2}(K^{a+1}) \right) \\*
  &= \{ m+1 \}' \left( [2r-m-1]_{a+1} + [m-r+1]_{a+1} \right) \\*
  &= [r]_{a+1} \{ m+1 \}', \\*
  \lambda(K^a w_m) &= [m+1]^2 \left( \tau'_m(K^{a+1}(C-\beta_m)) + \tau'_{3r-m-2}(K^{a+1}(C-\beta_m)) \right) \\*
  &= [m+1]^2 \left( [2r-m-1]_{a+1} + [m-r+1]_{a+1} \right) \\*
  &= [r]_{a+1} [m+1]^2. \qedhere
 \end{align*}
\end{proof}

\begin{proposition}\label{P:diagramatic_integral}
 The non-semisimple Kirby color of Definition~\ref{D:ns_Kirby_color} is a diagrammatic integral of $\TL$.
\end{proposition}

\begin{proof}
 Equation~\eqref{E:int_lemma_r-1} follows from equation~\eqref{E:t_f_m}, equation~\eqref{E:int_lemma_m} follows from equations~\eqref{E:t_g_m} and \eqref{E:int_g_g_m_k_low} with $k=0$, and equation~\eqref{E:int_lemma_m_prime} follows from equations~\eqref{E:t_h_m} and \eqref{E:int_g_h_m_k_low} with $k=0$.
\end{proof}

\section{Proofs}\label{S:proof}

Let us prove all the results we claimed in Sections~\ref{S:TL} and \ref{S:3-manifold}.

\subsection{Proof of results from Section~\ref{S:m-trace}}\label{S:proof_S:m-trace}

\begin{proof}[Proof of Proposition~\ref{P:m-trace_def}]
 As we recalled at the beginning of Section~\ref{S:beads}, if $U$ denotes Lusztig's divided power quantum group of $\fsl_2$, then $\TL$ is equivalent, as a braided monoidal category, to the full monoidal subcategory of $\mods{U}$ generated by the fundamental representation. Under this equivalence, $f_{r-1}$ is sent to the Steinberg module, which is projective \cite[Theorem~9.8]{APW91}. This means $f_{r-1}$ is projective too, and thus it generates $\Idf$ \cite[Lemma~4.4.1]{GKP10}. The rest of the statement follows from \cite[Theorem~5.5, Corollary~5.6]{GKP18}.
\end{proof}

\begin{remark}\label{R:mod_tr_TL}
 It was already observed in Section~\ref{S:ribbon_structures} that the ribbon structure of $\TL$ does not agree with the one of $\mods{\bar{U}}$, see for instance Remark~\ref{R:sign_discrepancy}. This implies in particular that, for every idempotent $p \in \TL(m)$ of $\Idf$ and every endomorphism $u \in \TL(p)$, we have
 \begin{align}\label{E:mod_tr_TL}
  \rmt^\TL_p(u) &= (-1)^m \rmt^{\bar{U}}_{F_\TL(p)} (F_\TL(u)).
 \end{align}
 This sign discrepancy in the comparison between the modified trace of $\TL$ and the one of $\mods{\bar{U}}$ is the reason behind the sign in equation~\eqref{E:ad_bic_link_inv}.
\end{remark}

\subsection{Proof of results from Section~\ref{S:bichrome_link_invariant}}\label{S:proof_bichrome_link_invariant}

\begin{proof}[Proof of Proposition~\ref{P:ad_bic_link_inv}]
  First of all, we need to show that a cutting presentation of $T$ exists.  In order to construct one, let us orient red components of $T$, let us consider disjoint paths $\gamma_i$ for every $1 \leqs i \leqs k$, each joining a basepoint $p_i$ on the $i$th red component $T_i$ to a basepoint $q_i$ on the top line $I \times \{ \frac 12 \} \times \{ 1 \} \subset I^3$, and let us choose a projective blue component of $T$, meaning a blue component labeled by an idempotent $p \in \TL(m)$ of $\Idf$. Let us cut open the projective blue component and all red ones following the specified paths and orientations, thus obtaining the bichrome $k$-top tangle $T'$ represented in Figure~\ref{F:cutting_presentation}. By construction, $T'$ is a cutting presentation of $T$.
  
 \begin{figure}[h]
  \centering
  \includegraphics{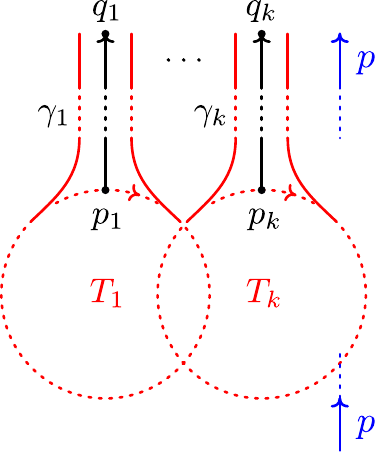}
  \caption{Cutting presentation $T'$ of $T$.}
 \label{F:cutting_presentation}
 \end{figure}
  
 We need to show $F_\nsKirby(T)$ does not depend on the choice of the cutting presentation of $T$. Thanks to equation~\eqref{E:mod_tr_TL}, we have 
 \[
  \rmt^\TL_p(F_{\nsKirby,p}(T')) = (-1)^m \rmt^{\bar{U}}_{F_\TL(p)}(F_\TL(F_{\nsKirby,p}(T'))).
 \]
 The advantage of looking at $F_\TL(F_{\nsKirby,p}(T'))$ rather than $F_{\nsKirby,p}(T')$ is that the former can be computed using a different approach. In order to do this, it will be convenient to fix some additional notation, so let us set
 \[
  \pic{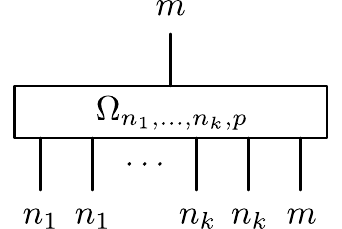} := \pic[22.5]{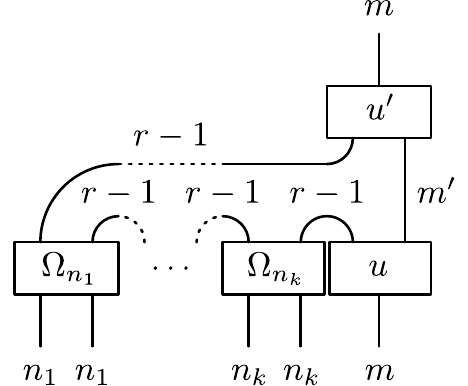}
 \] 
 This means morphisms $\nsKirby_{n_1,\ldots,n_k,p} \in \TL(n_1 \otimes n_1 \otimes \ldots n_k \otimes n_k \otimes m,m)$ satisfy
 \[
  F_{\nsKirby,p}(T') = \sum_{n_1,\ldots,n_k=r-1}^{2r-2} \nsKirby_{n_1,\ldots,n_k,p} \langle T' \rangle_{n_1,\ldots,n_k}^\rb.
 \]
 Thanks to Proposition~\ref{P:beads}, we have
 \begin{align*}
  F_\TL(F_{\nsKirby,p}(T')) 
  &= \sum_{n_1,\ldots,n_k=r-1}^{2r-2} F_\TL(\nsKirby_{n_1,\ldots,n_k,p}) \circ F_\TL(\langle T' \rangle_{n_1,\ldots,n_k}^\rb) \\*
  &= \sum_{n_1,\ldots,n_k=r-1}^{2r-2} F_\TL(\nsKirby_{n_1,\ldots,n_k,p}) \circ F_{\bar{U}} ( B_{n_1,\ldots,n_k} (T') ).
 \end{align*}
 Then, thanks to Proposition~\ref{P:diagramatic_integral}, we have
 \begin{align*}
  &\sum_{n_1,\ldots,n_k=r-1}^{2r-2} F_\TL(\nsKirby_{n_1,\ldots,n_k,p}) \circ F_{\bar{U}}(B_{n_1,\ldots,n_k}(T')) = \left( \prod_{i=1}^k \lambda(x_i(T')) \right) F_{\bar{U}}(\tilde{T}'),
 \end{align*}
 where $x_i(T') \in \bar{U}$ and $\tilde{T}' \in \beadC(p)$ are given by equation~\eqref{E:bead_presentation}. Summing up
 \begin{align*}
  F_\nsKirby(T) = \rmt^{\bar{U}}_{F_\TL(p)} \left( \left( \prod_{i=1}^k \lambda(x_i(T')) \right) F_{\bar{U}}(\tilde{T}') \right).
 \end{align*} 
 We will show now that the fact that $F_\nsKirby(T)$ is independent of the cutting presentation $T'$ of $T$ follows essentially from the fact that $\lambda$ is a quantum character, and that $\rmt^{\bar{U}}$ is a modified trace.

 \begin{figure}[b]
  \centering
  \includegraphics{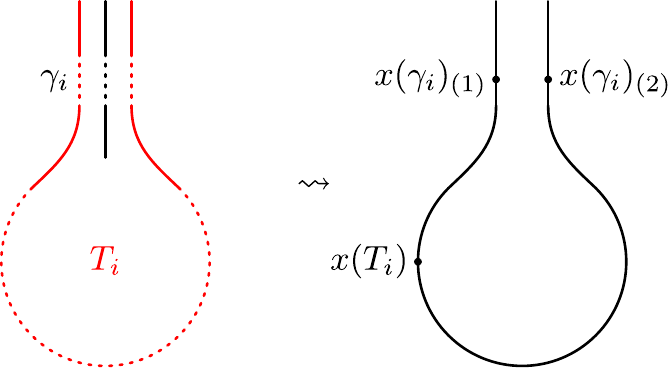}
  \caption{Independence of path.}
 \label{F:independence_path}
 \end{figure}  
 
 First of all, we claim $F_\nsKirby(T)$ does not depend on the choice of the path $\gamma_i$. Indeed, we can decompose $x_i(T')$ as $x(\gamma_i)_{(1)} x(T_i) S(x(\gamma_i)_{(2)})$, where $x(\gamma_i)$ is collected traveling along $\gamma_i$, and $x(T_i)$ is collected traveling along $T_i$, as shown in Figure~\ref{F:independence_path}. This means
 \begin{align*}
  \lambda(x(\gamma_i)_{(1)} x(T_i) S(x(\gamma_i)_{(2)})) &= \lambda(x(T_i) S(x(\gamma_i)_{(2)}) S^2(x(\gamma_i)_{(1)})) \\*
  &= \lambda(x(T_i) S(S(x(\gamma_i)_{(1)})x(\gamma_i)_{(2)})) \\*
  &= \epsilon(x(\gamma_i)) \lambda(x(T_i)),
 \end{align*}
 where the first equality follows from the fact that $\lambda$ is a quantum character. Then, since $x(\gamma_i)$ is a product of copies of components of the R-matrix and copies of the pivotal element, which satisfy
 \begin{align*}
  \epsilon(R') R'' &= \epsilon(R'') R' = 1, & \epsilon(g) &= 1,
 \end{align*}
 the contribution of the framed path $\gamma_i$ is trivial, both for the computation of $\lambda(x_i(T'))$ and for its effect on other components of $T'$.
 
 \begin{figure}[t]
  \centering
  \includegraphics{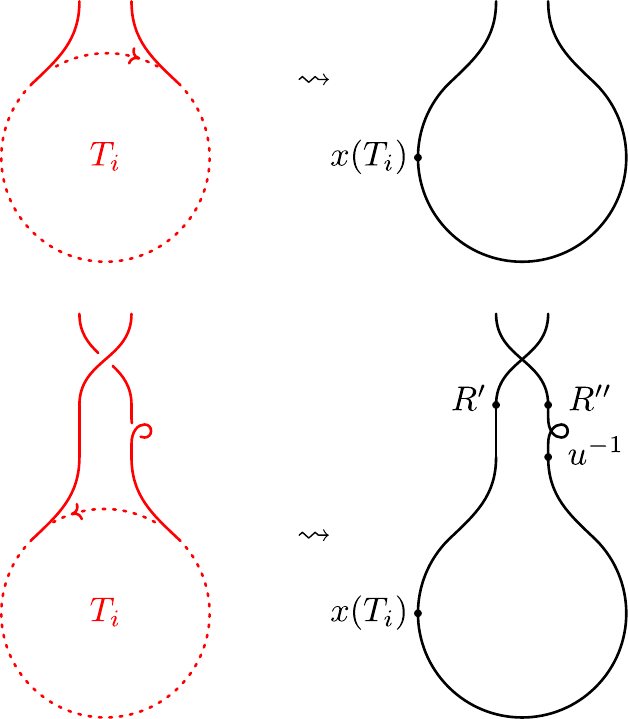}
  \caption{Independence of orientation.}
 \label{F:independence_orientation}
 \end{figure} 
 
 Next, we claim $F_\nsKirby(T)$ does not depend on the choice of the orientation of $T_i$. Indeed, we can switch between the two possible ones by adding a braiding and a twist, as shown in Figure~\ref{F:independence_orientation}. This means
 \begin{align*}
  \lambda(R''u^{-1}S(x(T_i))S(R')) &= \lambda(S^{-1}(R')R''u^{-1}S(x(T_i))) \\*
  &= \lambda(S(u)u^{-1}S(x(T_i))) \\*
  &= \lambda(g^{-2}S(x(T_i))) \\*
  &= \lambda(x(T_i)),
 \end{align*}
 where we are using the identities
 \begin{align*}
  S(R) \otimes S(R'') &= R' \otimes R'', &
  S(u) &= g^{-1}v_+, &
  u^{-1} &= g^{-1}v_-,
 \end{align*}
 as well as \cite[Proposition~4.2]{H96}, which gives, for every $x \in \bar{U}$, the identity
 \[
  \lambda(g^{-2}S(x)) = \lambda(x).
 \]

 \begin{figure}[b]
  \centering
  \includegraphics{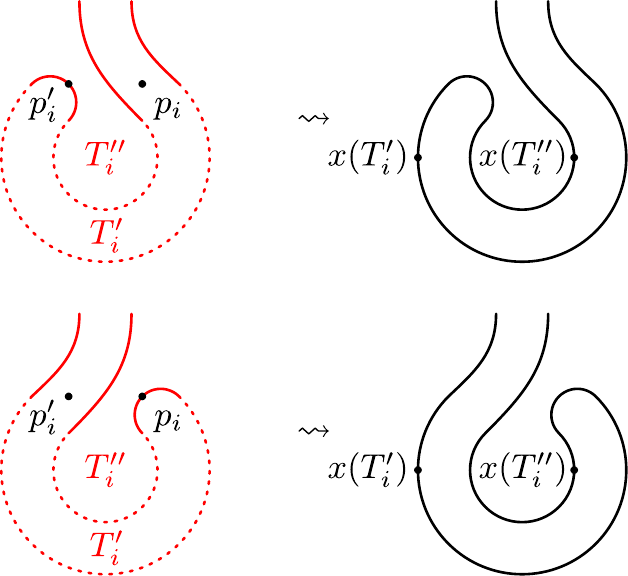}
  \caption{Independence of basepoint.}
 \label{F:independence_basepoint}
 \end{figure} 
 
 Now, we claim $F_\nsKirby(T)$ does not depend on the choice of the basepoint $p_i$. Indeed, if $p_i'$ is another basepoint, we can decompose $x(T_i)$ as $x(T''_i) x(T'_i)$, where $x(T'_i)$ is collected traveling from $p_i$ to $p'_i$, and $x(T''_i)$ is collected traveling from $p'_i$ to $p_i$, as shown in Figure~\ref{F:independence_basepoint}. This means
 \[
  \lambda(x(T''_i)x(T'_i)) = \lambda(x(T'_i)S^2(x(T''_i))).
 \]

 Finally, we claim $F_\nsKirby(T)$ does not depend on the choice of the projective blue component of $T$. Indeed, if
 \[
  \pic{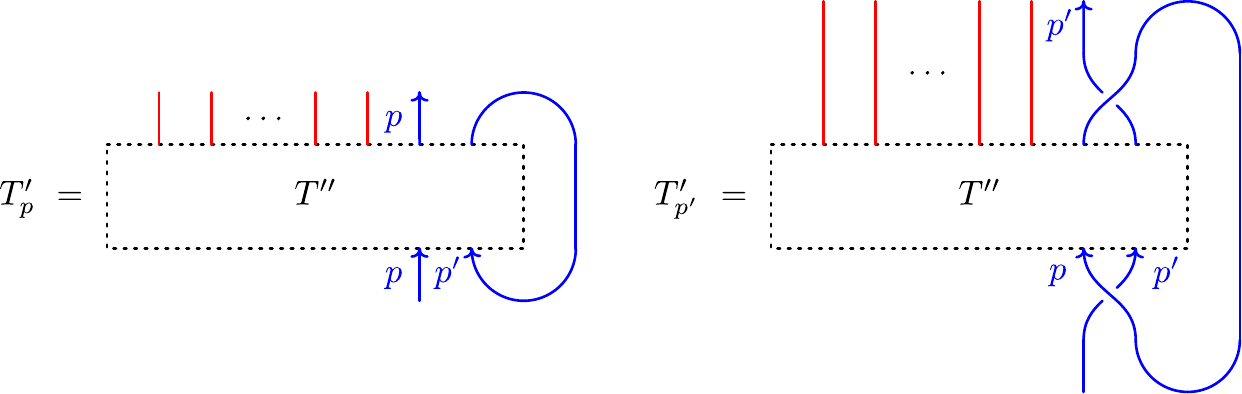}
 \]
 are different cutting presentations of $T$, then
 \begin{align*}
  \rmt^{\bar{U}}_{F_\TL(p)} (F_{\bar{U}}(\tilde{T}'_p))
  &= \rmt^{\bar{U}}_{F_\TL(p)} (\ptr_{F_\TL(p')}(F_{\bar{U}}(\tilde{T}''))) \\*
  &= \rmt^{\bar{U}}_{F_\TL(p \otimes p')} (F_{\bar{U}}(\tilde{T}'')) \\*
  &= \rmt^{\bar{U}}_{F_\TL(p' \otimes p)} (F_\TL(c_{p,p'}) \circ F_{\bar{U}}(\tilde{T}'') \circ F_\TL(c_{p,p'}^{-1})) \\*
  &= \rmt^{\bar{U}}_{F_\TL(p')} (\ptr_{F_\TL(p)}(F_\TL(c_{p,p'}) \circ F_{\bar{U}}(\tilde{T}'') \circ F_\TL(c_{p,p'}^{-1}))) \\*
  &= \rmt^{\bar{U}}_{F_\TL(p)} (F_{\bar{U}}(\tilde{T}'_{p'})),
 \end{align*}
 because $\rmt^{\bar{U}}$ is a modified trace.
\end{proof}

\subsection{Proof of results from Section~\ref{S:3-manifold_invariant}}\label{S:proof_3-manifold_invariant}

\begin{proof}[Proof of Lemma~\ref{L:nstab}]
 It follows from the proof of Proposition~\ref{P:ad_bic_link_inv} and from equation~\eqref{E:ns_stab_par} that
 \begin{align*}
  \Delta_\pm &= \lambda(v_\mp) = i^{\mp \frac{r-1}{2}} r^{\frac 32} q^{\frac{r \mp 3}{2}}. \qedhere
 \end{align*}
\end{proof}

\begin{proof}[Proof of Theorem~\ref{T:3-manifold_inv}]
 
 The proof of the invariance of $Z_r(M,T)$ under Kirby moves follows the same argument of \cite[Proposition~2.13]{DGP17}. Indeed, if the bead collected while traveling along a red component has label $x$, then the operation of sliding a strand, either blue or red, adds a bead with label $R' x_{(1)}$ on the sliding component, and changes the label of the red component to $R''_{(1)} x_{(2)} S(R''_{(2)})$, as shown in Figure~\ref{F:invariance}. This means
 \begin{align*}
  \lambda(R''_{(1)} x_{(2)} S(R''_{(2)})) R' x_{(1)} &= \lambda(x_{(2)} S(R''_{(2)}) S^2(R''_{(1)} )) R' x_{(1)} \\*
  &= \lambda(x_{(2)} S(S(R''_{(1)})R''_{(2)}) ) R' x_{(1)} \\*
  &= \epsilon(R'') \lambda(x_{(2)}) R' x_{(1)} \\*
  &= \lambda(x_{(2)}) x_{(1)} \\*
  &= \lambda(x) 1,
 \end{align*}
 because $\lambda$ is a left integral. Therefore, $Z_r(M,T)$ is invariant under Kirby II moves. 
 
 Furthermore, it follows from Lemma~\ref{L:nstab} that adding an unknotted red component of framing $\pm 1$ contributes a factor of 
 \[
  \lambda(v_{\mp}) = \Delta_\pm.
 \]
 Therefore, $Z_r(M,T)$ is also invariant under Kirby I moves. \qedhere
 
 \begin{figure}[t]
  \centering
  \includegraphics{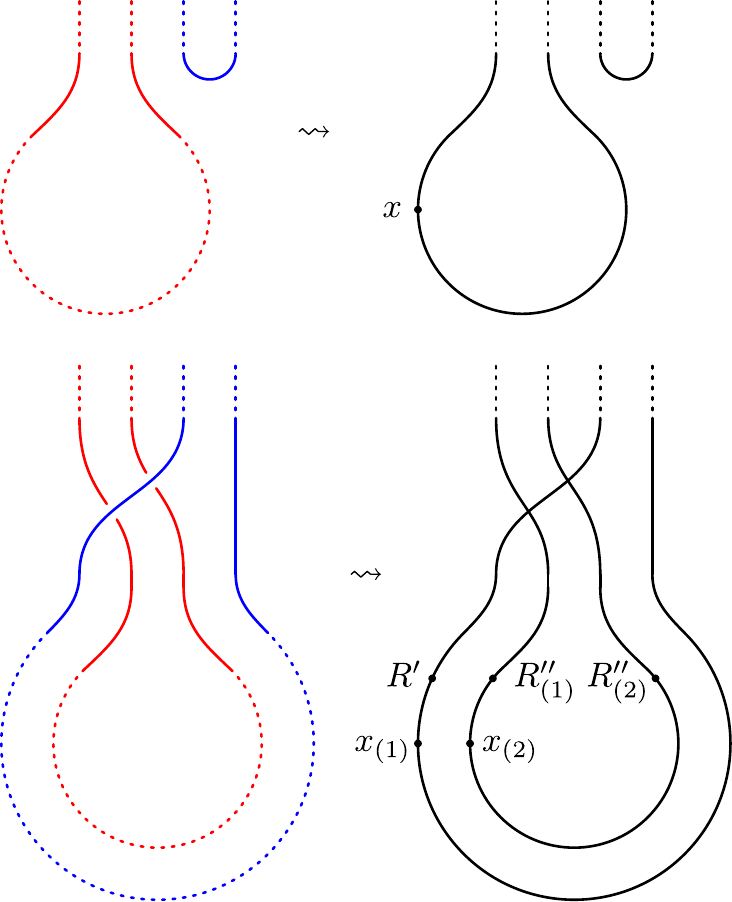}
  \caption{Invariance under Kirby II moves.}
 \label{F:invariance}
 \end{figure}
 
\end{proof}

\FloatBarrier

\begin{proof}[Proof of Lemma~\ref{L:sstab}]
 On one hand, thanks to Lemma~\ref{L:ribbon_Jordan} and equations~\eqref{E:bead_f_m} and \eqref{E:t_f}, we have
 \begin{align*}
  \delta_{+,m} &:= \hspace*{-10pt} \pic{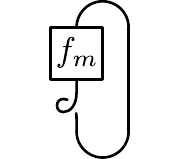} \hspace*{-10pt} 
  = \hspace*{-15pt} \pic{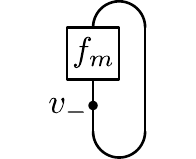} \hspace*{-10pt} = (-1)^m q^{\frac{r+1}{2} m^2 + m} [m+1], \\
  \delta_{-,m} &:= \hspace*{-10pt} \pic{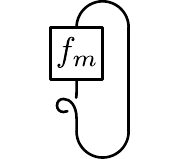} \hspace*{-10pt} 
  = \hspace*{-15pt} \pic{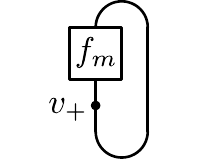} \hspace*{-10pt} = (-1)^m q^{\frac{r-1}{2} m^2 - m} [m+1].
 \end{align*} 
 Remark that for every integer $0 \leqs m \leqs r-2$ we have
 \begin{align*}
  \delta_{\pm,m} &= \delta_{\pm,r-m-2}, &
  (-1)^m [m+1] &= (-1)^{r-m-2 }[r-m-1].
 \end{align*}
 Therefore, we have
 \begin{align*}
  \delta_\pm
  &= \sum_{m=0}^{\frac{r-3}{2}} [2m+1] \delta_{\pm,2m} 
  = \sum_{m=0}^{\frac{r-3}{2}} (-1)^m [m+1] \delta_{\pm,m}
  = \sum_{m=0}^{\frac{r-3}{2}} q^{\frac{r\pm1}{2} m^2 \pm m} [m+1]^2.
 \end{align*}
 On the other hand, equation~\eqref{E:decomposition_integral} can be rearranged as 
 \begin{align*}
  \lambda = \tau_{r-1}(\_ K) 
  &+ \sum_{m=0}^{\frac{r-1}{2}} \{ m+1 \}' \left( \tau_{2r-m-2}(\_ K) + \tau_{m+r}(\_ K) \right) \\*
  &+ \sum_{m=0}^{\frac{r-1}{2}} [m+1]^2 \left( \tau'_{2r-m-2}(\_ K) + \tau'_{m+r}(\_ K) \right),
 \end{align*}
 and Lemma~\ref{L:ribbon_Jordan} can be rewritten as
 \begin{align*}
  v_\mp = q^{\frac{r \mp 1}{2}} e_{r-1} &+ \sum_{m=0}^{\frac{r-1}{2}} q^{\frac{r \pm 1}{2} m^2 \pm m} e_m  \\*
  &\pm \sum_{m=0}^{\frac{r-1}{2}} q^{\frac{r \pm 1}{2} m^2 \pm m} \left( \frac{(m+1) \{ 1 \}}{[m+1]} w_m^+ + \frac{(m-r+1) \{ 1 \}}{[m+1]} w_m^- \right).
 \end{align*}
 Therefore, we have
 \begin{align*}
  \lambda(v_\mp) &= \sum_{m=0}^{\frac{r-1}{2}} q^{\frac{r \pm 1}{2} m^2 \pm m} \{ m+1 \}' \left( [m+1] - [r-m-1] \right) \\*
  &\pm \sum_{m=0}^{\frac{r-1}{2}} q^{\frac{r \pm 1}{2} m^2 \pm m} [m+1]^2 \left( (m+1) \{ 1 \} - (m-r+1) \{ 1 \} \right) \\*
  &= \pm r \{ 1 \} \sum_{m=0}^{\frac{r-3}{2}} q^{\frac{r \pm 1}{2} m^2 \pm m} [m+1]^2.
 \end{align*}
 This means
 \begin{align*}
  \delta_\pm 
  &= \pm \frac{\lambda(v_\mp)}{r \{ 1 \}} 
  = \pm \frac{i^{\mp \frac{r-1}{2}} r^{\frac 12} q^{\frac{r \mp 3}{2}}}{\{ 1 \}}. \qedhere
 \end{align*}
\end{proof}

\begin{proof}[Proof of Proposition~\ref{P:connected_sum}]
 It follows from the construction that
 \[
  Z_r(M \# M',T \cup T') = \HKR_r(M,T) Z_r(M',T'),
 \]
 where $\HKR_r$ is the Hennings invariant associated with $\bar{U}$, compare with \cite[Proposition~2.11]{DGP17} and \cite[Proposition~3.11]{DGGPR19}. Now the result follows directly from \cite[Theorem~1]{CKS07}.
\end{proof}


\section{Computations}\label{S:computations}

In this section we collect formulas which are used in the proof of Proposition~\ref{P:diagramatic_integral}. Everything is based on equations~\eqref{E:bead_notation_1}-\eqref{E:bead_notation_3}.

\subsection{Traces}\label{SS:traces}

Let us start by computing traces of $K^a$ in $\beadC$ for every $a \in \Z$.

\begin{lemma}\label{L:trace_f_m}
 For every integer $1 \leqs m \leqs r-1$ we have
 \begin{equation}\label{E:t_f_m}
  \pic{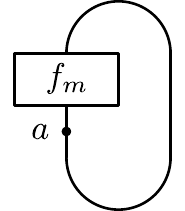} = (-1)^m [m+1]_{a+1}.
 \end{equation}
\end{lemma}

\begin{proof}
 Equation~\eqref{E:t_f_m} is proved by induction on $1 \leqs m \leqs r-1$. If $m = 1$ then, thanks to equation~\eqref{E:bead_notation_1},
 \begin{align*}
  \pic{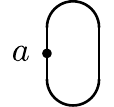} = -[2]_{a+1}.
 \end{align*}
 If $2 \leqs m \leqs r-1$ then by induction hypothesis, thanks to equations~\eqref{E:bead_notation_1}-\eqref{E:bead_notation_3},
 \begin{align*}
  &\pic{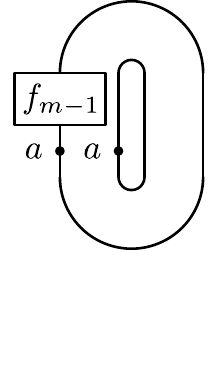} + \frac{[m-1]}{[m]} \cdot \pic{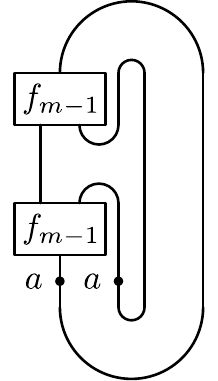} \\*
  &\hspace*{\parindent} = -[2]_{a+1} \cdot \pic{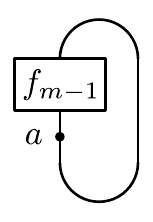} + \frac{[m-1]}{[m]} \cdot \pic{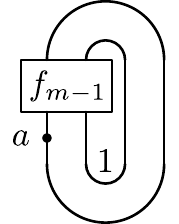} \\*
  &\hspace*{\parindent} = (-1)^m \left( [m]_{a+1} [2]_{a+1} - [m-1]_{a+1} \right) = (-1)^m [m+1]_{a+1}. \qedhere
 \end{align*}
\end{proof}

\begin{lemma}\label{L:trace_g_m_h_m}
 For every integer $r \leqs m \leqs 2r-2$ we have
\begin{align}
 \pic{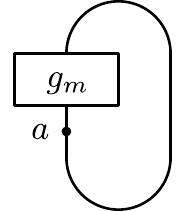} &= (-1)^m [r]_{a+1} \{ m-r+1 \}'_{a+1}, \label{E:t_g_m} \\*
 \pic{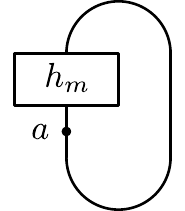} &= 0. \label{E:t_h_m}
\end{align}
\end{lemma}

\begin{proof}
 Equation~\eqref{E:t_h_m} follows from equations~\eqref{E:bead_notation_1}-\eqref{E:bead_notation_3}, which give
 \begin{align*}
  (-1)^{m+1}[m+1] \cdot \pic{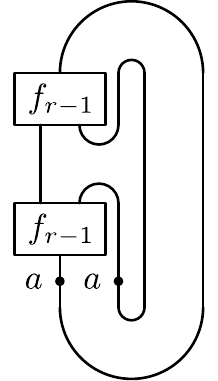} = 0.
 \end{align*}
 Equation~\eqref{E:t_g_m} is proved by induction on $r \leqs m \leqs 2r-2$. If $m = r$ then, thanks to equation~\eqref{E:bead_notation_1},
 \begin{align*}
  \pic{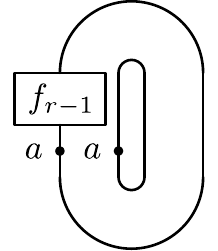} = - [r]_{a+1} \{ 1 \}'_{a+1}.
 \end{align*}
 If $m = r+1$ then, thanks to equations~\eqref{E:bead_notation_1}-\eqref{E:bead_notation_3},
 \begin{align*}
  &\pic{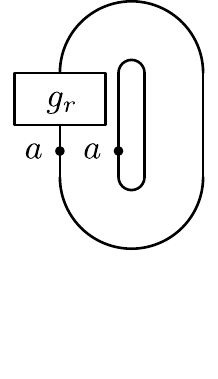} - \pic{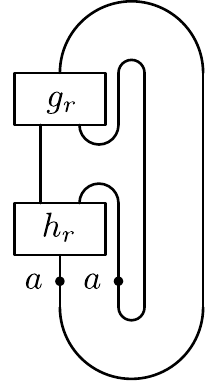} - \pic{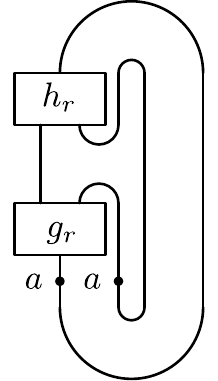} - [2] \cdot \pic{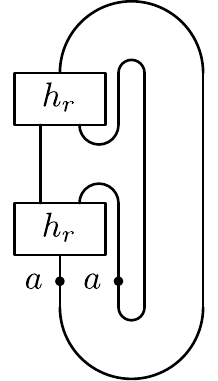} \\*
  &\hspace*{\parindent} = -[2]_{a+1} \cdot \pic{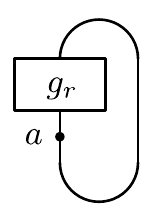} - 2 \cdot \pic{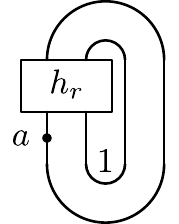} \\*
  &\hspace*{\parindent} = [r]_{a+1} \left( \{ 1 \}'_{a+1} [2]_{a+1} - 2 \right)
  = [r]_{a+1} \{ 2 \}'_{a+1}.
 \end{align*}
 If $r+2 \leqs m \leqs 2r-2$ then by induction hypothesis, thanks to equations~\eqref{E:bead_notation_1}-\eqref{E:bead_notation_3},
 \begin{align*}
  &\pic{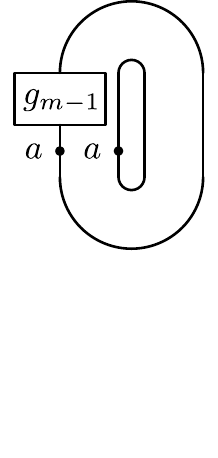} + \frac{[m-1]}{[m]} \cdot \pic{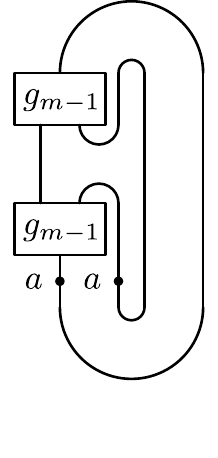} - \frac{2}{[m]^2} \cdot \pic{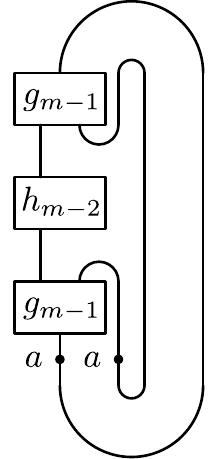} \\*
  &\hspace*{\parindent} = -[2]_{a+1} \cdot \pic{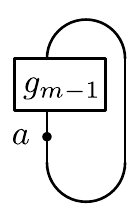} + \frac{[m-1]}{[m]} \cdot \pic{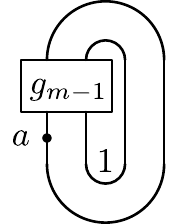} - \frac{2}{[m]^2} \cdot \pic{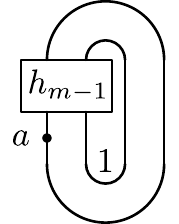} \\*
  &\hspace*{\parindent} = (-1)^m [r]_{a+1} \left( \{ m-r \}'_{a+1} [2]_{a+1} - \{ m-r-1 \}'_{a+1} \right) \\*
  &\hspace*{\parindent} = (-1)^m [r]_{a+1} \{ m-r+1 \}'_{a+1}. \qedhere
 \end{align*}
\end{proof}

\subsection{Partial traces}\label{SS:partial_traces}

Next, let us tackle some harder computations. In order to do this, we will make extensive use of the identity
\begin{equation}\label{E:three_terms}
 [a][b-c]+[b][c-a]+[c][a-b]=0,
\end{equation}
which holds for all integers $a,b,c \in \Z$.

\begin{lemma}\label{L:open_f_m}
 For all integers $1 \leqs m \leqs r-1$ and $0 \leqs k \leqs m$ we have
 \begin{align}
  \pic{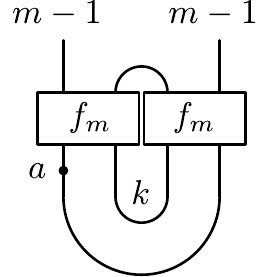} &= \left( - \frac{[m+1]}{[m]} + \frac{[m-k+1][m-k]}{[m]^2} \right) \cdot \pic{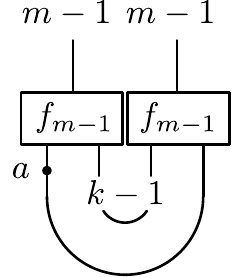} \nonumber \\*
  &\hspace*{\parindent} - \frac{[2]_{a+1} [m-k]^2}{[m]^2} \cdot \pic{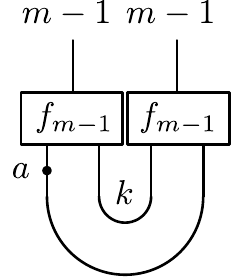} \nonumber \\*
  &\hspace*{\parindent} + \frac{[m-k][m-k-1]}{[m]^2} \cdot \pic{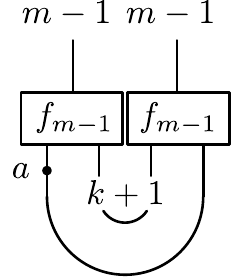} \label{E:open_f_m_k}
 \end{align}
\end{lemma}

\begin{proof}
 Equation~\eqref{E:open_f_m_k} is proved by induction on $1 \leqs m \leqs r-1$. In order to do this, let us set
 \begin{align*}
  \pic{open_p_f_m_k} &= \alpha_{m,k}^{k-1} \cdot \pic{open_f_m-1_k-1} \\*
  &\hspace*{\parindent} + \alpha_{m,k}^k \cdot \pic{open_f_m-1_k} \\*
  &\hspace*{\parindent} + \alpha_{m,k}^{k+1} \cdot \pic{open_f_m-1_k+1}
 \end{align*}
 Remark that we have 
 \[
  \alpha_{m,0}^{-1} = \alpha_{m,m-1}^m = \alpha_{m,m}^m = \alpha_{m,m}^{m+1} = 0
 \]
 for every integer $1 \leqs m \leqs r-1$. If $m = 1$, we have
 \begin{align*}
  \pic{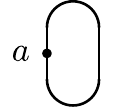} &= -[2]_{a+1}, & \pic{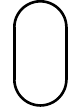} &= -[2].
 \end{align*}
 This gives the condition
 \begin{align*}
  \alpha_{1,0}^0 &= -[2]_{a+1}, & \alpha_{1,1}^0 &= -[2].
 \end{align*}
 If $1 < m \leqs r-1$ and $k = 0$, let us consider
 \[
  \pic{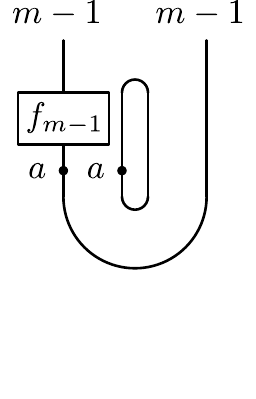} + \frac{[m-1]}{[m]} \cdot \pic{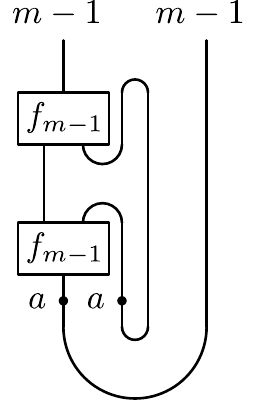}
 \]
 This gives the condition
 \begin{align*}
  \alpha_{m,0}^0 &= -[2]_{a+1}, & \alpha_{m,0}^1 &= \frac{[m-1]}{[m]}.
 \end{align*}
 If $1 < m \leqs r-1$ and $1 \leqs k \leqs m$, let us consider
 \begin{align*}
  &\phantom{+} \pic{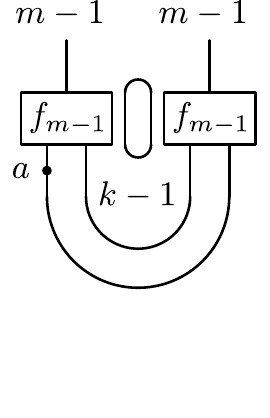} + \frac{[m-1]}{[m]} \cdot \pic{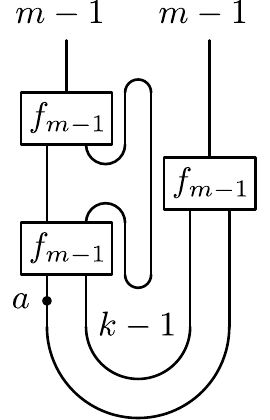} \\*
  &+ \frac{[m-1]}{[m]} \cdot \pic{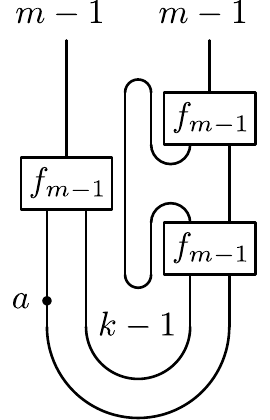} + \frac{[m-1]^2}{[m]^2} \cdot \pic{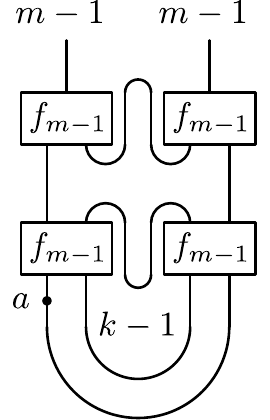}
 \end{align*}
 This gives the condition
 \begin{align*}
  \alpha_{m,k}^{k-1} &= -\frac{\{ m \}'}{[m]} + \frac{[m-1]^2}{[m]^2} \alpha_{m-1,k-1}^{k-2}, \\*
  \alpha_{m,k}^k &= \frac{[m-1]^2}{[m]^2} \alpha_{m-1,k-1}^{k-1}, \\*
  \alpha_{m,k}^{k+1} &= \frac{[m-1]^2}{[m]^2} \alpha_{m-1,k-1}^k.
 \end{align*}
 Thanks to equation~\eqref{E:three_terms}, the solution is 
 \begin{align*}
  \alpha_{m,k}^{k-1} &= - \frac{[m+1]}{[m]} + \frac{[m-k+1][m-k]}{[m]^2}, \\*
  \alpha_{m,k}^k &= - \frac{[2]_{a+1} [m-k]^2}{[m]^2}, \\*
  \alpha_{m,k}^{k+1} &= \frac{[m-k][m-k-1]}{[m]^2}. \qedhere
 \end{align*}
\end{proof}

\begin{lemma}\label{L:open_h_r} 
 We have
 \begin{align}
  \pic{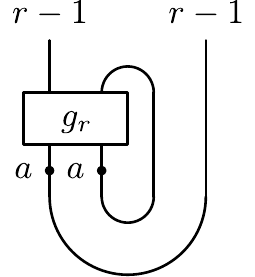} &= -[2]_{a+1} \cdot \pic{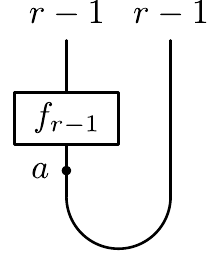} \label{E:open_g_r} \\*
  \pic{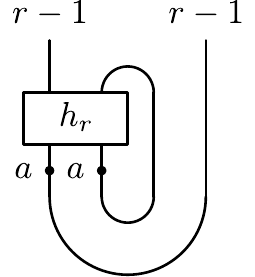} &= \pic{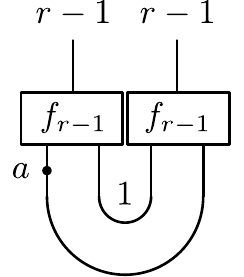} \label{E:open_h_r} 
 \end{align}
 Furthermore, for every integer $1 \leqs k \leqs r$, we have
 \begin{align}
  \pic{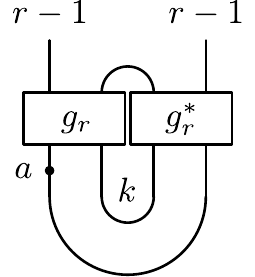} &= -[2] \cdot \pic{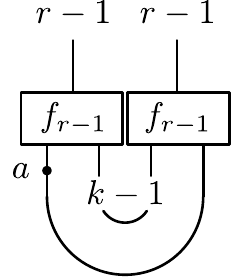} \label{E:open_g_g_r_k} \\
  \pic{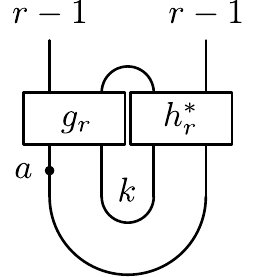} &= \pic{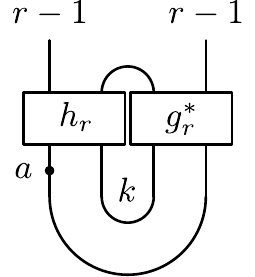} = \pic{open_f_f_r-1_k-1} \label{E:open_g_h_r_k} \\
  \pic{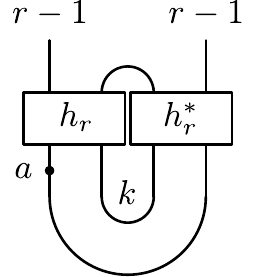} &= [k][k-1] \cdot \pic{open_f_f_r-1_k-1} \nonumber \\*
  &\hspace*{\parindent} -[2]_{a+1}[k]^2 \cdot \pic{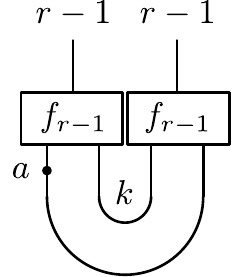} \nonumber \\* 
  &\hspace*{\parindent} + [k+1][k] \cdot \pic{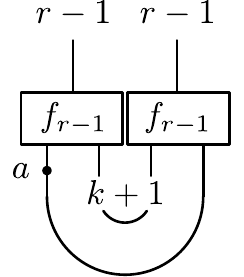} \label{E:open_h_h_r_k}
 \end{align}
\end{lemma}

\begin{proof}
 Equations~\eqref{E:open_g_r}-\eqref{E:open_g_h_r_k} are easy to prove, and left to the reader, while equation~\eqref{E:open_h_h_r_k} follows from equation~\eqref{E:open_f_m_k}.
\end{proof}

\begin{lemma}\label{L:open_g_m}
 For all integers $r+1 \leqs m \leqs 2r-2$, if $0 \leqs k \leqs m-r$ we have
 \begin{align}
  &\pic{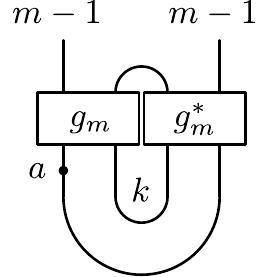} \nonumber \\*
  &\hspace*{0.5\parindent} = \left( - \frac{[m+1]}{[m]} + \frac{[m-k+1][m-k]}{[m]^2} \right) \cdot \pic{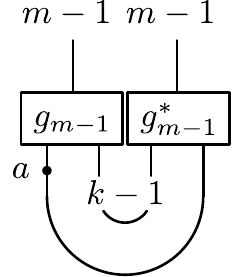} \nonumber \\*
  &\hspace*{\parindent} - \frac{2[k][m-k+1]}{[m]^3} \cdot \left( \pic{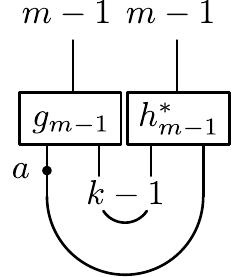} + \pic{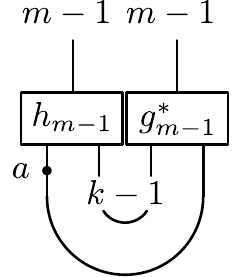} \right) \nonumber \\*
  &\hspace*{\parindent} + \frac{4[k][k-1]}{[m]^4} \cdot \pic{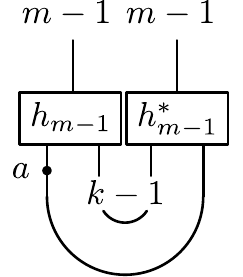} \nonumber \\
  &\hspace*{\parindent} - \frac{[2]_{a+1}[m-k]^2}{[m]^2} \cdot \pic{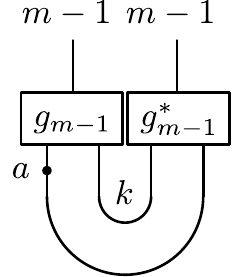} \nonumber \\*
  &\hspace*{\parindent} + \frac{2[2]_{a+1}[k][m-k]}{[m]^3} \cdot \left( \pic{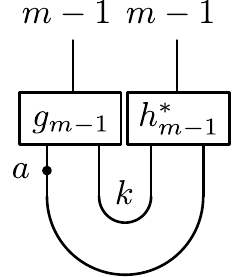} + \pic{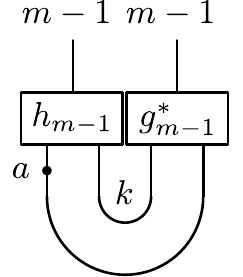} \right) \nonumber \\*
  &\hspace*{\parindent} + \left( \frac{3[2]_{a+1}}{[m]^2}\delta_{k,m-r} - \frac{4[2]_{a+1}[k]^2}{[m]^4} \right) \cdot \pic{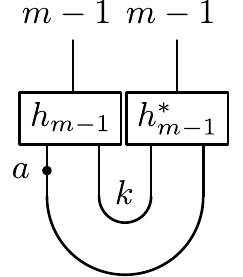} \nonumber \\
  &\hspace*{\parindent} + \frac{[m-k][m-k-1]}{[m]^2} \cdot \pic{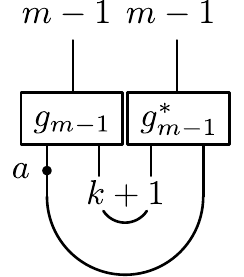} \nonumber \\*
  &\hspace*{\parindent} + \left( -\frac{1}{[m]^2} - \frac{2[k][m-k-1]}{[m]^3} \right) \cdot \left( \pic{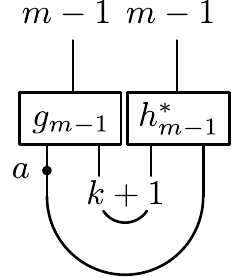} + \pic{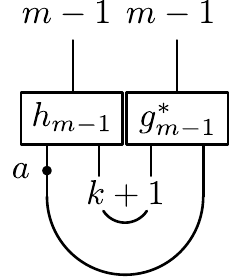} \right) \nonumber \\*
  &\hspace*{\parindent} + \left( -\frac{[2]}{[m]^2}\delta_{k,m-r-1} - \frac{2[2]}{[m]^2}\delta_{k,m-r} + \frac{4[k+1][k]}{[m]^4} \right) \cdot \pic{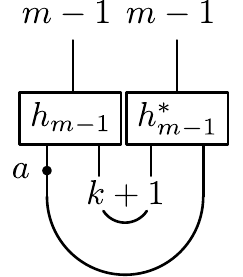} \label{E:open_g_g_m_k_low}
 \end{align}
 and if $m-r+1 \leqs k \leqs m$ we have
 \begin{align}
  &\pic{open_p_g_g_m_k} \nonumber \\*
  &\hspace*{0.5\parindent} = \left( - \frac{[m+1]}{[m]} + \frac{[m-k+1][m-k]}{[m]^2} \right) \cdot \pic{open_g_g_m-1_k-1} \nonumber \\*
  &\hspace*{\parindent} + \left( -\frac{1}{[m]^2} + \frac{\{ m \}'[m-k+1][m-k]}{[m]^3} \right)\cdot \left( \pic{open_g_h_m-1_k-1} \hspace*{-1.25pt} + \hspace*{-1.25pt} \pic{open_h_g_m-1_k-1} \right) \nonumber \\*
  &\hspace*{\parindent} + \left( \frac{[2]}{[m]^2} + \frac{(\{ m \}')^2 [m-k+1][m-k]}{[m]^4} \right) \cdot \pic{open_h_h_m-1_k-1} \nonumber \\
  &\hspace*{\parindent} - \frac{[2]_{a+1}[m-k]^2}{[m]^2} \cdot \pic{open_g_g_m-1_k} \nonumber \\*
  &\hspace*{\parindent} - \frac{[2]_{a+1} \{ m \}' [m-k]^2}{[m]^3} \cdot \left( \pic{open_g_h_m-1_k} + \pic{open_h_g_m-1_k} \right) \nonumber \\*
  &\hspace*{\parindent} - \frac{[2]_{a+1}(\{ m \}')^2[m-k]^2}{[m]^4} \cdot \pic{open_h_h_m-1_k} \nonumber \\
  &\hspace*{\parindent} + \frac{[m-k][m-k-1]}{[m]^2} \cdot \pic{open_g_g_m-1_k+1} \nonumber \\*
  &\hspace*{\parindent} + \frac{\{ m \}'[m-k][m-k-1]}{[m]^3} \cdot \left( \pic{open_g_h_m-1_k+1} + \pic{open_h_g_m-1_k+1} \right) \nonumber \\*
  &\hspace*{\parindent} + \frac{(\{ m \}')^2[m-k][m-k-1]}{[m]^4} \cdot \pic{open_h_h_m-1_k+1} \label{E:open_g_g_m_k_high}
 \end{align}
\end{lemma}

\begin{proof}
 Equation~\eqref{E:open_g_g_m_k_low} and \eqref{E:open_g_g_m_k_high} are proved by induction on $r+2 \leqs m \leqs 2r-2$. In order to do this, let us set
 \begin{align*}
  \pic{open_p_g_g_m_k} &= \alpha_{m,k}^{k-1} \cdot \pic{open_g_g_m-1_k-1} \\*
  &+ \beta_{m,k}^{k-1} \cdot \left( \pic{open_g_h_m-1_k-1} + \pic{open_h_g_m-1_k-1} \right) \\*
  &+ \gamma_{m,k}^{k-1} \cdot \pic{open_h_h_m-1_k-1} \\
  &+ \alpha_{m,k}^k \cdot \pic{open_g_g_m-1_k} \\*
  &+ \beta_{m,k}^k \cdot \left( \pic{open_g_h_m-1_k} + \pic{open_h_g_m-1_k} \right) \\*
  &+ \gamma_{m,k}^k \cdot \pic{open_h_h_m-1_k} \\
  &+ \alpha_{m,k}^{k+1} \cdot \pic{open_g_g_m-1_k+1} \\*
  &+ \beta_{m,k}^{k+1} \cdot \left( \pic{open_g_h_m-1_k+1} + \pic{open_h_g_m-1_k+1} \right) \\*
  &+ \gamma_{m,k}^{k+1} \cdot \pic{open_h_h_m-1_k+1}
 \end{align*}
 Remark that we have 
 \begin{align*}
  \alpha_{m,0}^{-1} = \beta_{m,0}^{-1} = \gamma_{m,0}^{-1} &= 0, \\*
  \alpha_{m,m-1}^m = \beta_{m,m-1}^m = \gamma_{m,m-1}^m &= 0, \\*
  \alpha_{m,m}^m = \beta_{m,m}^m = \gamma_{m,m}^m &= 0, \\*
  \alpha_{m,m}^{m+1} = \beta_{m,m}^{m+1} = \gamma_{m,m}^{m+1} &= 0
 \end{align*}
 for every integer $r+1 \leqs m \leqs 2r-2$. If $m = r+1$ and $k = 0$, let us consider
 \begin{align*}
  &\pic{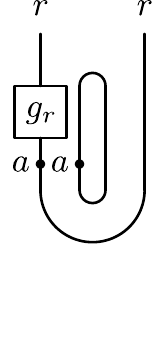} - \pic{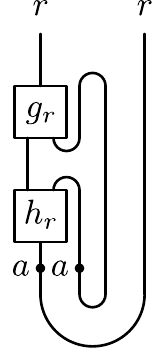} - \pic{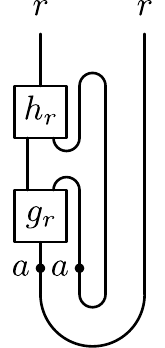} - [2] \cdot \pic{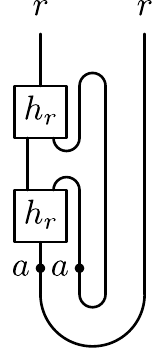}
 \end{align*}
 This gives the condition
 \begin{align*}
  \alpha_{r+1,0}^0 &= -[2]_{a+1}, & \beta_{r+1,0}^0 &= 0, & \gamma_{r+1,0}^0 &= 0, \\*
  \alpha_{r+1,0}^1 &= 0, & \beta_{r+1,0}^1 &= -1, & \gamma_{r+1,0}^1 &= -[2].
 \end{align*}
 If $m = r+1$ and $1 \leqs k \leqs r+1$, let us consider
 \begin{align*}
  &\phantom{+} \pic{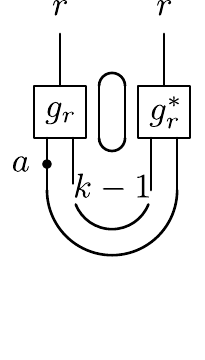} - \pic{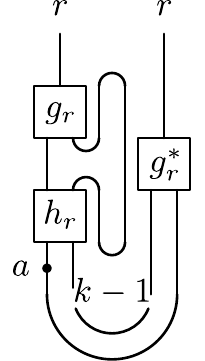} - \pic{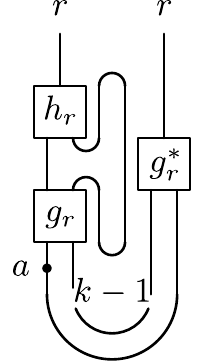} - [2] \cdot \pic{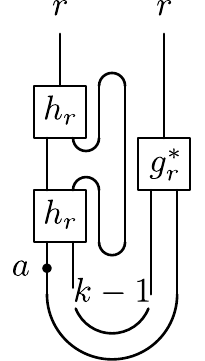} \\*
  &- \pic{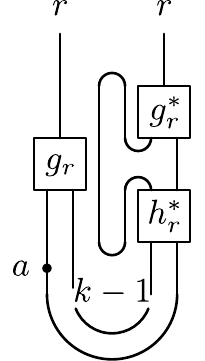} + \pic{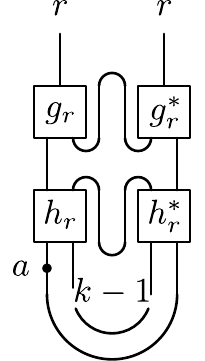} + \pic{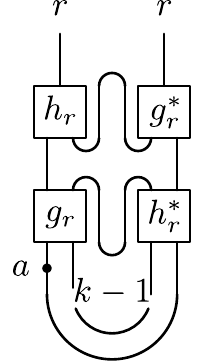} + [2] \cdot \pic{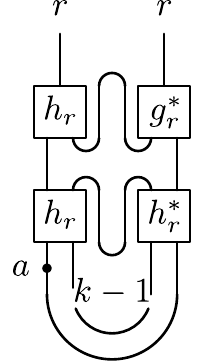} \\*
  &- \pic{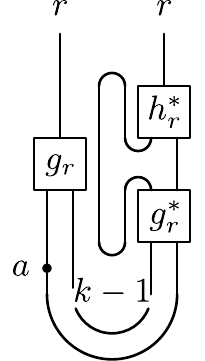} + \pic{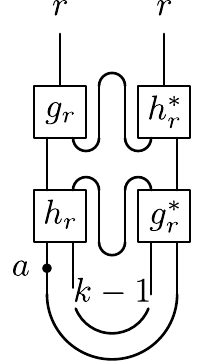} + \pic{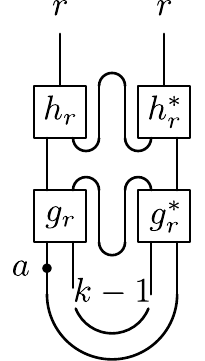} + [2] \cdot \pic{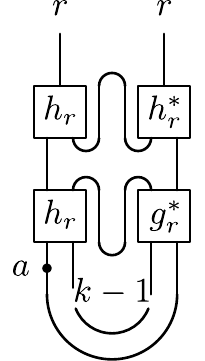} \\*
  &- [2] \cdot \pic{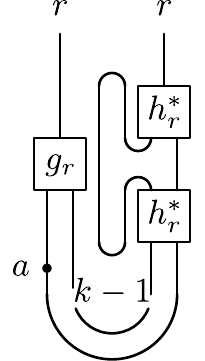} + [2] \cdot \pic{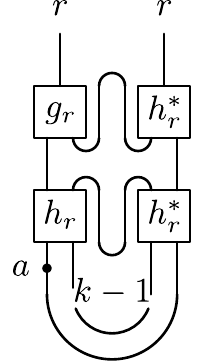} + [2] \cdot \pic{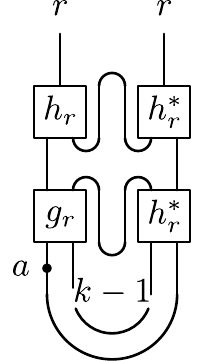} + [2]^2 \cdot \pic{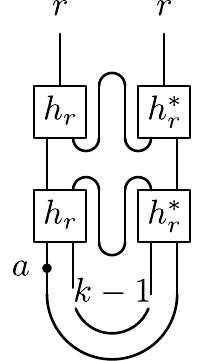}
 \end{align*}
 When $k = 1$, this gives the condition
 \begin{align*}
  \alpha_{r+1,1}^0 &= - \{ 1 \}', & \beta_{r+1,1}^0 &= -2, & \gamma_{r+1,1}^0 &= 0, \\*
  \alpha_{r+1,1}^1 &= 0, & \beta_{r+1,1}^1 &= 0, & \gamma_{r+1,1}^1 &= -[2]_{a+1}, \\*
  \alpha_{r+1,1}^2 &= 0, & \beta_{r+1,1}^2 &= 1, & \gamma_{r+1,1}^2 &= 2[2],
 \end{align*}
 thanks to equations~\eqref{E:open_g_r} and \eqref{E:open_h_r}, and when $2 \leqs k \leqs r+1$, this gives the condition
 \begin{align*}
  \alpha_{r+1,k}^{k-1} &= -[2]+[k-1][k-2], \\*
  \beta_{r+1,k}^{k-1} &= -1 + \{ 1 \}'[k-1][k-2], \\*
  \gamma_{r+1,k}^{k-1} &= [2] + (\{ 1 \}')^2[k-1][k-2], \\*
  \alpha_{r+1,k}^k &= - [2]_{a+1} [k-1]^2, \\*
  \beta_{r+1,k}^k &= - [2]_{a+1}\{ 1 \}'[k-1]^2, \\*
  \gamma_{r+1,k}^k &= - [2]_{a+1}(\{ 1 \}')^2[k-1]^2, \\*
  \alpha_{r+1,k}^{k+1} &= [k][k-1], \\*
  \beta_{r+1,k}^{k+1} &= \{ 1 \}'[k][k-1], \\*
  \gamma_{r+1,k}^{k+1} &= (\{ 1 \}')^2[k][k-1],
 \end{align*}
 thanks to equations~\eqref{E:open_g_g_r_k}, \eqref{E:open_g_h_r_k}, and \eqref{E:open_h_h_r_k}. If $r+2 \leqs m \leqs 2r-2$ and $k = 0$, let us consider
  \begin{align*}
  &\pic{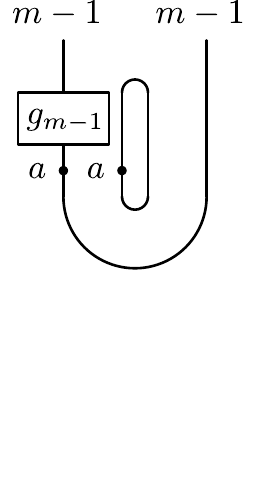} + \frac{[m-1]}{[m]} \cdot \pic{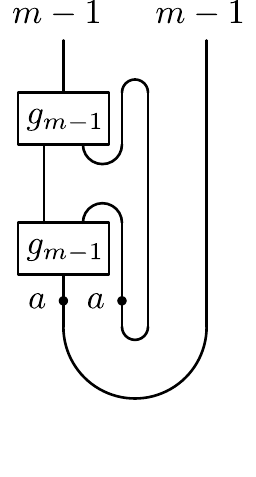} - \frac{2}{[m]^2} \cdot \pic{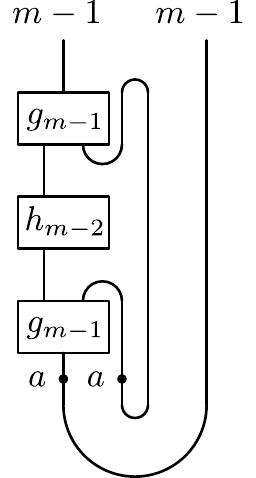}
 \end{align*}
 This gives the condition
 \begin{align*}
  \alpha_{m,0}^0 &= -[2]_{a+1}, & \beta_{m,0}^0 &= 0, & \gamma_{m,0}^0 &= 0, \\*
  \alpha_{m,0}^1 &= \frac{[m-1]}{[m]}, & \beta_{m,0}^1 &= -\frac{1}{[m]^2}, & \gamma_{m,0}^1 &= 0.
 \end{align*}
  If $r+2 \leqs m \leqs 2r-2$ and $1 \leqs k \leqs m$, let us consider
 \begin{align*}
  &\phantom{+} \pic[-20]{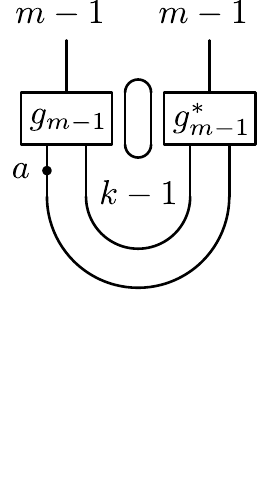} + \frac{[m-1]}{[m]} \cdot \pic[-20]{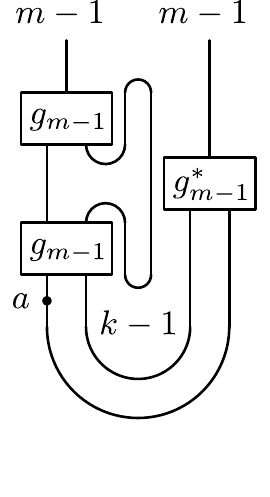} - \frac{2}{[m]^2} \cdot \pic[-20]{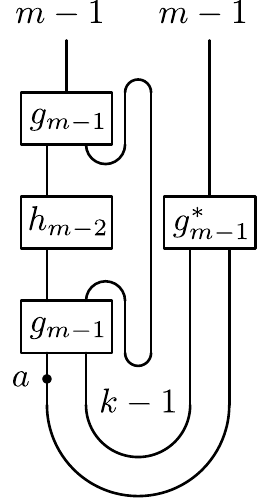} \\*
  &+ \frac{[m-1]}{[m]} \cdot \hspace*{-5pt} \pic[-20]{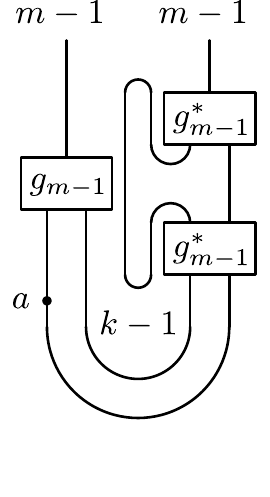} \hspace*{-10pt} + \frac{[m-1]^2}{[m]^2} \cdot \hspace*{-10pt} \pic[-20]{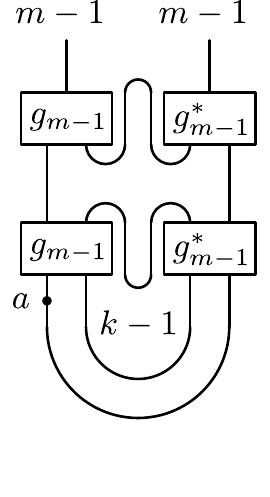} \hspace*{-10pt} - \frac{2[m-1]}{[m]^3} \cdot \hspace*{-10pt}  \pic[-20]{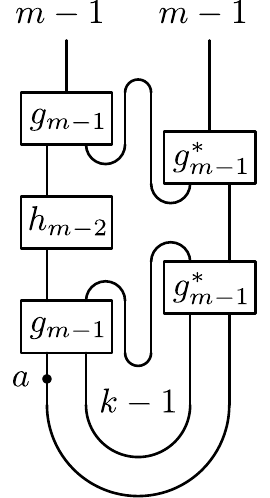} \\*
  &- \frac{2}{[m]^2} \cdot \hspace*{-5pt} \pic[-20]{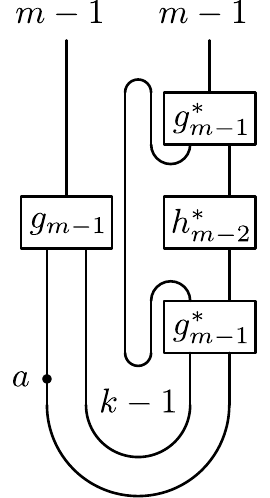} \hspace*{-5pt} - \frac{2[m-1]}{[m]^3} \cdot \pic[-20]{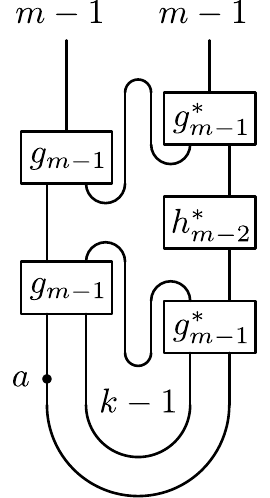} \hspace*{-5pt} + \frac{4}{[m]^4} \cdot \hspace*{-5pt} \pic[-20]{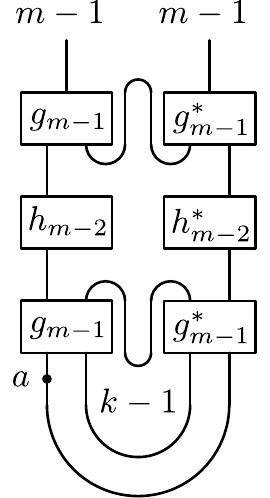}
 \end{align*}
 This gives the condition
 \begin{align*}
  \alpha_{m,k}^{k-1} &= -\frac{\{ m \}'}{[m]} + \frac{[m-1]^2}{[m]^2} \alpha_{m-1,k-1}^{k-2}, \\*
  \beta_{m,k}^{k-1} &= - \frac{2}{[m]^2} + \frac{[m-1]^2}{[m]^2} \beta_{m-1,k-1}^{k-2} - \frac{2[m-1]}{[m]^3} \alpha_{m-1,k-1}^{k-2}, \\*
  \gamma_{m,k}^{k-1} &= \frac{[m-1]^2}{[m]^2} \gamma_{m-1,k-1}^{k-2} - \frac{4[m-1]}{[m]^3} \beta_{m-1,k-1}^{k-2} + \frac{4}{[m]^4} \alpha_{m-1,k-1}^{k-2}, \\*
  \alpha_{m,k}^k &= \frac{[m-1]^2}{[m]^2} \alpha_{m-1,k-1}^{k-1}, \\*
  \beta_{m,k}^k &= \frac{[m-1]^2}{[m]^2} \beta_{m-1,k-1}^{k-1} - \frac{2[m-1]}{[m]^3} \alpha_{m-1,k-1}^{k-1}, \\*
  \gamma_{m,k}^k &= \frac{[m-1]^2}{[m]^2} \gamma_{m-1,k-1}^{k-1} - \frac{4[m-1]}{[m]^3} \beta_{m-1,k-1}^{k-1} + \frac{4}{[m]^4} \alpha_{m-1,k-1}^{k-1}, \\*
  \alpha_{m,k}^{k+1} &= \frac{[m-1]^2}{[m]^2} \alpha_{m-1,k-1}^k, \\*
  \beta_{m,k}^{k+1} &= \frac{[m-1]^2}{[m]^2} \beta_{m-1,k-1}^k - \frac{2[m-1]}{[m]^3} \alpha_{m-1,k-1}^k, \\*
  \gamma_{m,k}^{k+1} &= \frac{[m-1]^2}{[m]^2} \gamma_{m-1,k-1}^k - \frac{4[m-1]}{[m]^3} \beta_{m-1,k-1}^k + \frac{4}{[m]^4} \alpha_{m-1,k-1}^k.
 \end{align*}
 Thanks to equation~\eqref{E:three_terms}, when $1 \leqs k \leqs m-r$ the solution is 
 \begin{align*}
  \alpha_{m,k}^{k-1} &= - \frac{[m+1]}{[m]} + \frac{[m-k+1][m-k]}{[m]^2}, \\*
  \beta_{m,k}^{k-1} &= - \frac{2[k][m-k+1]}{[m]^3}, \\*
  \gamma_{m,k}^{k-1} &= \frac{4[k][k-1]}{[m]^4}, \\*
  \alpha_{m,k}^k &= -\frac{[2]_{a+1}[m-k]^2}{[m]^2}, \\*
  \beta_{m,k}^k &= \frac{2[2]_{a+1}[k][m-k]}{[m]^3}, \\*
  \gamma_{m,k}^k &= \frac{3[2]_{a+1}}{[m]^2}\delta_{k,m-r} - \frac{4[2]_{a+1}[k]^2}{[m]^4}, \\*
  \alpha_{m,k}^{k+1} &= \frac{[m-k][m-k-1]}{[m]^2}, \\*
  \beta_{m,k}^{k+1} &= -\frac{1}{[m]^2} - \frac{2[k][m-k-1]}{[m]^3}, \\*
  \gamma_{m,k}^{k+1} &= - \frac{[2]}{[m]^2} \delta_{k,m-r-1} - \frac{2[2]}{[m]^2} \delta_{k,m-r} + \frac{4[k+1][k]}{[m]^4},
 \end{align*}
 and when $m-r+1 \leqs k \leqs m$ the solution is
 \begin{align*}
  \alpha_{m,k}^{k-1} &= - \frac{[m+1]}{[m]} + \frac{[m-k+1][m-k]}{[m]^2}, \\*
  \beta_{m,k}^{k-1} &= - \frac{1}{[m]^2} + \frac{\{ m \}'[m-k+1][m-k]}{[m]^3}, \\*  
  \gamma_{m,k}^{k-1} &= \frac{[2]}{[m]^2} + \frac{(\{ m \}')^2 [m-k+1][m-k]}{[m]^4}, \\*
  \alpha_{m,k}^k &= -\frac{[2]_{a+1}[m-k]^2}{[m]^2}, \\*
  \beta_{m,k}^k &= -\frac{[2]_{a+1} \{ m \}' [m-k]^2}{[m]^3}, \\*
  \gamma_{m,k}^k &= - \frac{[2]_{a+1}(\{ m \}')^2[m-k]^2}{[m]^4}, \\*
  \alpha_{m,k}^{k+1} &= \frac{[m-k][m-k-1]}{[m]^2}, \\*
  \beta_{m,k}^{k+1} &= \frac{\{ m \}'[m-k][m-k-1]}{[m]^3}, \\*
  \gamma_{m,k}^{k+1} &= \frac{(\{ m \}')^2[m-k][m-k-1]}{[m]^4}. \qedhere
 \end{align*}
\end{proof}

\subsection{Pseudo traces}\label{SS:pseudo-traces}

Finally, let us move on to the most complicated case, and let us start by remarking that, if $1 \leqs k \leqs r-1$, then we have
\begin{align}\label{E:a_0_2r-2_k}
 \pic{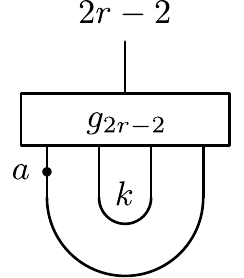} &= - \frac{[r-k]_{a+1}}{[k]} \cdot \pic{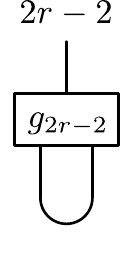}
\end{align}

\begin{lemma}\label{L:int_h_h_m_k}
 For all integers $r \leqs m \leqs 2r-2$ and $0 \leqs k \leqs m$ we have
 \begin{align}
  \pic{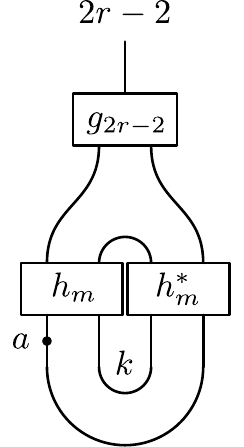} &= \delta_{k,m-r+1} (-1)^{m+1} [r]_{a+1} [m+1] \cdot \pic{a_0_2r-2_r-1} \label{E:int_h_h_m_k}
 \end{align}
\end{lemma}

\begin{proof}
 The computation is easy, and left to the reader.
\end{proof}

\begin{lemma}\label{L:int_g_h_m_k}
 For every integer $r \leqs m \leqs 2r-2$, if $0 \leqs k \leqs m-r$ we have
 \begin{align}
  \pic{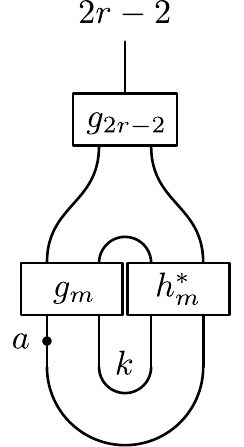} &= \pic{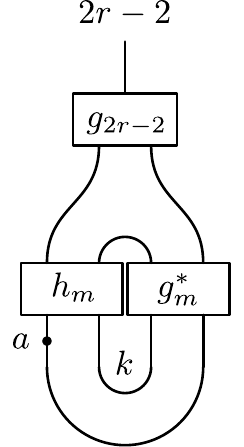} \nonumber \\*
  & = (-1)^m \frac{[2r-m+k-1]_{a+1}[m+1]}{[m-k+1]} \cdot \pic{a_0_2r-2_r-1} \label{E:int_g_h_m_k_low}
 \end{align}
 and if $m-r+1 \leqs k \leqs m$ we have
 \begin{align}
  \pic{p_g_h_m_k} &= \pic{p_h_g_m_k} = \delta_{k,m-r+1} (-1)^{m+1} [m+1] \cdot \pic{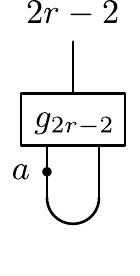} \nonumber \\*
  &+ (1-\delta_{k,m-r+1}) (-1)^{m+1} \frac{[m-k+1]_{a+1}[m+1]}{[m-k+1]} \cdot \pic{a_0_2r-2_r-1} \label{E:int_g_h_m_k_high}
 \end{align}
\end{lemma}

\begin{proof}
 Equations~\eqref{E:int_g_h_m_k_low} and \eqref{E:int_g_h_m_k_high} are proved by induction on $r \leqs m \leqs 2r-2$. In order to do this, let us set
 \begin{align*}
  \pic{p_g_h_m_k} &= \pic{p_h_g_m_k} = \alpha_{m,k} \cdot \pic{a_0_2r-2_0} + \beta_{m,k} \cdot \pic{a_0_2r-2_r-1}
 \end{align*}
 It is also convenient to set 
 \begin{align*}
  \alpha_{m,-1} = \beta_{m,-1} = \alpha_{m,m+1} = \beta_{m,m+1} = \alpha_{m,m+2} = \beta_{m,m+2} &= 0.
 \end{align*}
 If $m = r$ then equations~\eqref{E:open_h_r}, \eqref{E:open_g_h_r_k}, and \eqref{E:a_0_2r-2_k} give the condition
 \begin{align*}
  \alpha_{r,0} &= 0, & \beta_{r,0} &= - [r-1]_{a+1}, & \alpha_{r,k} &= \delta_{k,1}, & \beta_{r,k} &= - (1-\delta_{k,1}) \frac{[r-k+1]_{a+1}}{[k-1]}.
 \end{align*}
 If $r+1 \leqs m \leqs 2r-2$ and $0 \leqs k \leqs m-r$, then equations~\eqref{E:open_g_g_m_k_low}, \eqref{E:a_0_2r-2_k}, and \eqref{E:int_h_h_m_k} give the condition
 \begin{align*}
  \alpha_{m,k} &= \left( - \frac{[m+1]}{[m]} + \frac{[m-k+1][m-k]}{[m]^2} \right) \alpha_{m-1,k-1} \\*
  &\hspace*{\parindent} - \frac{[2]_{a+1}[m-k]^2}{[m]^2} \alpha_{m-1,k} \\*
  &\hspace*{\parindent} + \frac{[m-k][m-k-1]}{[m]^2} \alpha_{m-1,k+1}, \\  
  \beta_{m,k} &= \left( - \frac{[m+1]}{[m]} + \frac{[m-k+1][m-k]}{[m]^2} \right) \beta_{m-1,k-1} \\*
  &\hspace*{\parindent} - \frac{[2]_{a+1}[m-k]^2}{[m]^2} \beta_{m-1,k} \\*
  &\hspace*{\parindent} + \frac{[m-k][m-k-1]}{[m]^2} \beta_{m-1,k+1} \\*
  &\hspace*{\parindent} - \delta_{k,m-r-1} (-1)^m  \frac{[r]_{a+1}}{[m]}.
 \end{align*}
 Thanks to equation~\eqref{E:three_terms}, the solution is
 \begin{align*}
  \alpha_{m,k} &= 0, \\*
  \beta_{m,k} &= (-1)^m \frac{[2r-m+k-1]_{a+1}[m+1]}{[m-k+1]}.
 \end{align*}
 If $r+1 \leqs m \leqs 2r-2$ and $m-r+1 \leqs k \leqs m$, then equations~\eqref{E:open_g_g_m_k_high}, \eqref{E:a_0_2r-2_k}, and \eqref{E:int_h_h_m_k} give the condition
 \begin{align*}
  \alpha_{m,k} &= \left( - \frac{[m+1]}{[m]} + \frac{[m-k+1][m-k]}{[m]^2} \right) \alpha_{m-1,k-1} \\*
  &\hspace*{\parindent} - \frac{[2]_{a+1}[m-k]^2}{[m]^2} \alpha_{m-1,k} \\*
  &\hspace*{\parindent} + \frac{[m-k][m-k-1]}{[m]^2} \alpha_{m-1,k+1}, \\  
  \beta_{m,k} &= \left( - \frac{[m+1]}{[m]} + \frac{[m-k+1][m-k]}{[m]^2} \right) \beta_{m-1,k-1} \\*
  &\hspace*{\parindent} - \frac{[2]_{a+1}[m-k]^2}{[m]^2} \beta_{m-1,k} \\*
  &\hspace*{\parindent} + \frac{[m-k][m-k-1]}{[m]^2} \beta_{m-1,k+1} \\*
  &\hspace*{\parindent} - \delta_{k,m-r+1} (-1)^m  \frac{[r]_{a+1}}{[m]}.
 \end{align*}
 Thanks to equation~\eqref{E:three_terms}, the solution is
 \begin{align*}
  \alpha_{m,k} &= \delta_{k,m-r+1} (-1)^{m+1} [m+1], \\*
  \beta_{m,k} &= (1-\delta_{k,m-r+1}) (-1)^{m+1} \frac{[m-k+1]_{a+1}[m+1]}{[m-k+1]}. \qedhere
 \end{align*}
\end{proof}

\begin{lemma}
 For every integer $0 \leqs h \leqs r$ we have
 \begin{align}\label{E:aux_1}
  2[r-h]_{a+1} = 2[r]_{a+1} - 2[h]_{a+1}.
 \end{align}
\end{lemma}

\begin{proof}
 The computation is easy, and left to the reader.
\end{proof}

\begin{lemma}\label{L:int_g_m}
 For all integers $r \leqs m \leqs 2r-2$ and $0 \leqs k \leqs m-r$ we have
 \begin{align}
  &\pic{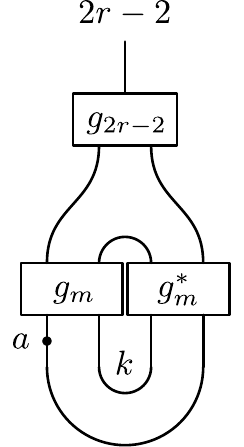} = (-1)^m \frac{\{ m-k-r+1 \}'_{a+1}[m+1]}{[m-k+1]} \cdot \pic{a_0_2r-2_0} \nonumber \\*
  &- (-1)^m \left( \frac{2[m-k-r+1]_{a+1}[k]}{[m-k+1]^2} + \frac{[r]_{a+1}[m+1]}{[m-k+1]} \sum_{h=1}^{m-k-r} \frac{\{ h \}'}{[h]} \right) \cdot \pic{a_0_2r-2_r-1} \label{E:int_g_g_m_k_low}
 \end{align}
\end{lemma}

\begin{proof}
 Equation~\eqref{E:int_g_g_m_k_low} is proved by induction on $r \leqs m-k \leqs 2r-2$. In order to do this, let us set
 \begin{align*}
  \pic{p_g_g_m_k} &= \alpha_{m,k} \cdot \pic{a_0_2r-2_0} + \beta_{m,k} \cdot \pic{a_0_2r-2_r-1}
 \end{align*}
 It is also convenient to set 
 \begin{align*}
  \alpha_{m,-1} = \beta_{m,-1} = \alpha_{m,m+1} = \beta_{m,m+1} = \alpha_{m,m+2} = \beta_{m,m+2} &= 0.
 \end{align*}
 If $m-k=r$ and $m=r$, then equations~\eqref{E:open_g_r} and \eqref{E:a_0_2r-2_k} give the condition
 \begin{align*}
  \alpha_{r,0} &= -[2]_{a+1}, & \beta_{r,0} &= 0.
 \end{align*}
 If $r \leqs m-k \leqs 2r-2$ and $r+1 \leqs m \leqs 2r-2$, then equations~\eqref{E:open_g_g_m_k_low} and \eqref{E:a_0_2r-2_k} give the condition
 \begin{align*}
  \alpha_{m,k} &= \left( - \frac{[m+1]}{[m]} + \frac{[m-k+1][m-k]}{[m]^2} \right) \alpha_{m-1,k-1} \\*
  &\hspace*{\parindent} - \frac{[2]_{a+1}[m-k]^2}{[m]^2} \alpha_{m-1,k} \\*
  &\hspace*{\parindent} + \frac{[m-k][m-k-1]}{[m]^2} \alpha_{m-1,k+1} \\*
  &\hspace*{\parindent} - \delta_{k,m-r-1} (-1)^m \frac{2}{[m]}, \\
  \beta_{m,k} &= \left( - \frac{[m+1]}{[m]} + \frac{[m-k+1][m-k]}{[m]^2} \right) \beta_{m-1,k-1} \\*
  &\hspace*{\parindent} - \frac{[2]_{a+1}[m-k]^2}{[m]^2} \beta_{m-1,k} \\*
  &\hspace*{\parindent} + \frac{[m-k][m-k-1]}{[m]^2} \beta_{m-1,k+1} \\*
  &\hspace*{\parindent} + (-1)^m \frac{2[2r-m+k+1]_{a+1}}{[m][m-k-1]}.
 \end{align*}  
 Thanks to equations~\eqref{E:three_terms} and \eqref{E:aux_1} the solution is
 \begin{align*}
  \alpha_{m,k} &= (-1)^m \frac{\{ m-k-r+1 \}'_{a+1}[m+1]}{[m-k+1]}, \\*
  \beta_{m,k} &= (-1)^{m+1} \left( \frac{2[m-k-r+1]_{a+1}[k]}{[m-k+1]^2} + \frac{[r]_{a+1}[m+1]}{[m-k+1]} \sum_{h=1}^{m-k-r} \frac{\{ h \}'}{[h]} \right).
 \end{align*}
\end{proof}

\end{document}